\documentclass[preprint]{imsart}
\RequirePackage[OT1]{fontenc}
\usepackage{mypackage}
\usepackage{mymacros2}
\usepackage{mymacros_3}
\usepackage{mymacros}

\numberwithin{equation}{section}

\newtheorem*{cond*}{\assumptionnumber}
\providecommand{\assumptionnumber}{}
\makeatletter
\newenvironment{cond}[1]
 {%
  \renewcommand{\assumptionnumber}{Condition #1}%
  \begin{cond*}%
  \protected@edef\@currentlabel{#1}%
 }
 {%
  \end{cond*}
 }
\makeatother

\theoremstyle{plain}
\theoremstyle{definition}
\theoremstyle{remark}

\theoremstyle{remark}\newtheorem{remark}{Remark}
\theoremstyle{definition}\newtheorem{example}{Example}
\theoremstyle{plain}
\theoremstyle{plain}\newtheorem{theorem}{Theorem}
\theoremstyle{plain}\newtheorem{lemma}{Lemma}
\theoremstyle{plain}\newtheorem{proposition}{Proposition}
\theoremstyle{plain}\newtheorem{corollary}{Corollary}
\theoremstyle{plain}

\begin{document}

\begin{frontmatter}
\title{Location estimation for symmetric  log-concave densities}
\runtitle{Location estimation  }

\begin{aug}
\author{\fnms{Nilanjana} \snm{Laha}\thanksref{m1}\ead[label=e1]{nlaha@hsph.harvard.edu}}

\runauthor{N. Laha}

\affiliation{Department of Biostatistics, Harvard University}
\address{Department of Biostatistics, Harvard University, 677 Huntington Ave, Boston, MA 02115, U.S.A\thanksmark{m1}\\
\printead{e1}}
\end{aug}	

\begin{abstract}
 We revisit the problem of  estimating  the center of symmetry $\theta$ of an unknown symmetric density $f$.  Although \cite{stone}, \cite{eden}, and \cite{saks}  constructed adaptive estimators of $\th$ in this model, their estimators depend on tuning parameters. In an effort to circumvent the dependence on tuning parameters,  we impose  an additional assumption of log-concavity on $f$. We show that in this shape-restricted model, the maximum likelihood estimator (MLE) of $\theta$ exists.  We also study some  truncated one-step estimators and show that they are $\sqrt{n}-$consistent, and nearly achieve the asymptotic efficiency bound. We also show that the rate of convergence for the MLE is $O_p(n^{-2/5})$. Furthermore, we show that our estimators are robust with respect to the violation of the log-concavity assumption. In fact, we show that the one step estimators  are still $\sqn$-consistent under some mild  conditions.  These analytical conclusions are supported by simulation studies. 
\end{abstract}


\begin{keyword}
\kwd{log-concave}
\kwd{shape constraint}
\kwd{symmetric location model}
\kwd{one step estimator}
\end{keyword}

\date{\today}
\end{frontmatter}


\section{Introduction}
\label{sec:intro}

In this paper, we  revisit the symmetric location model with an additional shape-restriction of log-concavity. 
We  let $\mathcal{P}$  denote the class of  all densities  on the real line $\RR$. 
For any $\th\in\RR$, denote by $\mathcal{S}_{\th}$ the class of all  densities that are symmetric about $\theta$.
We take
\begin{equation}\label{def: class of concave functions}
\mathcal{C}:=\bigg\{\ph:\RR\mapsto[-\infty,\infty)\ \bl\ \ph \text{ is concave, closed, and proper} \bigg\},
\end{equation}
and define $\mathcal{SC}_{\th}=\mathcal{S}_{\th}\cap\mathcal{C}$. Here  a proper and closed concave function is as defined in \cite{rockafellar}, page 24 and 50.
Letting $\mathcal{LC}$ denote the class of log-concave densities
\[\mathcal{LC}:= \bigg\{f\in\mathcal{P}\ \bl\ \phi=\log f\in\mathcal{C}\bigg\},\]
we set $\mathcal{SLC}_{\th}= \mathcal{LC}\cap\mathcal{S}_{\th}$.
We aim to estimate $\theta$ in the log-concave symmetric location model
\begin{equation}
\label{definition: model}
\mP_0=\bigg\{f\in\mP \ \bl\ f(x;\th)=g(x-\th),\ \th\in\RR,\ g\in\mathcal{SLC}_0,\ \If<\infty \bigg\},
\end{equation}
where $\If$ is the Fisher information for location. 
It is well-established that \citep[][Theorem~3]{huber} $\If$ is finite if and only if $f$ is an absolutely continuous density satisfying
\[\edint\lb\dfrac{f'(x)}{f(x)}\rb^2 f(x)dx<\infty,\]
where $f'$ is an $L_1$-derivative of $f$. Also, in this case, $\If$ takes the form 
\begin{equation*}
\If=\edint\lb\dfrac{f'(x)}{f(x)}\rb^2 f(x)dx.
\end{equation*} 

    Estimation of $\th$ in the full symmetric location model
\begin{equation}\label{definition: the symmetric model Ps}
\mathcal{P}_s =\bigg\{f\in\mathcal{P}\ \bl\ f(x;\th)=g(x-\th),\ \th\in\RR,\ g\in\mathcal{S}_0,\ \If<\infty\bigg\}
\end{equation}
is an old semi-parametric problem, dating back to \cite{stein}. From then on, this problem has been considered by many early authors including, but not limited to, \cite{stone}, \cite{beran}, \cite{saks}, and \cite{eden}. 
   There are two main reasons behind the assumption of symmetry in this model. First, as \cite{stone} has noted, if $f$ is totally unrestricted, $\th$ is not identifiable.  Second,
 the definition of location becomes less clear in the absence of symmetry  \citep{unimodal1}. 
 
 The appeal of the above  model lies in the fact that
 adaptive estimation of $\th$ is possible in this model \citep{stone}.
In other words,  there exist consistent estimators of $\th$ in $\mP_s$, whose variances attain the parametric lower bound, which is $\If^{-1}$ in this case.  
See Sections $3.2$, $3.3$, and $6.3$ of \cite{jonsemi} for more discussion on adaptive estimation in $\mathcal{P}_s$.  

Now we will briefly discuss some existing methods of estimating  $\theta$ in $\mathcal{P}_s$. 
  \cite{stone} considers a one-step estimator. One-step estimators generally depend on the estimation of the score  $-f'/f$. This paper  estimates the scores using symmetrized Gaussian kernels based on truncated data, thereby incorporating  two tuning parameters, one for the truncation,  and another for the scaling of the Gaussian Kernel. Although \cite{stone} requires the tuning parameters  to satisfy some  conditions \citep[Theorem $5.3$]{stone}, he does not provide any  prescription  for choosing them. \cite{saks} considers a subclass of $\mP_s$ with some restrictions on the score function that are weaker than the assumption of log-concavity, and includes heavy-tailed distributions like Cauchy. His estimator is motivated by  a linear function of the order statistics, whose coefficients depend on the score functions. The estimation procedure involves three tuning parameters. \cite{saks} gives some examples of these tuning parameters,  but does not provide any data-dependent method for choosing them. \cite{beran} uses linearized rank estimate for adaptive estimation of $\th$ in $\mP_s$. But he makes the additional assumption that the score is twice continuously differentiable, which excludes densities like Laplace. His method is based on the Fourier series expansion of the score function and requires a tuning parameter for choosing the degree of the associated complex trigonometric polynomial. \cite{beran} does not provide any data-dependent procedure to choose this tuning parameter.

   Now we provide some justification for imposing the shape restriction of log-concavity.  The class of log-concave densities, $\mathcal{LC}$, belongs to the larger class of unimodal densities.  Unimodality is a reasonable assumption in our context,  because, as \cite{unimodal1} points out, in practice, multimodal densities generally result from unimodal mixtures. Separate procedures are available for the latter class.   The problem with the class  of unimodal densities, however, is that it  does not admit an MLE \citep{birge1997}. Many common symmetric unimodal continuous densities, notwithstanding,  are in $\mathcal{LC}$, which, being a smaller and richer class,  admit  MLE \citep[cf.][]{ exist, 2009rufi,dosssymmetric, xuhigh}. Since the MLE  does not depend on tuning parameters, it is possible to estimate $f$  without relying on any tuning parameter, which can lead to tuning parameter free estimation of $\theta$.  This is the main reason to opt for the shape-restriction of log-concavity.

 Although shape-constrained estimation has a rich history,  so far there has been very little to no use of shape-constraints in connection with semiparametric 
models.  
In fact, to  the best of our knowledge, \cite{eden} is the only one to incorporate shape-constraints in treating  the one-sample symmetric location problem considered here. Actually \cite{eden} uses log-concavity and considers estimation in $\mP_0$ although her paper  does not mention log-concavity.  She restricts $\mP_s$ by assuming that $f'(F^{-1}(u))/f(F^{-1}(u))$ is non-increasing in $u$,  which is equivalent to $f$ being log-concave. \cite{eden} considers data-partitioning to estimate scores from a small fraction of the data and constructs the Hodges-Lehmann rank estimate of location \citep{hodges} from the remaining data.  
 The partitioning of the data and the construction of the score function  involve some tuning parameters.
\cite{eden} does not explicitly describe how to choose  these tuning parameters. 
 We also want to remark about  \cite{sharmada}, who consider both location  and scale estimation in an  elliptical symmetry model, which albeit bearing some resemblance,  is  different from $\mathcal{P}_0$. Also, \cite{sharmada}'s estimation procedure is completely different from ours.  
 
To summarize,  our goal is  to bridge the gap between  the symmetric location model and log-concavity. In doing so, we also try to exploit the structure in the class $\mathcal{LC}$ to  get rid of the dependence on external tuning parameters.   Here is a brief account of our main contributions.
 \subsubsection*{The MLE} We show that the MLE of $\theta$ and $g$ exist in  $\mP_0$. It also follows that \hthm, i.e. the MLE of $\theta$, has the rate of convergence $O_p(n^{-2/5})$. In addition, we show that the  Hellinger distance between $f\in\mP_0$ and its MLE $\hgf$ is of order $O_p(n^{-2/5})$, where the Hellinger distance $H(f,g)$ between two densities $f$ and $g$ is defined by
 \[H^2(f,g)=2^{-1}\int(\sqrt{f(x)}-\sqrt{g(x)})^2dx.\]
  The above-mentioned rate is analogous to the rate obtained by \cite{dosssymmetric} for the MLE  of $f$ in $\mathcal{SLC}_0$, or  by \cite{dossglobal}   for the MLE  of $f$ in $\mathcal{LC}$.
  \subsubsection*{The One step estimator} We also propose four one step estimators for $\theta$, which  generally use on the mean, median or a trimmed mean as the  preliminary estimator.  We estimate the score using either the MLE of $f$ in $\mathcal{LC}$, or in $\mathcal{SLC}_{\bth}$, which can be efficiently computed using the R packages "logcondens" and "logcondens.mode", respectively. We establish that  a truncated version of the one-step estimators is $\sqrt{n}-$consistent, where its variance also  nearly achieves the asymptotic efficiency bound $\If^{-1}$. The truncated estimator \hth, however, use a tuning parameter.
 
%
   

 \subsubsection*{Model misspecification} 
We show that our estimators are robust to the violation of the log-concavity assumption.  
   In particular, we show that, even if $f\notin\mP_0$, as long as $f$ is symmetric about  $\th$,  the MLE and the truncated one step estimators are still strongly consistent.   Additionally, for the specific case of the truncated one-step estimators, $\sqn$-consistency is also preserved. The MLE, on the other hand, may still possess a limit  even when $f$ is not symmetric, or if the data generating distribution does not even possess a density. For the analysis in the misspecified case, we rely on the log-concave projection theory developed by \cite{dumbreg}, \cite{theory}, and \cite{xuhigh}. 
   The early authors, \cite{stone}, \cite{beran}, \cite{saks}, and \cite{eden} did not discuss  the behavior of their estimators  under model miss-specification although Stone (1975) did discuss the behavior of his estimator when $\I = \infty$.

   The article is organized as follows. In Section \ref{sec: one-step}, we introduce the one step estimators, and in \ref{sec: MLE}, we discuss  the existence of the MLE and its characterizations.  We study the asymptotic properties of our estimators in Section \ref{sec: asymptotic properties}. We provide a simulation study in Section \ref{sec: simulation}. The subtleties of our estimators are deferred to the Discussion Section. We collect  the proofs in  Section \ref{sec:proofs}.

\subsection{Preliminaries/ Notation and terminology}
\label{sec: Preliminaries}

We assume that $X_1,\ldots, X_n$ are independent and identically distributed (i.i.d.) random variables with density $f_0\in\mathcal{P}_0$. Therefore,  \eqref{definition: model} indicates that $f_0=g_0(\mathord{\cdot}-\th_0)$, where $g_0\in\mathcal{SLC}_{0}$, the class of all log-concave densities symmetric about $0$. We set  $\ph_0=\log f_0$, and $\ps_0=\log g_0$. 
 We let $F_0$ and $G_0$ denote the respective distribution functions of $f_0$ and $g_0$, and denote by $P_0$ the measure corresponding to $F_0$. We  denote the empirical distribution function of the $X_i$'s by $\Fm$, and write $\Pm$ for the corresponding empirical measure.  The (classical) empirical process of the $X_i$'s will be denoted by $\Zn=\sqn(\Fm-F_0)$. 
 For any integrable function $h:\RR\mapsto\RR$, we write
\[\Pm h:=n^{-1}\si {h(X_i)}\quad{\text{and} }\quad P_0 h=\edint h(x)f_0(x)dx.\] 
As usual, we denote the order statistics of a sample  $(Y_1,\ldots,Y_n)$  by
 \[Y_{()}=(Y_{(1)},\ldots,Y_{(n)}).\] 
  
  In context of a concave function $\ps:\RR\mapsto\RR$, the domain  $\dom(\ps)$ will be defined as in \cite{rockafellar} (page $40$), that is, 
$\dom(\ps)=\{x\in\RR|\ps(x)>-\infty\}$. 
  For any concave function $\ps$, we say $x\in\RR$ is a knot of $\ps$, if either $\ps'(x+)\neq\ps'(x-)$, or  $x$ is at the boundary of $\dom(\ps)$. We denote by $\mathcal{K}(\ps)$ the set of the knots of $\ps$. Unless otherwise mentioned, for a real valued function $h$, provided they exist,  $h'$ and $h'(\mathord{\cdot}-)$ will refer to the right and left derivatives of $h$ respectively. For a distribution function $F$, we let $J(F)$ denote the set $\{0<F<1\}$.  For any non-negative real-valued function $h$, we denote the support by $\supp(h)=\{x\in\RR|h(x)>0\}$. 
  As usual, we denote the set of natural numbers by $\mathbb{N}$.

Next, we  develop some notation for  maximum likelihood estimation.
%
  For $\ps\in\mathcal{SC}_0$ and $\th\in\RR$, following \cite{dumbreg} and \cite{xuhigh}, we define the criterion function for maximum likelihood estimation  by
 \begin{equation}\label{criterion function: xu samworth}
 \Psi(\th,\ps,F)=\edint \ps(x-\th)dF(x)-\edint e^{\ps(x-\th)}dx.
 \end{equation}
Following \cite{silverman1982}, we  included a Lagrange term  to get rid of the normalizing constant  involved in density estimation. This is a common device in log-concave density estimation literature \citep[cf.][]{2009rufi, dosssymmetric}.
 We use the notation $\Psi_n(\th,\ps)$ to denote the sample version  $\Psi(\th,\ps,\Fm)$ of $\Psi(\th,\ps,F)$. Thus,
  \begin{equation}\label{criterion function: Doss}
  \Psi_n(\th,\ps)=\edint \ps(x-\th)d\Fm(x)-\edint e^{\ps(x-\th)}dx.
  \end{equation}
We denote the maximized criterion function   by 
\begin{equation}\label{maximum of criterion function}
L(F)=\sup_{\th\in\RR,\ps\in\mathcal{SC}_0}\Psi(\th,\ps,F).
\end{equation}


\section{Estimators}
\label{sec: estimators}


\subsection{One-step estimators}
\label{sec: one-step}
This section focuses on  constructing the one-step estimators of $\theta$. To this end, our first task is to  choose suitable preliminary estimator \bth\ of $\th$. We will impose two requirements on $\bth$: $\sqn$-consistency and strong consistency. The sample mean and the sample median also satisfy our requirements on $\bth$ when $f_0\in\mP_0$. For a general $f_0$, the $Z$-estimator of the shift in the logistic location shift model satisfies these  requirements under minimal regularity conditions \citep[cf. Example 5.40 and Theorem $5.23$,][]{vdv}. 

Suppose $\hn$ is an estimator of the centered density $g_0=f_0(\cdot+\th_0)$. We let $\tilde{\ps}_n$ denote the respective log-density $\log \hn$,
 and write $\hH$ for the distribution function corresponding to \hn. 
%
Ideally, to construct a one-step estimator in this setting, one would first estimate the Fisher information by
\begin{equation}\label{def: untruncated fisher info}
\hin=\edint \tilde{\ps}'_n(x-\bth)^2d\Fm(x),
\end{equation}
and set
\begin{equation}\label{def: one-step estimator full}
\hthf=\bth-\edint\dfrac{\tilde{\ps}'_n(x-\bth)}{\hin}d\Fm(x).
\end{equation}

It is not hard to see that the  choice of $\hn$ will heavily influence the performance of $\hthf$. We want $\hn$ to be symmetric about $0$, and if possible, log-concave.

A simple way of estimating $f_0=g_0(\mathord{\cdot}-\th_0)$ is to compute its MLE in $\mathcal{LC}$, the class of all log-concave densities. However, this log-concave  MLE, which we will denote by $\hf$ from now on, may violate  the requirement of symmetry. In what follows,  we show that, this naive estimator $\hf$  can be improved to yield better estimators of $g_0$.  
 
\subsubsection*{Symmetrized  estimator $\hn^{sym}$}
The simplest way of obtaining a symmetrized estimator of $g_0$ from $\hf$ is by symmetrizing the latter about $\bth$, which yields
\begin{equation}\label{est 1}
\hn^{sym}(z)=\dfrac{1}{2}\lb\hf(\bth+z)+\hf(\bth-z)\rb,\quad z\in\RR.
\end{equation}
Although $\hn^{sym}$ is symmetric, it is not, in general, log-concave. Additionally, the performance of $\hn^{sym}$ can suffer from the lack of smoothness of $\hf$. Our simulations show that the one-step estimator based on $\hn^{sym}$ has considerably less efficiency, especially for smaller sample sizes. Therefore, we seek to improve the performance of $\hn^{sym}$  by incorporating smoothing, which leads to our next estimator of $g_0.$

\subsubsection*{Smoothed symmetrized  estimator $(\hn^{sym})^{sm}$}
  To perform smoothing, our first task is to choose a data-dependent smoothing parameter $\widehat{b}_n$. Towards this end, following \cite{smoothed}, we take 
  \begin{equation}\label{definition bn}
 \widehat{b}_n^2:=\hs-\ts,
 \end{equation}
 where   $\hs$ is the sample variance and $\ts$ is the variance corresponding to the density $\hf$, i.e.
  \[\hs=\dfrac{1}{n-1}\si (X_i-\bX)^2,\]
  and
\[\ts=\edint z^2\hf(z)dz-\lb\edint z\hf(z)dz\rb^2.\]
that the right hand side of \eqref{definition bn} is positive follows from (2.1) of \cite{smoothed}. 
We define the smoothed symmetrized  estimator by
\begin{equation}\label{25eq3}
(\hn^{sym})^{sm}(z):=\dfrac{1}{\widehat{b}_n}\edint\hn^{sym}(z-t)\gd(t/\widehat{b}_n) dt,\quad z\in\RR,
\end{equation}
where $\gd$ is the standard normal density.
The estimator \htsm\ can also be represented in terms  of the  
  smoothed version of $\hf$, which we define by
 \begin{equation}\label{def: smoothed logconcave MLE}
 \hf^{sm}(z)=\dfrac{1}{\widehat{b}_n}\edint \hf(z-t)\gd(t/\widehat{b}_n)dt,\quad z\in\RR.
 \end{equation}
To see this, note that
   \begin{align}\label{representation of htsm}
\htsm(z)=\dfrac{\hts(\bth+z)+\hts(\bth-z)}{2}.
   \end{align}
 Though $(\hn^{sym})^{sm}$ leads to a more efficient one-step estimator of $\th_0$ in small samples,  it is not necessarily log-concave, which leaves room for further improvement.
\subsubsection*{Partial MLE estimator $\widehat{g}_{\bth}$}
The following estimator of $g_0$ is obtained by maximizing the criterion function $\Psi_n(\bth,\ps)$ in \eqref{criterion function: Doss}  over $\ps\in\mathcal{SC}_0$. From Theorem 2.1(C) of \cite{dosssymmetric}, it follows that the maximizer $\widehat{\ps}_{\bth}$ satisfies 
\[\edint e^{\widehat{\ps}_{\bth}(x)}dx=1,\]
 which leads to the density estimator
 \begin{equation}\label{def: partial MLE estimator}
  \widehat{g}_{\bth}=e^{\widehat{\ps}_{\bth}}.
  \end{equation} 
 We call this estimator a {\sl Partial MLE estimator} to distinguish it from the traditional MLE of $g_0$, which we will discuss in Section \ref{sec: MLE}.
Notice that $\widehat{g}_{\bth}$ satisfies both requirements of symmetry and log-concavity. 

  From \cite{2009rufi} and \cite{dosssymmetric}, it follows that  the three symmetric density estimators $\hn$ we have discussed so far share the same support $[-d,d]$, where
\[d=\max(X_{(n)}-\bth,\bth-X_{(1)}).\]

\subsubsection*{Geometric mean type symmetrized estimator $\hn^{geo,sym}$}

 Similar to $\hn^{sym}$, this estimator is also a symmetrized version of $\hf$. However, this time the symmetrization  is based on the geometric mean of $\hf(\bth+\mathord{\cdot})$ and $\hf(\bth-\mathord{\cdot})$. For $z\in\RR$, we define this estimator by
 \begin{equation}\label{est 5}
 \hn^{geo,sym}(z):=C_{n}^{geo}\lb \hf(\bth+z)\hf(\bth-z)\rb^{1/2},
 \end{equation}
  where $C_{n}^{geo}$ is a random normalizing constant. Since
 the sum of two concave functions is again concave, it is not hard to see that $ \hn^{geo,sym}$ is log-concave. However, its support
  \[\supp(\hn^{geo,sym})=[-\min(X_{(n)}-\bth,\bth-X_{(1)}),\min(X_{(n)}-\bth,\bth-X_{(1)})]\] 
   is smaller than that of the previous estimators of $g_0$.

\subsubsection*{Constructing the truncated one-step estimator}

Ideally, one's first choice for a one step estimator will be the estimator \hthf\ suggested by \eqref{def: one-step estimator full}, which we will refer to as the untruncated one-step estimator from now on. However, when  $\phi'_0$ is unbounded near the boundary of its domain, $\hln'$ is also unbounded on its domain (see Corollary \ref{lemma: divergence} in Section 3.1). In this case, $\hln$'s behavior can be hard to control in the tails. To avoid these difficulties, we  consider a truncated one-step estimator. \cite{piet}
 used a similar idea in a current status linear regression model where  a  truncated log-likelihood was considered to avoid problems in the tails of the unknown error distribution.

 We take the truncation parameter $\eta$ to be a small fixed positive number  in the interval $(0,1/2)$, and denote by $\xin$ the $(1-\eta)$-th quantile of $\hH$, the distribution function of $\hn$. 
 The symmetry of $\hn$ about $0$ implies that $\hH^{-1}(\eta)=-\xin$.  
 
Two potential estimators of the Fisher information are
  \begin{align*}
&\hin^{*}(\eta)=\dint _{\bth-\xin}^{\bth+\xin}\tilde{\ps}'_n(x-\bth)^2d\hH(x-\bth),
\end{align*}
and
  \begin{align}\label{27indef}
&\hin(\eta)=\dint _{\bth-\xin}^{\bth+\xin}\tilde{\ps}'_n(x-\bth)^2d\Fm(x).
\end{align}
Though it can be shown that the above estimators differ by a $o_p(1)$ term, our simulations indicate that the estimator $\hin(\eta)$ yields a more efficient one-step estimator. Therefore, we take $\hin(\eta)$ to be our estimator of the Fisher information.

 In the same spirit as the untruncated one-step estimator \hthf\ in \eqref{def: one-step estimator full},  we now  construct a truncated one-step estimator of $\th_0$ as follows:
\begin{equation}\label{def: one-step estimator: truncated}
\hth=\bth-\dint_{\bth-\xin}^{\bth+\xin}\dfrac{\hln'(x-\bth)}{\hin(\eta)}d\Fm(x).
\end{equation} 
Note that we use all the data to estimate $\bth$, $\tilde{\ps}'_n$, and $\hn$. The truncation comes only in the step where we set up the one step estimator. Next, observe that  $\hin(\eta)$ is always smaller than $\hin$,  the estimator of the untruncated Fisher information (see \eqref{def: untruncated fisher info}). In Section \ref{sec: asymptotics: one step estimator}, we show that the asymptotic variance of $\hth$ has an inverse relation with the Fisher information, indicating that  the untruncated estimator in \eqref{def: one-step estimator full} will be asymptotically more efficient. However, our simulations in Section \ref{sec: simulation} show that for sufficiently small $\eta$, this loss in efficiency is negligible for most densities.  It is also evident that $\hin(\eta)$ increases in $\eta$, so  smaller values of  $\eta$ will lead to higher asymptotic efficiency. We will discuss about the relation between $\eta$ and this asymptotic efficiency  in more detail in Section \ref{sec: asymptotics: one step estimator}.


\subsection{Maximum likelihood estimator (MLE)}
\label{sec: MLE}
This section aims to establish the existence and the basic properties of the MLE of $(\th_0,g_0)$. We already mentioned in Section \ref{sec:intro}  that the MLE of $f_0$ in  $\mathcal{SLC}_{\th}$ exists for a fixed $\theta\in\RR$. As we will see, this fact    plays a major role in the derivation of the MLE $(\hthm,\hgm)$ of $(\th_0,g_0)$.

Let $\hpm=\log\hgm$.  Recalling the definition of $\Psi_n(\th,\ps)$ from \eqref{criterion function: Doss},  we observe that $(\hthm,\hpm)$  satisfies
 \[(\hthm,\hpm)=\argmax_{\th\in\RR,\ps\in\mathcal{SC}_0}\Psi_n(\th,\ps).\]
  Let us define
  \[\pst:=\argmax_{\ps\in\mathcal{SC}_0}\Psi_n(\th,\ps),\]
 which  exists for a fixed $\th\in\RR$  by Theorem~2.1(c) of \cite{dosssymmetric}. It is not hard to see that
   \[\hthm=\argmax_{\th\in\RR}\Psi_n(\th,\pst)\quad \text{and}\quad\hpm=\widehat{\ps}_{\hthm}.\]

Note that $\hgm=e^{\hpm}$ is the MLE of $g_0$,  and $\hgf=\hgm(\mathord{\cdot}-\hthm)$ is the MLE of $f_0$. We let $\hFm$ and $\hgF$ denote the  distribution functions corresponding to $\hgm$ and $\hgf$, respectively.  Also, we let $\hpfm$ denote the  log-density $\log\hgf$. 
 
   Before getting into the analysis of the existence of $(\hthm,\hpm)$ for a general case,  let us consider a special case first.  Suppose $\Fm$ is degenerate, i.e.  $\Fm\{x_0\}=1$ for some $x_0\in\RR$.
 Intuition leads us to make the guess that in this case, $\hthm=x_0$.  Now for $k\geq 1$, we denote $A_k$ to be the set $[-1/(2k),1/(2k)]$, and consider
 the sequence of  functions $\{\ps_k\}_{k\geq 1}\in\mathcal{SC}_0$ defined by
 \[\ps_k(x)=\begin{cases}
 \begin{matrix}
 \log k, & \quad x\in A_k\\
 -\infty, &\quad  \text{o.w.}
 \end{matrix}
 \end{cases}\]
   Observe that
 \[\Psi_n(x_0,\ps_k)= \log k-1\to\infty,\quad\text{ as }k\to\infty.\]
 Therefore, $x_0$ indeed is  a candidate for the MLE of $\th_0$. However,  the MLE of $\ps_0$, i.e. $\hpm$, does not exist in this case. To verify, observe that if \hpm\ does exist for some $\hthm\in\RR$, we also have
 \[\hpm(x_0-\hthm)-\edint e^{\hpm(x)}dx= \Psi_n(x_0,\hpm)\geq \lim_{k\to\infty}\Psi_n(x_0,\ps_k)=\infty,\]
 leading to
 \[\hpm(x_0-\hthm)=\infty,\]
 which contradicts the fact that $\hpm$ is a proper concave function. 
 Hence, we conclude that the MLE of $(\th_0,\ps_0)$ does not exist when  $\Fm$ is degenerate. So we only need to consider the case when $\Fm$ is non-degenerate.
 In this case, we will show that, the map $\th\mapsto\Psi_n(\th,\pst)$ is continuous, and the MLE $(\hthm,\hpm)$ exists.
\begin{lemma}\label{theorem: continuity of criterion function}
If $\Fm$ is non-degenerate,
$\th\mapsto\Psi_n(\th,\pst)$ is a continuous map.
\end{lemma}

Figure~\ref{Plot of criterion function: normal distribution} illustrates the continuity of $\Psi_n(\th,\pst)$ in $\theta$, based on for two standard Gaussian samples of size $5$ and $100$.

\begin{figure}[h]
\centering
\begin{subfigure}{.45\textwidth}
\includegraphics[width=\textwidth]{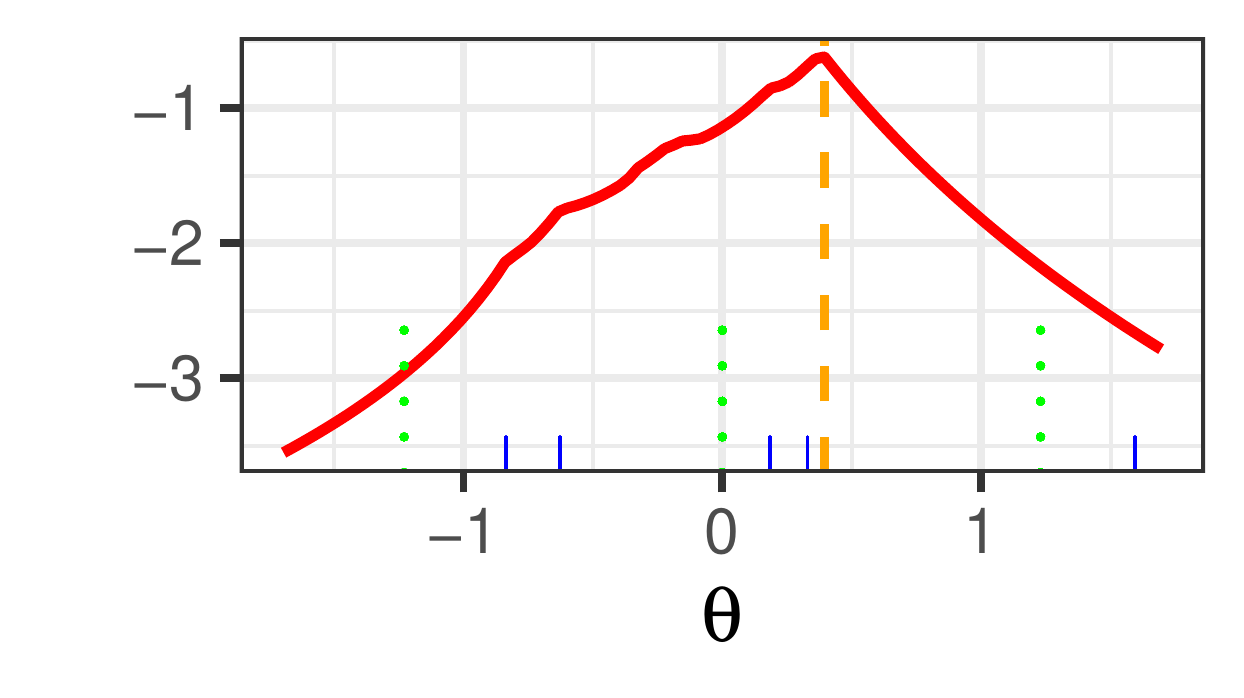}
\caption{$n=5$}
\end{subfigure}~
\begin{subfigure}{.45\textwidth}
\includegraphics[width=\textwidth]{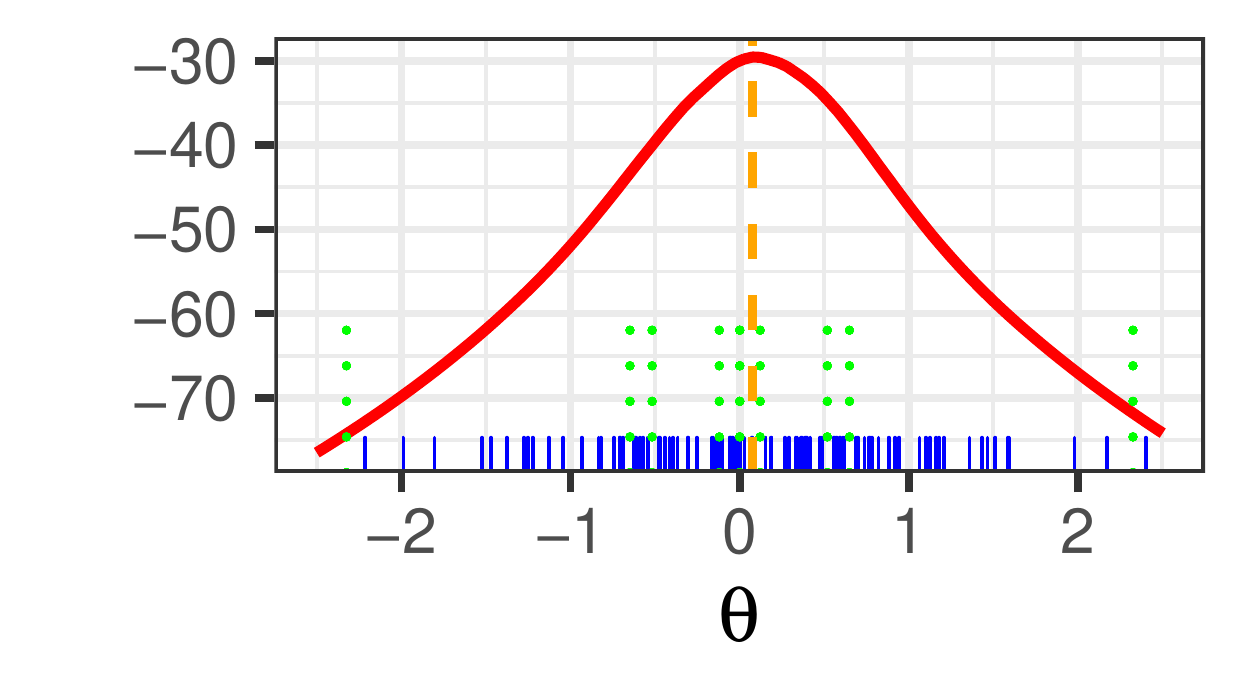}
\caption{$n=100$}
\end{subfigure}
\caption{Plot of $\Psi_n(\th,\pst)$ vs $\th$  generated from the standard Gaussian  distribution. Here the blue ticks, the green  line segments (dotted), and the orange dashed line represent the data points, the knots of $\hpm$, and  $\hthm$  respectively.}\label{Plot of criterion function: normal distribution}
\end{figure}
Our next theorem establishes the existence of a maximizer of $\Psi_n(\th,\pst)$ for non-degenerate $\Fm$. 
\begin{theorem}\label{theorem: existence}
When $\Fm$ is non-degenerate, the MLE $(\widehat{\theta}_n , \widehat{\psi}_{n} )$ of $(\theta_0 , \psi_0 )$ exists. Moreover, $\hthm\in[X_{(1)},X_{(n)}]$.
\end{theorem}

Observe that Theorem~\ref{theorem: existence} does not ensure the uniquenesss of \hthm. Since $\Psi_n(\th,\ps)$ may not be jointly concave in $\theta$ and $\ps$, existence of a maximizer does not lead to its uniqueness. For a particular choice of $\hthm$ though, the estimator   $\hpm=\widehat{\ps}_{\hthm}$ is unique by Theorem~$2.1(c)$ of \cite{dosssymmetric}. Therefore, if $(\th,\ps_1)$ and $(\th,\ps_2)$ both are MLEs of $(\th_0,\ps_0)$, we must have $\ps_1=\ps_2$. 

In all our simulations, \hthm\ turned out to be unique, even when the underlying density $f_0$ was skewed or non-log-concave. Considering this fact, in what follows, we  refer to \hthm\ as ``the MLE" instead of ``an MLE". We must remark that even if $\hthm$ is not unique, all our theorems  still hold for each version of \hthm.   

 Now we discuss some finite-sample properties of $\hpm$. The following theorem sheds some light on the structure of $\hpm$. We find that $\hpm$ is piecewise linear,  and $\mathcal{K}(\hpm)$, i.e. the set of the knots of $\hpm$, is a subset of the dataset. This theorem is a direct consequence of Theorem~$2.1(c)$ of \cite{dosssymmetric}.
%
%
\begin{theorem}\label{theorem: structure of the MLE of the density}
For $\Fm$ non-degenerate, the MLE $\hpm$ of $\ps_0$ is piecewise linear with knots belonging to a subset of the set $\{0,\pm|X_1-\hthm|,\ldots,\pm|X_{n}-\hthm|\}$. Also,  for $x\notin[-|X-\hthm|_{(n)},|X-\hthm|_{(n)}]$, we have $\hpm(x)=-\infty$.
Moreover if $0\notin \{\pm|X_1-\hthm|,\ldots,\pm|X_{n}-\hthm|\}$, then $\hpm'(0\pm)=0.$
\end{theorem}
There is no closed form for $\hpm$,  though $\hpm$ can be characterized by a family of inequalities.  The following two theorems provide  characterizations of $\hpm$. 


\begin{theorem}\label{thm: charactarization 1}
Suppose $\hthm$ is the MLE of $\th_0$. Consider $g=e^{\psi}\in\mathcal{SLC}_{0}$. Then $(\hthm,g)$ is the MLE of $(\th_0,g_0)$ if and only if
\[\edint\Delta(x-\hthm)d\Fm(x)\leq\edint\Delta(x)g(x)dx \]
for all $\Delta:\RR\mapsto\RR$ such that $\psi+t\Delta\in\mathcal{SC}_{0}$ for some $t>0.$ 
\end{theorem}
Note that the MLE  in $\mathcal{LC}$ \citep{2009rufi} or $\mathcal{SLC}_0$ \citep{dosssymmetric} also satisfy similar characterizations. The proof of Theorem~\ref{thm: charactarization 1} follows from Theorem~2.2(c) of \cite{dosssymmetric}. 
 
Our next theorem  exhibits another characterization of $\hpm$. This theorem also follows as a direct consequence of Theorem~2.4(C) of \cite{dosssymmetric}.
%
%
\begin{theorem}\label{thm: Doss-Wellner Theorem 2.4}
Suppose $\hthm$ is the MLE of $\th_0$. For $g\in\mathcal{SLC}_{0}$, define
\[\sFm(x)=2\int_{x}^{|X-\hthm|_{(n)}}g(y)dy,\] 
and
\[\Fzm(x)=n^{-1}\si 1\{|X-\hthm|_{(i)}\geq x\}.\]
 Then $(\hthm,g)$ is the MLE of $(\th_0,g_0)$ if and only if
\begin{align}\label{thm statement: Doss-Wellner Theorem 2.4}
\dint_{t}^{|X-\hthm|_{(n)}}\sFm(x)dx
\begin{cases}
\begin{matrix}
\leq \dint_{t}^{|X-\hthm|_{(n)}}\Fzm(x)dx, & \text{ if }t\in[0,|X-\hthm|_{(n)}],\\
=  \dint_{t}^{|X-\hthm|_{(n)}}\Fzm(x)dx,
& \text{ if }t\in\mathcal{K}(\psi)\cap[0,|X-\hthm|_{(n)}],
\end{matrix}
\end{cases}
\end{align}
where $\psi=\log g$.
\end{theorem}

The implication of Theorem~\ref{thm: Doss-Wellner Theorem 2.4} is that the function
\begin{equation}\label{def: characterization function}
 h_{g}(t)= \int_{t}^{|X-\hthm|_{(n)}}[\Fzm(x)-\sFm(x)]dx,\quad t\geq 0
 \end{equation}
 is non-negative for $g=\hgm$.
 Figure~\ref{Figure: Theorem 5} displays the plot of $h_{\hgm}$  constructed from a standard Gaussian sample of size $50$. 

\begin{figure}[h]
\includegraphics[width=\textwidth]{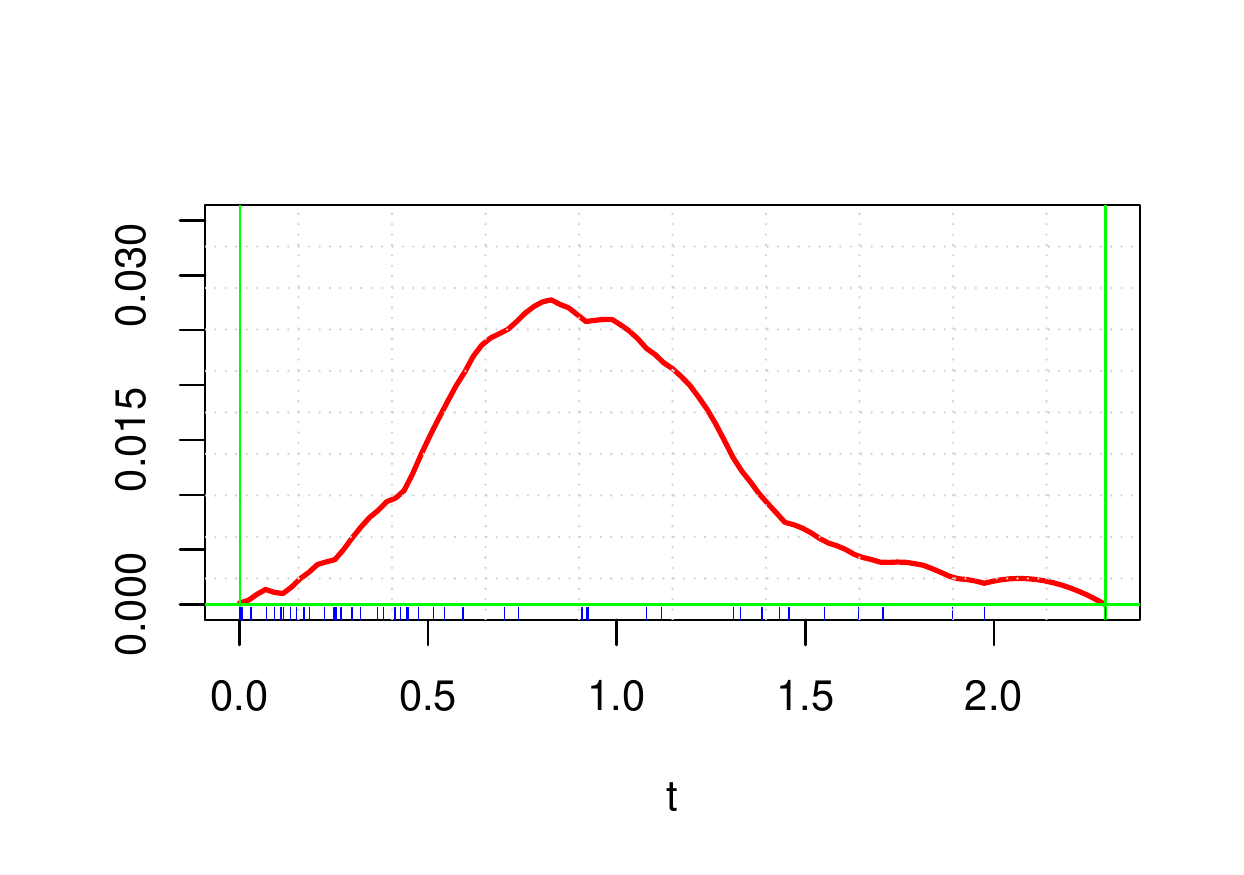}
\caption{Plot of the function $h_{\hgm}$ defined in \eqref{def: characterization function} versus $t$ for a standard Gaussian sample of size $50$. The green vertical lines represents the non-negative knots of \hpm,\ and the blue ticks represent the $|X_i-\hthm|$'s.}\label{Figure: Theorem 5}
\end{figure}
The characterization theorems stated above  lead to some interesting facts about the distribution function  $\hFm$ of $\hgm$. In fact, we show that  $\hFm$ has a close connection with the empirical distribution function of the $|X_i-\hthm|$'s.


\begin{corollary}\label{Corollary 2.7 of Doss-Wellner (2018)}
Suppose $(\hthm,\hpm)$ is the MLE of $(\th_0,\ps_0).$ Let us denote the empirical distribution function of the $|X_i-\hthm|$'s by  $\mathbb{F}_{n,|X-\hthm|}$. Then the following holds almost surely on $\mathcal{K}(\hpm)\cap(0,\infty):$ 
\[\mathbb{F}_{n,|X-\hthm|}-n^{-1}\leq 2\hFm-1\leq \mathbb{F}_{n,|X-\hthm|}.\]
\end{corollary}
The function $ 2\hFm-1$ can be interpreted as the distribution function of the random variable $|Y|$, where  $Y\sim \hFm$.
 Corollary \ref{Corollary 2.7 of Doss-Wellner (2018)} implies  that at the knots, $\mathbb{F}_{n,|X-\hthm|}$ and  $ 2\hFm-1$  differ at most by $1/n$. The proof of Corollary \ref{Corollary 2.7 of Doss-Wellner (2018)} follows from The proof of Corollary 2.7(C) of \cite{dosssymmetric}. 

 Our next corollary provides an upper bound for the second central moment of $\hFm$. This corollary is analogous to Corollary 2.8(C) of \cite{dosssymmetric}. 
 \begin{corollary}\label{Corollary 2.8 of Doss-Wellner (2018)}
Suppose $(\hthm,\hpm)$ is the MLE of $(\th_0,\ps_0).$ Then  
\[\text{Var}(\hFm)=\edint x^2d\hFm(x)\leq \edint (x-\hthm)^2d\Fm(x),\]
where $\hFm$ is the distribution function corresponding to the density $\hgm=e^{\hpm}$.
\end{corollary}
  Since $\int x d\hFm(x)=0$, the proof follows by taking $\Delta=-x^2$ in Theorem~\ref{thm: charactarization 1}.  
%

\section{Asymptotic properties}
\label{sec: asymptotic properties}
We now move on to a discussion of the asymptotic properties of the estimators proposed in sections \ref{sec: one-step} and \ref{sec: MLE}.
Before proceeding any further, we introduce some new notations.
 For  any real valued function $h:\RR\mapsto\RR$, we let $||h||_p$ denote its $L_p$ norm, i.e.
\[||h||_{p}=\lb\edint |h(x)|^p dx\rb^{1/p},\quad p\geq 1.\]
We define the 
Wasserstein distance between two measures $\mu$ and $\nu$ on $\RR$ by
\begin{equation}\label{def: Wasserstein distance}
d_W(\mu,\nu)=\edint|F(x)-G(x)|dx,
\end{equation}
   where $F$ and $G$ are the distribution functions corresponding to $\mu$ and $\nu$ respectively. 
  This representation of $d_W(\mu,\nu)$ follows from  \cite{villani2003}, page $75$. By an abuse of notation, sometime we will denote the above distance by $d_W(F,G)$ as well. For a sequence of distribution functions $\{F_n\}_{n\geq 1}$, we say  $F_n$ converges weakly to  $F$, and write $F_n\to_d F$, if  for all bounded continuous functions $h:\RR\mapsto\RR$, we have $\lim\limits_{n\to\infty}\int h dF_n=\int hdF$.
%


\subsection{Asymptotic properties of the one step estimators}
\label{sec: asymptotics: one step estimator}

In this section,  we will explore the asymptotic properties of the truncated one-step estimator $\hth$. First, we will establish $\sqn$-consistency of this estimator  when $f_0\in\mP_0$. Then we will  study the case when $f_0$ is not log-concave.  In this section, we keep following the notations developed in Section \ref{sec: one-step}. Thus,  $\hn$  denotes any of the density estimators of $g_0$ developed in Section \ref{sec: one-step}.

\subsubsection{Correctly specified model}
We have already mentioned in Section \ref{sec: one-step}  that, since $\hth$ depends on the nuisance  score, its performance relies on how well $\hn$ estimates $g_0$.
Therefore, before going into the analysis of the asymptotic efficiency of \hth, first we have to show that $\hn$ is a reasonably good estimator of $g_0$.
Towards that end, first we   establish that $\hn$ is  $L_1$ consistent for $g_0$. %
 \begin{lemma}\label{lemma: L1 convergence: one step density estimators: model}
Suppose $f_0\in\mathcal{P}_0$. Let $\hn$ be one of the estimators of $g_0$ defined in Section \ref{sec: one-step}. Suppose the preliminary estimator $\bth$ is a  consistent estimator of $\th_0$. Then it follows that
 \begin{equation}\label{L1 convergence of one step estimators: model}
 ||\hn-g_0||_1\to_p 0.
 \end{equation}
  If $\bth$ is a strongly consistent estimator of $\th_0$, then $``\to_p"$ can be replaced by $``\as"$ in the above display.
  \end{lemma}

 Proving Lemma~\ref{lemma: L1 convergence: one step density estimators: model}  is a key step for us because  it leads to further consistency results involving $\hln$, $\hln'$, etc.  
 In particular, we find that $\hln$ converges uniformly to $\ps_0$ over all compact subsets of $\iint(\dom(\ps_0))$ with probability one.
  We can  derive a very similar  result concerning the convergence of $\hln'(x)$ to $\ps_0'(x)$ except that at the derivative level,  the convergence  can only  be proved at the continuity points of $\ps_0'.$   
    The above discussion can be formulated in the following consistency theorem.

\begin{theorem}\label{Theorem: the L1 convergence of the density estimators of one-step estimators: model}
Assume $f_0\in\mathcal{P}_0$. Suppose $\hn$ is one of the estimators of $g_0$ defined in Section \ref{sec: one-step}, and $\bth$ is a consistent estimator of $\th_0$. Let $\{y_n\}_{n\geq 1}$ be any sequence of random variables  converging to $0$ in probability. Then  on any compact set $K\subset\iint(\dom(\ps_0))$ we have,
\begin{itemize}
\item[(A)]
\[\sup_{x\in \RR}|\hn(x+y_n)-g_0(x)|\to_p 0.\]
\item[(B)] 
 \[\sup_{x\in K}|\hln(x+y_n)-\ps_0(x)|\to_p 0.\]
\item[(C)]
Suppose $x\in\iint(\dom(\ps_0))$ is a continuity point of $\ps_0'.$  Then
\[\hln'(x+y_n)\to_p\ps_0'(x).\]
\end{itemize}
  If $\bth$ is a strongly consistent estimator of $\th_0$, then $``\to_p"$ can be replaced by $``\as"$ in the above displays.
\end{theorem}

Now we present a  corollary to Theorem~\ref{Theorem: the L1 convergence of the density estimators of one-step estimators: model}. We show that if $\ps_0$ is unbounded, $\hln'$ imitates $\ps_0'.$ However, we do not know whether the  sequence of concave functions $\hln'$ stays uniformly bounded inside the domain when  $\ps_0'$  is bounded. 

\begin{corollary}\label{lemma: divergence}
Suppose $f_0\in\mP_0$. Then under the assumptions of Theorem~\ref{Theorem: the L1 convergence of the density estimators of one-step estimators: model},   if $\ps_0$ satisfies
\[\sup_{x\in\iint(\dom(\ps_0))}|\ps_0'(x)|=\infty,\]
 we  have
\[\limsup_n\sup_{x\in\iint(\dom(\hln))}|\hln'(x)|=\infty\quad a.s.\]
\end{corollary}

Our next step is analyzing the asymptotic limit  of the truncated Fisher information estimator $\hin(\eta)$.
It  turns out that $\hin(\eta)$ consistently estimates
 \begin{equation}\label{definition: I(eta)}
\I(\eta)=\dint_{F_0^{-1}(\eta)}^{F_0^{-1}(1-\eta)}\ph_0'(x)^2dF_0(x),
 \end{equation}
 which can be seen as a truncated version of the original Fisher information \I.
  The convergence of $\hin(\eta)$, combined with the consistency results in Lemma~\ref{lemma: L1 convergence: one step density estimators: model} and Theorem~\ref{Theorem: the L1 convergence of the density estimators of one-step estimators: model}, leads to the desired consistency of $\hth$.

\begin{lemma}\label{lemma: consistency of Fisher information: model}
Under the conditions of Theorem~\ref{Theorem: the L1 convergence of the density estimators of one-step estimators: model},   we have $\hthm\to_p\th_0$ and $\hin(\eta)\to_p\I(\eta)$, where $\I(\eta)$ was defined in \eqref{definition: I(eta)}. 
 If $\bth$ is a strongly consistent estimator of $\th_0$, then $\hthm\as\th_0$ and $\hin(\eta)\as\I(\eta)$.
\end{lemma}

Now we focus on the rate of convergence of $\hth-\th_0.$ In this case,  we make the assumption that $\bth$ is  $\sqn$-consistent. A similar assumption was made by \cite{stone}.

\begin{theorem}\label{theorem: main: one-step: model}
Suppose $f_0\in\mP_0$ and $\bth$ is a $\sqrt{n}$-consistent estimator of $\th_0$. Then the following assertion holds for the estimator $\hth$ of $\th_0$ defined in \eqref{def: one-step estimator: truncated}:
\begin{equation}\label{statement: one-step: main theorem}
\sqn(\hth-\th_0)\to_{d}N(0,\I(\eta)^{-1}),
\end{equation}
where $\I(\eta)$ is as defined in \eqref{definition: I(eta)}.
\end{theorem} 

\begin{remark}
Although we do not know  the rate of convergence of $\hthf$ for a general $f_0\in\mP_0$, 
in the special case when $f_0$ is bounded away from $0$ on its support, it can be shown  that $\sqn(\hthf-\th_0)$  is asymptotically normal with variance $\I^{-1}$. 
We conjecture that the same holds for any $f_0\in\mP_0$. Our simulation study in Section \ref{sec: simulation} does not refute this conjecture.
\end{remark}

  Theorem~\ref{theorem: main: one-step} underscores the importance of
$\I(\eta)$ for us in that it is inversely proportional to the asymptotic efficiency of $\hth.$ Now we  describe in greater detail 
 how $\eta$ affects the asymptotic efficiency of $\hth.$ 
 
It is natural to expect $\I(\eta)$ to be close to $\I$ if $\eta\in(0,1/2)$ is small. However,  the magnitude of this closeness is likely to be controlled by the underlying density $f_0$. Keeping that in mind, we seek to investigate how the  ratio $\I(\eta)/\I$ behaves as a function of $\eta$ for different choices of $f_0.$  Our choices of $f_0$  include the standard normal, Laplace and logistic densities. Note that all these densities are members of $\mP_0$ with the center of symmetry $\th_0=0.$ We have plotted them in Figure~\ref{Figure: comparison of truncated information densities 1}(a) for convenience.

   Another choice for $f_0$ is the standard symmetrized beta density with parameter $r>2$ (see Figure~\ref{Figure: comparison of truncated information densities 1}(b)), which we define by
\begin{equation}\label{definition: Symmetrized beta}
f_0(x)\equiv f_{0,r}(x)=\dfrac{\Gamma\lb(3+r)/2\rb}{\sqrt{\pi r}\Gamma(1+r/2)}\lb 1-\dfrac{x^2}{r}\rb^{r/2}1_{[-\sqrt{r},\sqrt{r}]}(x),\quad r> 0,
\end{equation}
where $\Gamma$ is the usual Gamma function.
 It is straightforward to verify that in this case the score equals 
 \[\phi_0'(x)=\dfrac{-x}{1-x^2/r}1_{[-\sqrt{r},\sqrt{r}]}(x).\]
 Some computation shows  that $r\leq 2$ leads to $\mathcal{I}_{f_{0,r}}=\infty.$ However   for $r>2$, 
 \[\mathcal{I}_{f_{0,r}}=\dfrac{r\Gamma\lb \dfrac{r}{2}-1\rb \Gamma\lb \dfrac{3+r}{2}\rb}{2\Gamma\lb \dfrac{r}{2}+1\rb \Gamma\lb \dfrac{1+r}{2}\rb}<\infty,\]
  and $f_0\in\mP_0$.
   Figure~\ref{Figure: comparison of truncated information densities 1}(c) displays the plot of $\mathcal{I}_{f_{0,r}}$ versus $r$ for the symmetrized beta density, which depicts that $\mathcal{I}_{f_{0,r}}$  decreases steeply for $r>2$. This finding is consistent with $\I$ being $\infty$ when $f_0$ is the uniform density on $[-1,1].$  

Now we examine the ratio $\I(\eta)/\I$, which gives the asymptotic efficiency of the truncated one step estimators in the correctly specified model.
 Figure~\ref{Figure: comparison of truncated information} illustrates the plot of $\I(\eta)/\I$ versus $\eta$ for the above-mentioned densities.
 It may not be observable from the plot but we verified that for logistic, normal, and the Laplace distribution, the asymptotic efficiency is greater than $0.98$ for $\eta\in(0,0.001).$
However, Figure \ref{Figure: comparison of truncated information} illustrates that for  the symmetrized beta distribution, the situation is somewhat different. We considered  symmetrized beta distributions with parameter $r=2.1$, $2.5$, $3.5$, and $4.5$. Figure~\ref{Figure: comparison of truncated information} indicates that in this case, the asymptotic efficiency is quite small even for smaller values of $\eta$. This trend becomes more extreme as $r$ decreases to $2.1.$  We calculated that for $r=2.1$, as $\eta$ goes down from $0.01$ to $10^{-6}$, the ratio  $\I(\eta)/\I$ only increases from $0.05$ to $0.233$.

The aberration for the symmetrized beta case can be explained inspecting  Figure~\ref{Figure: comparison of truncated information densities 1}(c), which elucidates that as $r$ decreases, the slope of the density near the boundary of the support becomes sharper. As a result, the score, which equals $f_0'/f_0$,  also becomes large. In particular,
\[|\phi'_0(x)|\sim \dfrac{\sqrt{r}/2}{1-|x|/\sqrt{r}}\quad\text{ as }|x|
\to\sqrt{r}.\]
 Therefore, $\mathcal{I}_{f_0,r}$ grows very fast as $r$ decreases.  Since  $\mathcal{I}_{f_0,r}$ does not depend on the value of the score near the boundary, it does not grow as fast as $r\downarrow 0$.  Therefore,   truncation leads to greater loss of asymptotic efficiency for symmetrized beta densities, especially for smaller values of $r$, which is reflected in  Figure~\ref{Figure: comparison of truncated information}.
\begin{figure}[h]
\includegraphics[width=\textwidth, height=.38\textwidth]{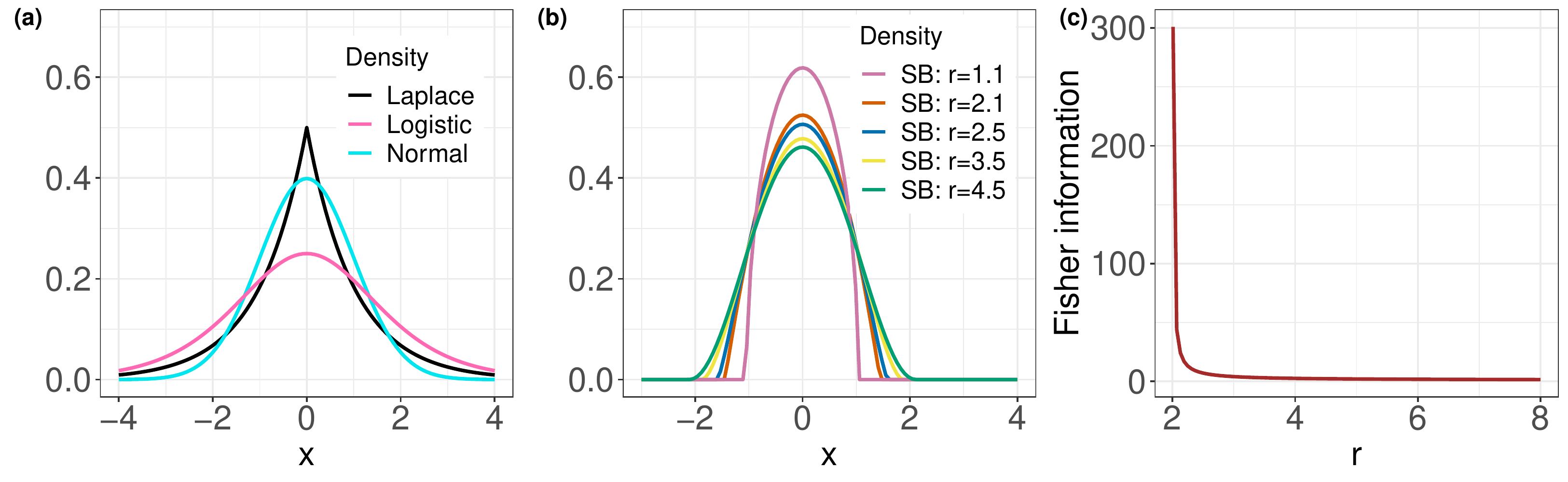}
\caption{
(a) Plot of the standard Laplace, standard normal and standard logistic densities.
(b) Plot of the
 symmetrized beta (SB) densities $f_{0,r}$,  defined in \eqref{definition: Symmetrized beta} for  different values of $r$.
 (c) Plot of Fisher information $\mathcal{I}_{f_{0,r}}$ versus $r$ where $f_{0,r}$ is the symmetrized beta density. }
\label{Figure: comparison of truncated information densities 1}
\end{figure}
\begin{figure}[h]
\includegraphics[width=\textwidth]{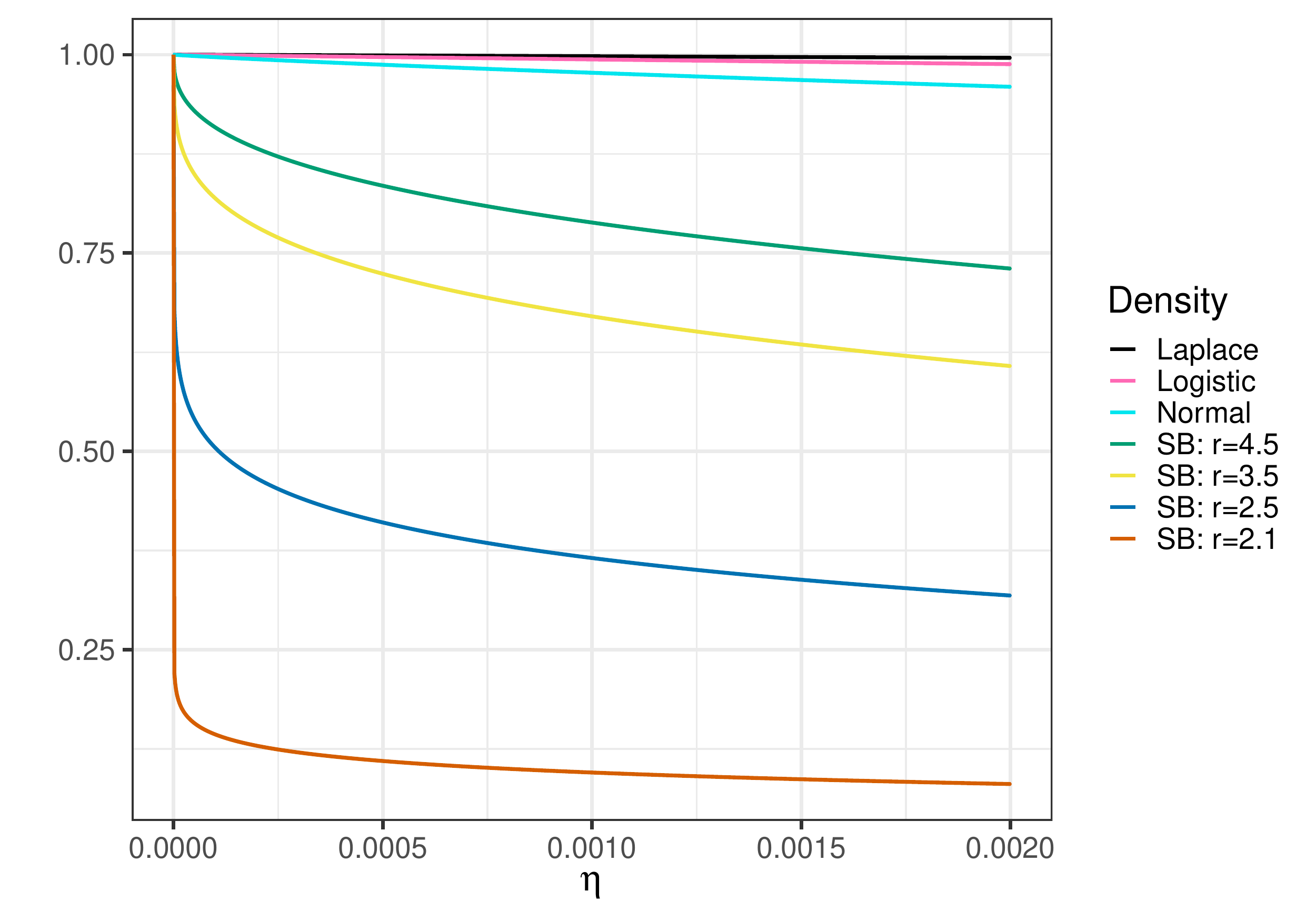}
\caption{Plot of $\I(\eta)/\I$ versus $\eta$ for different densities   }\label{Figure: comparison of truncated information}
\end{figure}


\subsubsection{Misspecified model: when $f_0$ may  not be log-concave}
\label{sec: asymptotic: one-step: misspecified}
The consistency properties of $\hth$ we have discussed so far rely on the log-concavity of $f_0$. However, even if $f_0$ is not log-concave,  the parameter $\th_0$ is still  well-defined as the location parameter as long as $f_0\in\mathcal{S}_{\th_0}$.  It is natural to ask then, in such a case, can \hth\ be still consistent? It turns out that the answer is affirmative under some additional conditions. In what follows, we address this question  in detail.

Let us assume that $f_0\in\mathcal{S}_{\th_0}$. First we will state the additional conditions which are required for the consistency of $\hth$ when $f_0\notin\mP_0$.
First of all, we  impose a restriction on $F_0$. We require $F_0$ to be non-degenerate and to  possess a finite first moment.
\begin{cond}{A}
$F_0$ is a non-degenerate distribution function satisfying 
\[\edint|x|dF_0(x)<\infty.\]
\end{cond}
We will see that this restriction, which also features in \cite{dumbreg}, \cite{smoothed}, and \cite{xuhigh},  is needed for the existence of an $L_1$-limit of  $\hn$. The latter is an important step in proving the consistency of  \hth. Note that, in the special case when $f_0\in\mP_0$,  condition A is automatically satisfied \citep[][Lemma~1] {theory}. We also require $f_0$ to be bounded above, which is needed to ensure the pointwise convergence of $\hn$ to a limit. Note that any member of $\mP_0$ is bounded by Lemma~1 of \cite{theory}.
 Thus, we restrict our attention to the  model
\begin{align}\label{definition: model P1}
\mathcal{P}_1=\bigg\{f\in\mathcal{P}\ \bl\ f\in\mathcal{S}_{\th}\text{ for some  }\th\in\RR,\ \sup_x f(x)<\infty, \int_{-\infty}^{\infty}|x|f(x)dx<\infty\bigg\}.
\end{align}
Note that $\mP_0\subsetneq\mP_1$ but  $\mathcal{P}_1\nsubseteq\mathcal{P}_s$, where the latter was considered by \cite{stone} and  \cite{beran}  (see \eqref{definition: the symmetric model Ps}) because $\mP_1$ allows densities with infinite Fisher information. 

The log-concave projection theorem developed in \cite{dumbreg}, \cite{xuhigh}, and \cite{smoothed} plays a key role in our analysis. 
 Therefore,  we will provide some background on it. For the time being, we assume that $F$ is any distribution function.
we denote by $\tilde{\ph}$ the maximizer of the criterion function 
\begin{equation}\label{definition: new criterion function }
\om(\ph,F)=\edint \ph(x)dF(x)-\edint e^{\ph(x)}dx
\end{equation}
subject to $\ph\in\mathcal{C}$, the class of concave functions defined in \eqref{def: class of concave functions}.  Theorem~$2.7$ of \cite{dumbreg} implies that when $F$ satisfies Condition A, a unique maximizer $\tilde{\ph}$ exists,  and  $\tilde{f}=e^{\tilde{\ph}}$ integrates to  one.  Also, in the case when $F$ has a density $f$,    $\tilde{f}$ can be interpreted as the minimizer of the Kullback-Leiber divergence 
 \[d_{KL}(f,h)=\edint\log\dfrac{f(x)}{h(x)}f(x)dx\]
over all $h\in\mathcal{LC}$. The  distribution function $\tilde{F}$ of $\tilde{f}$ is regarded as the log-concave projection of $F$ onto the space of all distributions with densities in $\mathcal{LC}$.

  Now for  $f_0\in\mP_1$ and $\om$ as defined in \eqref{definition: new criterion function },   we denote
  \begin{equation}\label{def: one-step: limits of misspecified model}
  \php=\argmax_{\phi\in \mathcal{C}}\om(\phi,F_0),\quad \text{ and  }\quad \psp=\php(\mathord{\cdot}+\th_0).
  \end{equation}
   By our earlier arguments, $e^{\php}$ yields a density; which we denote by $\p$, and write \pF\ for the corresponding distribution function.  We also set $\q=e^{\psp}$, and $\pG=\pF(\mathord{\cdot}+\th_0)$. 
  
  We are interested in \p\ because under Condition A, the unconstrained log-concave MLE $\hf$ converges almost surely to $\p$ in $L_1$. See Theorem~$2.15$ of \cite{dumbreg} for a proof.  Therefore, from \eqref{est 1} and \eqref{est 5}, it is not hard to guess that the limit of $\hn^{sym}$ and $\hn^{geo,sym}$ will be dependent on $\p.$ We will show that for $f_0\in\mP_1$, both these density estimators are strongly $L_1$ consistent for $\q$ when $f_0$ satisfies Condition A.

 Next we consider the maximization of the  criterion function $\om(\ph,F)$  in  $\ph$ over the constrained class $\mathcal{SC}_{\th}$ for some fixed $\th\in\RR$. 
 We denote 
 \[ \tph_{\th}=\argmax_{\ph\in\mathcal{SC}_{\th}}\om(\ph,F),\quad\text{ and }\quad\tps_{\th}=\tph_{\th}(\mathord{\cdot}+\th).\]  
 Observe that \eqref{criterion function: xu samworth} provides an alternative representation of $\tps_{\th}$:
 \[\tps_{\th}=\argmax_{\ps\in\mathcal{SC}_0}\Psi(\th,\ps,F).\]
    Proposition 2(iii) of \cite{xuhigh} entails that under Condition A, $\tph_{\th}$ exists, it is unique, and  its exponential  yields a density. Notice that when $F=\Fm$, the density
  $e^{\tps_{\bth}}$ equals our partial MLE estimator $\widehat{g}_{\bth}$, which is why we  study the constrained projection of $F$ as well.  
  
  For a general $F$ satisfying condition A,  $\tph_{\th}$ need not equal $\tilde{\ph}$ since $\tph_{\th}$ is symmetric while $\tilde{\ph}$ is not necessarily symmetric. However, there is some connection between these two. Our next lemma underscores this connection.
 
  \begin{lemma} \label{Lemma: projection theory for a distribution function F}
 Suppose $F$ satisfies condition A.
  Define
\begin{equation}\label{def: symmetrized F}
F^{sym}_{\th}(x)=2^{-1}\lb F(x)+1-F(2\th-x)\rb.
\end{equation}
Then it follows that
\[\argmax_{\ph\in\mathcal{SC}_{\th}}\om(\ph,F)=\argmax_{\ph\in\mathcal{C}}\om(\ph,F^{sym}_{\th}),\]
where  $\om$ is the criterion function defined in \eqref{definition: new criterion function }.
 \end{lemma} 
   Lemma~\ref{Lemma: projection theory for a distribution function F} indicates that if $F$ has a density in $\mathcal{S}_{\th}$, we do have $\tph_{\th}=\tilde{\ph}$.   From the above, it follows that the projection of  $f_0\in\mP_1$   onto $\mathcal{SLC}_{\th_0}$ yields $\p$. 
    
    Now let us consider the case of the partial MLE estimator $\widehat{g}_{\bth}$.
It turns out that if additionally $\bth$ is  consistent, $\widehat{g}_{\bth}$  is an $L_1$ consistent estimator of  $\q$, where $\q$ denotes the centered density $\p(\mathord{\cdot}+\th_0)$.
 
 Indeed for $f_0\in\mP_1$ and $\bth\to_p\th_0$, all the estimators $\hn$ developed in Section \ref{sec: asymptotics: one step estimator} are $L_1$ consistent estimators of  $\q$, except the smoothed symmetrized estimator $\htsm$. Before describing the limit for this case, we will elaborate a little bit on why this difference occurs.

%
%
%

The different behavior of  $\htsm$   stems from the fact that the smoothing parameter $\widehat{b}_n$ defined in \eqref{definition bn} does not necessarily decrease to $0$ as $n$ increases unless $f_0$ is log-concave. Also, Condition A is not sufficient for $\lim_{n\to\infty}\widehat{b}_n$ to exist, for which,  the finiteness of the second central moment of $f_0$ is also required \citep{smoothed}. Provided
\begin{equation}\label{condition: second central moment}
\edint z^2f_0(z)dz<\infty,
\end{equation} 
It can be shown that
 \begin{equation}\label{limit of bn}
 \widehat{b}_n^2\as\tilde{b}^2=\edint z^2 (f_0(z)-\p(z))dz,
 \end{equation}
which is non-negative  for $f_0\in\mP_1$ \citep[cf. Remark 2.3 of][]{dumbreg}, and $0$ for $f_0\in\mP_0$.  Denote by $\qsm$   the convolution of $\q$ and a centered Gaussian density with variance $\tilde{b}^2$.  It turns out that $\hts$ and $\htsm$  are   $L_1$ consistent for $\qsm=\psm(\mathord{\cdot}-\th_0)$.


 Lemma~\ref{lemma: L1 convergence: one step density estimators}, which is a generalization of Lemma~\ref{lemma: L1 convergence: one step density estimators: model},  summarizes the above discussion.   %
 \begin{lemma}\label{lemma: L1 convergence: one step density estimators}
Suppose $f_0\in\mathcal{P}_1$, which was defined in \eqref{definition: model P1}. Let $\hn$ be one of the estimators of $g_0$ defined in Section \ref{sec: one-step}, and suppose that the preliminary estimator  $\bth$ is a  consistent estimator of $\th_0$. Let $\q=e^{\psp}$, where $\psp$ was defined in \eqref{def: one-step: limits of misspecified model}. Then  for $\hn=\hn^{sym}$, $\widehat{g}_{\bth}$, and $\hf^{geo,sym}$, it follows that
 \begin{equation}\label{L1 convergence of one step estimators}
 ||\hn-\q||_1\to_p 0.
 \end{equation}
Suppose $f_0$ additionally satisfies \eqref{condition: second central moment}. Then   we also have
 \[||\htsm-\qsm||_1\to_p 0,\]
 where $\qsm$ is the convolution of $\q$ and a centered Gaussian density with variance $\tilde{b}^2$ defined in \eqref{definition bn}.
 If $\bth$ is a strongly consistent estimator of $\th_0$, then ``$\to_p$" can be replaced by ``$\as$" in the above displays.
  \end{lemma}

 We also have an analog of Theorem~\ref{Theorem: the L1 convergence of the density estimators of one-step estimators: model} for  $f_0\in\mathcal{P}_1$. (See Theorem~\ref{Theorem: the L1 convergence of the density estimators of one-step estimators} in Section \ref{sec:proofs:one-step}.)



Now we are in a position to address the  convergence of $\hin(\eta)$ defined in \eqref{27indef}.  It turns out that, unlike the case for $f_0\in\mP_0$, $\hin(\eta)$ is no longer a consistent  estimator of $\I(\eta)$ because the asymptotic limit of the former now  depends on  $\php'$.  Let us denote $\xi_0=\pG^{-1}(\eta).$ We find that, for $\hn\neq\htsm$, $\hin(\eta)$ consistently estimates $\Ig(\eta)$,  where
 \begin{equation}\label{definition: Ig(eta)}
\Ig(\eta)=\dint_{\pF^{-1}(\eta)}^{\pF^{-1}(1-\eta)}\php'(x)^2  f_0(x)dx=\dint_{-\xi_0}^{\xi_0}\psp'(x)^2  g_0(z)dz.
 \end{equation}
 $\Ig(\eta)$ can be thought of as a truncated Fisher information under model miss-specification. When $\psp=\ps_0$, $\Ig(\eta)$ equals $\I(\eta)$.
 
Now suppose $\hn=\htsm$. In this case, Lemma \ref{lemma: L1 convergence: one step density estimators}  says that $\hn$ consistently estimates $\qsm$ in $L_1$ metric. Letting $\psm=\qsm(\cdot-\bth)$, we denote the distribution function corresponding to $\qsm$ and $\psm$ by $\pGsm$ and $\pFsm$, respectively. We also denote the corresponding  log-densities $\log \qsm$ and $\log\psm$ by $\pspm$ and $\phpm$.  It follows that  $\hin(\eta)$ converges in probability  to
 \begin{align}\label{definition: Ig(eta): smoothed}
\Ig^{sm}(\eta)\ =&\ \dint_{(\pFsm)^{-1}(\eta)}^{(\pFsm)^{-1}(1-\eta)}(\phpm)'(x)^2  f_0(x)dx\nonumber\\
=&\ \dint_{-\xi^{sm}_0}^{\xi^{sm}_0}(\pspm)'(z)^2  g_0(z)dz,
 \end{align}
 where $\xi^{sm}_0$ is the $(1-\eta)$-th quantile of $\pG^{sm}.$ 
%
  
 Even though the limits of  $\hn$ and $\hin(\eta)$ deviate from $g_0$ and $\I(\eta)$ in the misspecified case, it turns out that $\hth$ is still consistent for $\th_0$, provided  $f_0\in\mP_1$. This establishes that $\hth$ is   robust to the violation of log-concavity of $f_0$ as long as $f_0\in\mP_1$. Our next lemma formalizes the above discussion.

\begin{lemma}\label{lemma: consistency of Fisher information}
Under the conditions of Lemma~\ref{lemma: L1 convergence: one step density estimators},  for $\hn=\hn^{sym}$, $\widehat{g}_{\bth}$, and $\hn^{geo,sym}$, we have $\hin(\eta)\to_p\Ig(\eta)$, which was defined in \eqref{definition: Ig(eta)}. Also, $\hthm\to_p\th_0$ in these cases.

  When $\hn=\htsm$, if additionally \eqref{condition: second central moment} holds,  we  have $\hin(\eta)\to_p\Ig^{sm}(\eta)$ $($defined in \eqref{definition: Ig(eta): smoothed}$)$, and $\hthm\to_p\th_0$. When $\bth$ is a strongly consistent estimator of $\th_0$, all the above convergences hold almost surely.
\end{lemma}

Theorem~\ref{theorem: main: one-step: model} established the $\sqn$-consistency of $\hth$ under the $\sqn$-consistency of $\bth$ when $f_0\in\mP_0$. Likewise, Theorem~\ref{theorem: main: one-step} provides an asymptotic representation of $\sqn(\hth-\th_0)$  for $f_0\in\mP_1$.

\begin{theorem}\label{theorem: main: one-step}
Suppose $f_0\in\mP_1$ is absolutely continuous and
 $\bth$ is a $\sqrt{n}$-consistent estimator of $\th_0$. Let $\hth$ be the  truncated one-step estimator defined in \eqref{def: one-step estimator: truncated} with truncation parameter $\eta\in(0,1/2)$.    Suppose $\hn=\hn^{sym}$, $\widehat{g}_{\bth}$ or $\hn^{geo,sym}$. 
Then  under the conditions of Theorem~\ref{Theorem: the L1 convergence of the density estimators of one-step estimators}, the following assertion holds:
\begin{align}\label{statement: one-step: main theorem: misspecified}
\sqn(\hth-\th_0)=\int_{\th_0-\xi_0}^{\th_0+\xi_0}\dfrac{\psp'(x-\th_0)}{\Ig(\eta)}d\mathbb{Z}_n(x)+\sqn(1-\gamma_{\eta})(\bth-\th_0)+o_p(1),\nonumber\\
\end{align}
where $\mathbb{Z}_n=\sqn(\Fm-F_0)$, 
\begin{equation}\label{definition of gamma}
\gamma_{\eta}= 1-\dfrac{\dint_{-\xi_0}^{\xi_0}\psp'(z)\lb\psp'(z)-\ps_0'(z)\rb g_0(z)dz}{\Ig(\eta)},
\end{equation}
 $\xi_0=\pG^{-1}(1-\eta)$, and  $\Ig(\eta)$ is as  defined in \eqref{definition: Ig(eta)}.
When $\hn=\htsm$,  we have
\begin{align}\label{statement: one-step: main theorem: misspecified: smoothef}
\sqn(\hth-\th_0)=\int_{\th_0-\xi_0}^{\th_0+\xi_0}\dfrac{(\pspm)'(x-\th_0)}{\Ig^{sm}(\eta)}d\mathbb{Z}_n(x)\nonumber+\sqn(1-\gamma^{sm}_{\eta})(\bth-\th_0)+o_p(1),
\\
\end{align}
where
 \begin{equation}\label{definition of gammasm}
\gamma^{sm}_{\eta}= 1-\dfrac{\dint_{-\xi^{sm}_0}^{\xi^{sm}_0}(\pspm)'(z)\lb(\pspm)'(z)-\ps_0'(z)\rb g_0(z)dz}{\Ig^{sm}(\eta)},
\end{equation}
 $\xi_0^{sm}$ is the $(1-\eta)$-th quantile of $\pG^{sm}$, and $\Ig^{sm}(\eta)$ is as defined in \eqref{definition: Ig(eta): smoothed}.
\end{theorem} 

Theorem~\ref{theorem: main: one-step} shows that if $\bth$ is an asymptotically linear estimator of $\th_0$, then so is $\hth$. Moreover, the central limit theorem, \eqref{definition: Ig(eta)}, and \eqref{definition: Ig(eta): smoothed} entail that the first terms in the  expansion of $\sqn(\hth-\th_0)$ in \eqref{statement: one-step: main theorem: misspecified} and \eqref{statement: one-step: main theorem: misspecified: smoothef} converge weakly to mean zero Gaussian random variables with variances $\Ig^{-1}$ and $(\Ig^{sm})^{-1}$,  respectively.
From \eqref{statement: one-step: main theorem: misspecified} and \eqref{statement: one-step: main theorem: misspecified: smoothef}, it is also clear that for $f_0\notin\mP_0$, the asymptotic distribution  of $\sqn(\hth-\th_0)$  may depend on $\bth$. The coefficient  $\gamma_{\eta}$ or $\gamma_{\eta}^{sm}$ quantifies this dependence. 

For  $\hn\neq \htsm$,  using Cauchy-Schwartz inequality, we obtain that
\begin{align*}
|1-\gamma_{\eta}|=&\ \dfrac{\dint_{-\xi_0}^{\xi_0}\psp'(z)(\psp'(z)-\ps_0'(z))g_0(z)dz}{\Ig(\eta)}\\
\leq &\ \dfrac{\lb\dint_{-\xi_0}^{\xi_0}\psp'(z)^2g_0(z)dz\rb^{1/2}}{\Ig(\eta)}\lb\dint_{-\xi_0}^{\xi_0}(\psp'(z)-\ps_0'(z))^2g_0(z)dz\rb^{1/2}\\
\leq &\ \lb{\Ig(\eta)}^{-1}\edint (\php'(z)-\ph_0'(z))^2p_0(z)dz\rb ^{1/2},
\end{align*}
where the last inequality follows from  \eqref{definition: Ig(eta)}. 
The above bound on $|1-\gamma_{\eta}|$ hints that  the asymptotic dependence of  $\sqn(\hth-\th_0)$ on  $\bth$ may be  controlled by  the $L_2(P_0)$ distance between $\php'$ and $\phi'_0$. We can draw a similar conclusion for $\gamma_{\eta}^{sm}$ replacing $\psp$ by $\pspm$.

\subsection{Asymptotic properties of the MLE }
\label{sec: asymptotics: MLE}
In this section, we explore the asymptotic properties of the MLE $\hthm$ and $\hgm$. First, we will consider the case when $f_0\in\mP_0$, and then we will move on to the discussion  of the misspecified case.

\subsubsection{Correctly specified model}
\label{sec: asymp: one-step: misspecified}

For $f_0\in\mP_0$,  Lemma~\ref{lemma: consistency of Fisher information: model} showed that the one-step estimators are consistent, whereas Lemma~\ref{lemma: L1 convergence: one step density estimators: model} established the $L_1$-consistency of  the density estimators $\hn$'s established in Section \ref{sec: one-step}. Theorem~\ref{theorem: almost sure convergence of the MLE : true model} below shows that the MLEs $\hthm$ and $\hgm$ satisfy similar consistency 
properties.

\begin{theorem}\label{theorem: almost sure convergence of the MLE : true model}
Suppose $f_0=g_0(\mathord{\cdot}-\th_0)\in\mP_0$ for some $\th_0\in\RR$.
   Then as $n\to\infty$, $L(\Fm)\as L(F_0)$, and the following assertions hold: 
 \begin{enumerate}[label=(\alph*)]
 \item $\hthm\as \th_0$.
  \item $ \hpm(x)\as\ps_0(x)$,  for all $x\in\iint(\dom(\ps_0))$.
 \item  $\limsup_n \hpm(x)\leq \ps_0(x)$, a.s. for all $x\in\RR$.
 \item $ \hgm(x)\as g_0(x)$, for all $x\in\RR$.
 \item There exists $\alpha>0$ such that  for all $\alpha_1<\alpha$,
 \[\edint e^{\alpha_1|x|} |\hgm(x)-g_0(x)|dx\as 0.\]
Moreover, 
  $$\sup_{x\in\RR}e^{\alpha_1|x|}|\hgm(x)-g_0(x)|\as 0.$$
 \end{enumerate}
\end{theorem}
Since $g_0$ is a log-concave density, there exist $\alpha$ and $\beta$ such that $\ps_0(x)\leq -\alpha|x|+\beta$ for all $x\in\RR$ \citep[][Lemma 1]{theory}. This is the $\alpha$ that appears in Theorem~\ref{theorem: almost sure convergence of the MLE : true model}. 

 Our next lemma is an analog of Theorem~\ref{Theorem: the L1 convergence of the density estimators of one-step estimators} for one step estimators, and addresses the  uniform convergence of $\hpfm$ over compact subsets of $\iint(\dom(\ps_0))$ and the pointwise convergence of $\hpfm'$.

\begin{lemma}\label{lemma: MLE: convg-lemma}
Suppose $f_0=g_0(\mathord{\cdot}-\th_0)\in\mP_0$ for some $\th_0\in\RR$. Denoting $\phi_0=\log f_0$, on any compact set $K\subset\dom(\phi_0)$, we have,
  \[\sup_{x\in K}|\hpfm(x)-\phi_0(x)|\as 0.\] 
  Suppose $\phi_0$ is differentiable at $x\in K.$ Then it also follows that
  \[\hpfm'(x)\as\phi_0'(x),\] 
  and
   \[\hgf'(x)\as f_0'(x).\] 
 Similarly for $\psi_0=\log g_0$,  on any compact set $K\subset\dom(\psi_0)$, we have,
\[\sup_{x\in K}|\hpm(x)-\ps_0(x)|\as 0.\] 
If $\ps_0$ is differentiable at $x\in K$, then we also have
\[\hpm'(x)\as\ps_0'(x),\] 
and
   \[\hgm'(x)\as g_0'(x).\] 
\end{lemma}

Now we turn our attention to  the global convergence of $\hgm$. Our Theorem~\ref{MLE: Rate results} establishes that  $H(\hgf,f_0)^2$ and $H(\hgm,g_0)^2$ are $O_p(n^{-4/5})$. This rate is standard for  a density estimator in log-concave class $\mathcal{LC}$. For instance,   the unrestricted log-concave MLE $\hf$ satisfies $H(\hf,f_0)^2=O_p(n^{-4/5})$ \citep[Theorem 3.2]{dossglobal}. Moreover, Theorem~1 of \cite{global2015} establishes that $n^{-4/5}$ is the minimax rate of convergence  for density estimation in $\mathcal{LC}$. 
Theorem~4.1(c) of \cite{dosssymmetric} obtains the same rate in connection to the MLE in the subclass $\mathcal{SLC}_{0}$ of $\mathcal{LC}$. \cite{dosssymmetric} conjectured in their Remark 4.2 that $n^{-4/5}$ is the minimax rate of convergence for estimation in $\mathcal{SLC}_0$ as well. 
  In light of the above, we conjecture that the minimax rate for estimating $f_0$ in $\mP_0$ is also $n^{-4/5}$.
 As a corollary to Theorem~\ref{MLE: Rate results}, it follows that $||\hgm-g_0||_1=O_p(n^{-2/5})$. Also, for the MLE of $\theta_0$,  we can  show that $|\hthm-\th_0|$ is  $O_p(n^{-2/5})$. 

\begin{theorem}\label{MLE: Rate results}
Suppose $\th_0\in\RR$ and $f_0=g_0(\mathord{\cdot}-\th_0)$ for some $g_0\in\mathcal{SLC}_0$.
   Then the following assertions hold:
   \begin{enumerate}[label=(\Alph*),ref=(\Alph*)]
  \item\label{lemma 2: E} As $n\to\infty$, 
  \[H(\hgf,f_0)\as 0\quad \text{and}\quad H(\hgm,g_0)\as 0.\]
Moreover,
  \[ H^2(\hgf,f)=O_p(n^{-4/5}).\]
  \item\label{lemma 2: F} We also have $|\hthm-\th_0|=O_p(n^{-2/5})$ and 
  \[H^2(\hgm,g_0)=O_p(n^{-4/5}).\]
   \end{enumerate}
\end{theorem}

In the special case when  $f_0$ is bounded away from $0$ on its support, it can be shown that $\hthm$ is $\sqn$-consistent.
  We conjecture that for any $f_0\in\mP_0$, the MLE $\hthm$ is $\sqn$-consistent and asymptotically efficient.
\subsubsection{Misspecified model: when $f_0$ may  not be log-concave}
\label{sec: asymp: mle: misspecified}
In Subsection \ref{sec: asymp: one-step: misspecified}, we saw that even if $f_0\notin\mP_0$, the one step estimators are still consistent as long as $f_0\in\mP_1$.
It turns out that,  for  $f_0\in\mP_1$,  not only the one step estimators, but also  the MLEs  $\hthm$ and $\hgm$ are consistent.    However, in case of the MLE, it is possible to say more.  In particular,  if $f_0\notin\mP_1$, it turns out that $\hthm$ converges almost surely to some limit under some  minimal conditions.  Much as in Section \ref{sec: asymptotics: one step estimator}, we again appeal to the log-concave projection theory to analyze the asymptotic properties. However, in this case we need to consider projection onto the class of distributions with densities in $\mP_0$, which is not covered by the existing literature. Therefore, we construct a new framework for projection onto the above-mentioned class. The consistency results for $f_0\in\mP_0$ follow as a special case of the general theory developed in this subsection.



In order to study the log-concave projection in context of the class $\mP_0$,  first we need to introduce some new notation. For any distribution function $F$, let us denote
\begin{equation}\label{MLE: maximizer of criterion}
(\th^*(F),\ps^*(F))=\argmax_{\th\in\RR,\ps\in\mathcal{SC}_0}\Psi(\th,\ps,F),
\end{equation}
where $\Psi$ was defined in \eqref{criterion function: xu samworth}. This maximization problem is different from those considered in Section \ref{sec: asymptotics: one step estimator} because in those cases, the maximization was only with respect to $\ps$. We are interested in $\th^*(F)$ and $\ps^*(F)$  because, as we will see, under some mild conditions, $\hthm\as\th^*(F)$ and $||\hgm-e^{\ps^*(F)}||_1\as 0$ as $n\to\infty$ provided that $\th^*(F)$ is unique.

Proposition \ref{Prop: existence of maximizers} below entails that Condition A is sufficient for the existence of the maximizers.
  Following the same argument as in Theorem~2.1(c) of \cite{dosssymmetric}, one can show that when $\ps^*(F)$ exists,  $g^*(F)=e^{\ps^{*}(F)}$ is a density. We denote the  corresponding distribution function by  $G^*(F)$.  We also denote $\ph^{*}(F)=\ps^{*}(F)(\mathord{\cdot}-\th^{*}(F))$, and write $f^{*}(F)$ for $e^{\ph^{*}(F)}$. The corresponding distribution function will be denoted by $F^{*}(F)$, which can be interpreted as the projection of $F$ onto the space of all distribution functions with densities in $\mathcal{SLC}_{\th^*(F)}$.

 \begin{proposition}\label{Prop: existence of maximizers}
 Suppose $F$ satisfies Condition A. Then $L(F)<\infty$, and $\th^*(F)$ and $\ps^*(F)$  exist for $F$, where $\th^*(F)$ and $\ps^*(F)$ were defined in \eqref{MLE: maximizer of criterion}.
 \end{proposition}   
 Our next lemma  provides  a representation of $\th^*(F)$ and $\ps^*(F)$ in terms of the unconstrained log-concave projection in \eqref{def: one-step: limits of misspecified model}.  Lemma~\ref{Lemma: projection nw} is a direct consequence of Lemma~\ref{Lemma: projection theory for a distribution function F}. 
  \begin{lemma}\label{Lemma: projection nw}
Suppose a distribution function $F$ satisfies condition A.  Then
\[\th^*(F)=\argmax_{\th\in\RR}\lb\max_{\ph\in\mathcal{C}}\om(\ph,F^{sym}_{\th})\rb,\]
and
\[\ps^*(F) =\argmax_{\ph\in\mathcal{C}}\om\lb\ph,F^{sym}_{\th^*(F)}\rb,\]
where  $F^{sym}_{\th}$ is defined in \eqref{def: symmetrized F} for any $\th\in\RR$, and $\om$ is the criterion function defined in \eqref{definition: new criterion function }.
\end{lemma}
The benefit of the representation of $\th^*(F)$ and $\ps^*(F)$ in terms of the unrestricted log-concave projection  lies in the fact that characterizations and  tools for computing the latter are available \citep[cf.][Remark 2.3, Theorem 2.7, Remark 2.10, Remark 2.11]{dumbreg}. 

 Note that neither Proposition \ref{Prop: existence of maximizers} nor Lemma~\ref{Lemma: projection nw}  provides any information about the uniqueness of $\th^*(F)$ and $\ps^*(F)$. However, the uniqueness of  $\th^*(F)$  is  critical for the consistency of $\hthm$. Therefore, we only consider those distribution functions which satisfy the following condition.
  \begin{cond}{B}
The maximizer $\th^*(F)$ defined in \eqref{MLE: maximizer of criterion}  is unique for distribution function $F$.
  \end{cond}
  
Though Condition B does not say anything about the uniqueness of $\ps^*(F)$, from Theorem~2.1(c) of \cite{dosssymmetric} or Proposition $2$ of \cite{xuhigh}, it follows that if $\th^*(F)$ is unique, so is  $\ps^*(F)$. We have not yet discovered whether there exist some characteristics of $F$ which  enforce the uniqueness condition B.
  Interestingly, we have not yet found  an example of $F$ satisfying condition A but violating Condition B. However, it is easy to find  distributions satisfying  both Conditions A and B, a trivial example being distribution functions satisfying 
  \begin{equation}\label{condition: symmetry}
  F(x)=1-F(2\th-x)
  \end{equation}
   for some $\th\in\RR$. If $F$ has a density in $\mathcal{S}_{\theta}$, \eqref{condition: symmetry} holds. 
    \begin{lemma}\label{Lemma: projection for symmetric densities}
Suppose a distribution function $F$ satisfies \eqref{condition: symmetry} for some $\th\in\RR$,  and set $G=F(\mathord{\cdot}+\th)$.   Then condition B is satisfied with 
\[\th^*(F)=\th,\]
and
\[\ps^*(F) =\argmax_{\ps\in\mathcal{SC}_0}\Psi(0,\ps,G)=\argmax_{\ph\in\mathcal{C}}\om(\ph,G),\]
where $\om$ is the criterion function defined in \eqref{definition: new criterion function }.
\end{lemma}
  
  As a corollary to Lemma~\ref{Lemma: projection for symmetric densities}, it follows that the density class $\mP_1$ satisfies Condition B.  The  following example, which is taken from  Example~$2.9$ of \cite{dumbreg}, helps in understanding the implications of Lemma~\ref{Lemma: projection for symmetric densities}. 
  \begin{example}\label{example: class of symmetric distributions}
 Consider a rescaled version of Student's  $t_2$ distribution, whose density takes the form $f(x)=g(x-\th)$,
with $g$ satisfying
\[g(x)=2^{-1}(1+x^2)^{-3/2},\quad x\in\RR.\]
 Lemma~\ref{Lemma: projection for symmetric densities}  implies that $\th^*(F)=\th$.  Lemma~\ref{Lemma: projection for symmetric densities}  also indicates that $\ps^*(F)$ can be found using the unrestricted log-concave projection theorem developed in \cite{dumbreg}. In fact, using  Example $2.9$ of \cite{dumbreg}, we derive that  $\ps^*(F)(x)=-|x|-\log 2$, which corresponds to  the Laplace density. 
  \end{example}

\begin{example}\label{example: class of bimodal symmetric distributions}
 This example focuses on some bimodal densities. First, we consider  a Gaussian mixture density
 \begin{equation}\label{mixture: normal}
  f(x)=2^{-1}(\varphi(x-2)+\varphi(x+2)).
  \end{equation} 
  Here $\varphi$ denotes the standard Gaussian density.  Since $f\in\mathcal{S}_0$, Lemma~\ref{Lemma: projection for symmetric densities} implies that $\th^*(F)=0$. As in Example~\ref{example: class of symmetric distributions}, $\ps^*(F)$ can be obtained using the log-concave projection illustrated in \cite{dumbreg}. Remark 2.11 of \cite{dumbreg} implies that there exists $z>2$ such that the projection $f^*(F)$ equals $f$ on $\RR\setminus[-z,z]$, and is constant on $[-z,z]$ (see Figure~\ref{Plot: mixture}). In addition, since $f^*(F)\in\mathcal{SLC}_0$, we have
  \begin{equation}\label{example 2: equation}
  \edint f^*(F)(x)dx=2zf(z)+2\lb 1-F(z)\rb,
  \end{equation}
 which implies
 \[1=z\lb \varphi(z-2)+\varphi(z+2)\rb+2-\Phi(z-2)-\Phi(z+2),\]
 where $\Phi$ is the standard Gaussian distribution function. Solving the above equation, we obtain that $z\approx 2.83$.
 
  Similarly, consider the following mixture of Laplace densities:
 \begin{equation}\label{mixture: laplace}
 f(x)=4^{-1}\lb e^{-|x-2|}+e^{-|x+2|}\rb.
 \end{equation}
 Its projection $f^*(F)$ can also be found by Remark 2.11 of \cite{dumbreg} which implies  that $f^*(F)$ equals $f$ on $\RR\setminus[-z,z]$, and is constant on $[-z,z]$, where $z> 2$. Using \eqref{example 2: equation}, which  holds in this case as well, we find that $z=2.61$.   Figure~\ref{Plot: mixture} displays the above bimodal densities and their symmetric log-concave projections.
 \begin{figure}[h]
 \centering
 \includegraphics[width=\textwidth]{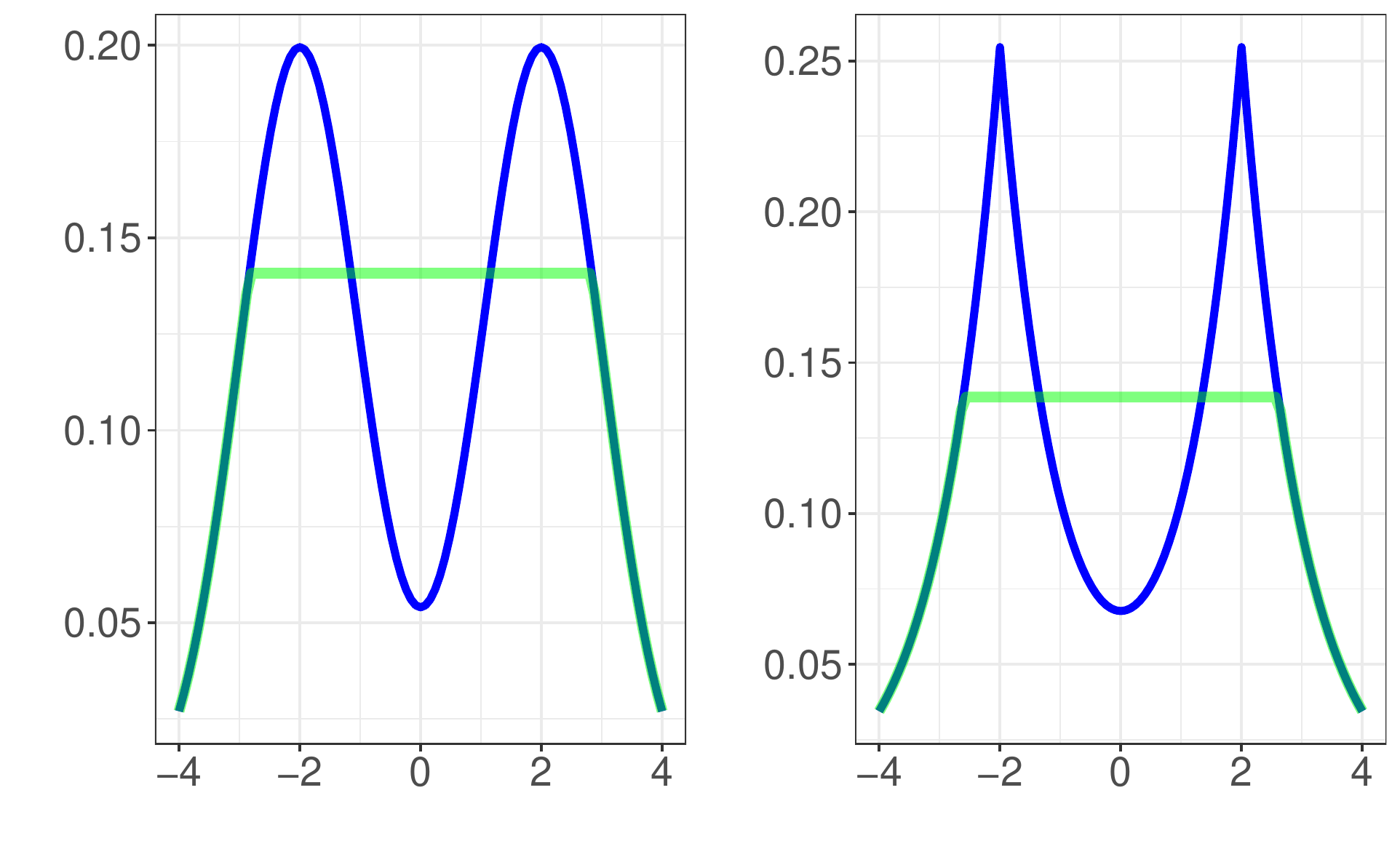}
 \caption{The plot in the left panel displays the density $f$ (in blue) in \eqref{mixture: normal} and its projection $f^*(F)$ (in green); the plot in the right panel displays the  density $f$ in \eqref{mixture: laplace} (drawn in blue) and the corresponding $f^*(F)$ (in green).}\label{Plot: mixture}
 \end{figure}
  \end{example}  
%
%
 Our next lemma gives some insight into the domain of $\ps^*(F).$   
  \begin{lemma}\label{lem: support of projection map}
Suppose Conditions A and B hold for a distribution function $F$.  Then  $\th^*(F)\in\overline{J(F)}$, where $J(F)=\{0<F<1\}$.  Moreover, letting $a$ and $b$ denote  the left and right end-points of the set $J(F)$ respectively, we have,
\[\iint(\dom(\ps^*(F)))=(-d,d),\] 
where $d=(b-\th^*(F))\wedge (\th^*(F)-a).$
\end{lemma}
  Lemma~\ref{lem: support of projection map} implies  that for a general $F_0$, the domain of $\ps_0$ is contained the domain of $\ps^*(F_0)$. In  the special case when $\ps_0\in\mathcal{S}_0$, which is the case for $f_0\in\mP_1$, we have $b-\th^*(F)=\th^*(F)-a$, leading to
\[\iint(\dom(\ps^*(F_0)))=\iint(\dom(\ps_0)).\] 
On the other hand, substituting $F=\Fm$ in Lemma~\ref{lem: support of projection map}, we obtain that $\hthm\in[X_{(1)},X_{(n)}]$, which was already stated in Theorem~\ref{theorem: existence}.

Now we move on to the consistency properties of the MLE, which is connected with  the continuity of the maximized criterion function $F\mapsto L(F)$ defined in \eqref{maximum of criterion function}.
 As mentioned earlier in Section \ref{sec: asymptotics: one step estimator}, \cite{dumbreg}  considered projection of $F$  onto the class of distributions with density in $\mathcal{LC}$. 
  \cite{dumbreg} showed that under  some regularity conditions, the maximized criterion function corresponding to their projection operator 
  is continuous with respect to  the Wasserstein distance $d_W$.  \cite{xuhigh} obtained similar continuity result  in context of  $\mathcal{SLC}_{0}$, which is a subclass of $\mathcal{LC}$. We are working with a further subclass $\mP_0$. In light of the above, it is not unnatural to expect the same continuity result to hold in our set-up. Indeed, our Theorem~\ref{theorem: almost sure convergence: general theorem of the MLE} establishes the continuity of the map $F\mapsto L(F)$ at an $F$ satisfying Conditions A and B. Moreover, the map $F\mapsto\th^*(F)$ is also continuous in this case.

\begin{theorem}\label{theorem: almost sure convergence: general theorem of the MLE}
 Let $\{F_n\}_{n\geq 1}$ be a sequence of distribution functions such that $d_W(F_n,F)\to 0$ for some distribution function $F$. Further suppose, $F$ and $F_n$  satisfy  Conditions A and B for each $n\geq 1$.  Then as $n\to\infty$,  
 \begin{equation}\label{Convergence: maximized criterion function: Fn}
 L(F_n)\to L(F),
 \end{equation} 
and the following assertions hold:
 \begin{enumerate}[label=(\alph*)]
 \item $\lim\limits_n\th^*(F_n)=\th^*(F).$
  \item $\lim\limits_n \ps^*(F_n)(x)=\ps^*(F)(x)$ for all $x\in\iint(\dom(\ps^*(F))).$
 \item  $\limsup_n \ps^*(F_n)(x)\leq \ps^*(F)(x)$ for all $x\in\RR.$
 \item $\lim\limits_n g^*(F_n)(x)= g^*(F)(x)$ for all $x\in\RR.$
 \item There exists $\alpha>0$ such that  for all $\alpha_1<\alpha$,
 \[\lim_n\edint e^{\alpha_1|x|} |g^*(F_n)(x)-g^*(F)(x)|dx= 0.\]
Moreover, if $g^*(F)$ is continuous on $\RR$, then as $n\to\infty$,
  $$\sup_{x\in\RR}e^{\alpha_1|x|}|g^*(F_n)(x)-g^*(F)(x)|\to 0.$$
 \end{enumerate}
\end{theorem}

Now we are in a situation to address the convergence  of $\hthm$ and $\hgm$. Suppose $F_0$  satisfies Conditions A and B. Then Theorem~\ref{theorem: almost sure convergence of the MLE :specific theorem} entails that  $\hthm\as\th^*(F_0)$ and $\hgm$ converges to $e^{\ps^*(F_0)}$  almost surely, both pointwise and in $L_1$.  Theorem~\ref{theorem: almost sure convergence of the MLE :specific theorem} follows as a corollary of  Theorem~\ref{theorem: almost sure convergence: general theorem of the MLE}.  This can be verified easily observing that under Condition A, $d_W(\Fm,F_0)\as 0$ \citep[][Theorem $6.9$]{villani2009}, and $\mathbb{F}_n$ is non-degenerate with probability $1$ for large $n$.

\begin{theorem}\label{theorem: almost sure convergence of the MLE :specific theorem}
Suppose   $F_0$ satisfies Conditions A and B.
   Then as $n\to\infty$,
   $L(\mathbb{F}_n)\as L(F_0)$, and  the following assertions hold: 
 \begin{enumerate}[label=(\alph*)]
 \item $\hthm\as \th^*(F_0)$.
  \item $ \hpm(x)\as\ps^*(F_0)(x)$,  for all $x\in\iint(\dom(\ps^*(F_0)))$.
 \item  $\limsup_n \hpm(x)\leq \ps^*(F_0)(x)$,  for all $x\in\RR$ a.s.
 \item $ \hgm(x)\as g^*(F_0)(x)$, for all $x\in\RR.$
 \item There exists $\alpha>0$ such that  for all $\alpha_1<\alpha$,
 \[\edint e^{\alpha_1|x|} |\hgm(x)-g^*(F_0)(x)|dx\as 0.\]
 Moreover, if $g^*(F_0)$ is continuous on $\RR$, 
  $$\sup_{x\in\RR}e^{\alpha_1|x|}|\hgm(x)-g^*(F_0)(x)|\as 0.$$
 \end{enumerate}
\end{theorem}

Theorem~\ref{theorem: almost sure convergence of the MLE :specific theorem} has some important consequences for $F_0$ satisfying \eqref{condition: symmetry} with $\th_0\in\RR$. Note that  $f_0\in\mP_1$  corresponds to such a distribution.
 Lemma~\ref{Lemma: projection for symmetric densities} implies that Condition B holds for $F_0$ satisfying \eqref{condition: symmetry} with $\th_0\in\RR$, which leads to the following corollary.
 
 \begin{corollary}\label{corollary: almost sure convergence of the MLE :symmetry}
If $F_0$ has  density $f_0\in\mP_1$,  all  conclusions of Theorem~\ref{theorem: almost sure convergence of the MLE :specific theorem} hold with $\th^*(F_0)=\th_0$ and $f^*(F_0)=f_0$.
\end{corollary}

 For $f_0\in\mP_1$, not only $\hthm$ is a consistent estimator of $\th_0$, but also  $\hgm$ converges to $g_0=f_0(\mathord{\cdot}+\th_0)$ almost surely, both pointwise and in  $L_1$. Also, note that the results for the correctly specified model $\mP_0$ in Theorem~\ref{theorem: almost sure convergence of the MLE : true model} follow  directly from Theorem~\ref{theorem: almost sure convergence of the MLE :specific theorem} and Corollary \ref{corollary: almost sure convergence of the MLE :symmetry}.


\section{Simulation study}\label{sec: simulation}

   In this section, we study the performance of the  one step estimators \hth\ and the MLE \hthm\ for several  log-concave densities symmetric about $0$. These densities include the standard normal, standard logistic, and standard Laplace density, which are plotted in Figure~\ref{Figure: comparison of truncated information densities 1}(a).  We also consider the symmetrized beta distribution discussed in Section \ref{sec: asymptotics: one step estimator} (see Figure~\ref{Figure: comparison of truncated information densities 1}(b)), whose density is given by \eqref{definition: Symmetrized beta}. Taking $g_0$ to be any of the above densities, we let $\th_0=0$.
   
   The general set-up of the simulations is as follows. We generate $3000$ samples of size $n=30$, $100$, $200$, and $500$ from each of the above-mentioned densities. Then for each sample, we construct the MLE \hthm, and the one step estimators, both  truncated and  untruncated, defined in \eqref{def: one-step estimator: truncated} and \eqref{def: one-step estimator full} respectively. The truncated one step estimator requires specification of the truncation parameter $\eta$.  We set $\eta$ to be $0.002$ because  Figure~\ref{Figure: comparison of truncated information densities 1} indicates that the value of $\I(\eta)/\I$ does not vary significantly for $\eta\in(0.0001,0.002)$, at least for normal, logistic, and the Laplace distribution. 
   
    Different choices of the preliminary estimator $\bth$ and the density estimator $\hn$ in \eqref{def: one-step estimator full} and \eqref{def: one-step estimator: truncated}  lead to different one-step estimators. The preliminary estimator $\bth$ is generally taken to be the mean, median, or the trimmed mean ($12.5\%$ trimming from each tail). However, for the logistic distribution, we  also consider the parametric MLE as an initial estimator. 
   The density estimator $\hn$  is taken to be either  $\hf(\mathord{\cdot}-\bth)$, or  the symmetrized estimators  in \eqref{est 1}, \eqref{25eq3},  \eqref{def: partial MLE estimator},  and \eqref{est 5}.  The performance of the resulting estimators are compared with the preliminary estimator $\bth$.
   
  To compare the performance of different estimators, we need a measure of efficiency, which we formulate as follows:
   \begin{equation}\label{Def: efficiency}
   \text{Efficiency}(\th_n)=\dfrac{1/(n\I)}{Var(\th_n)},
    \end{equation} 
    where $\th_n$ is some estimator of $\th_0$.
    The reason behind this choice is that $\I^{-1}$ is the lower bound on the asymptotic variance in this problem. 
   Since we can not directly compute $Var(\th_n)$, we will replace it by its  Monte Carlo estimate. 
   
   Our simulations reveal that all our estimators are consistent for $\th_0$.   
  However, our simulations also suggest that among the one step estimators,   the smoothed symmetrized one step estimator in \eqref{25eq3} and the partial MLE estimator in  \eqref{def: partial MLE estimator} have the highest efficiency uniformly over all sample sizes. Therefore, in this section, among the one step estimators, we present the asymptotic efficiency of only these two one step estimators.

%

 
   \subsection{Normal density}
   In this case, $f_0$ is the standard Gaussian density, which implies $\I=1$, and $\th_0=0$. In case of the Gaussian location model, the sample mean is the parametric MLE, which is also  (asymptotically) efficient.
 Figure~\ref{Plot: normal: efficiency} compares the  efficiency of  the MLE, the partial MLE one-step estimator in \eqref{def: partial MLE estimator}, and the smoothed symmetrized one step estimator in \eqref{25eq3}. Some patterns become clear from Figure~\ref{Plot: normal: efficiency}.
 \begin{enumerate}
 \item As expected, the parametric MLE, i.e. the mean, has the highest efficiency.  Our estimators, however, outperform all other preliminary estimators.
 \item The untruncated one step estimators have higher efficiency than their truncated counterparts.
 \item Among our estimators, the smoothed symmetrized estimator has the best  efficiency for small samples. It turns out that it also slightly outperforms the other estimators in large sample ($n=500$) as well.
 \end{enumerate}
  
\begin{figure}[h]
\centering
\includegraphics[width= \textwidth]{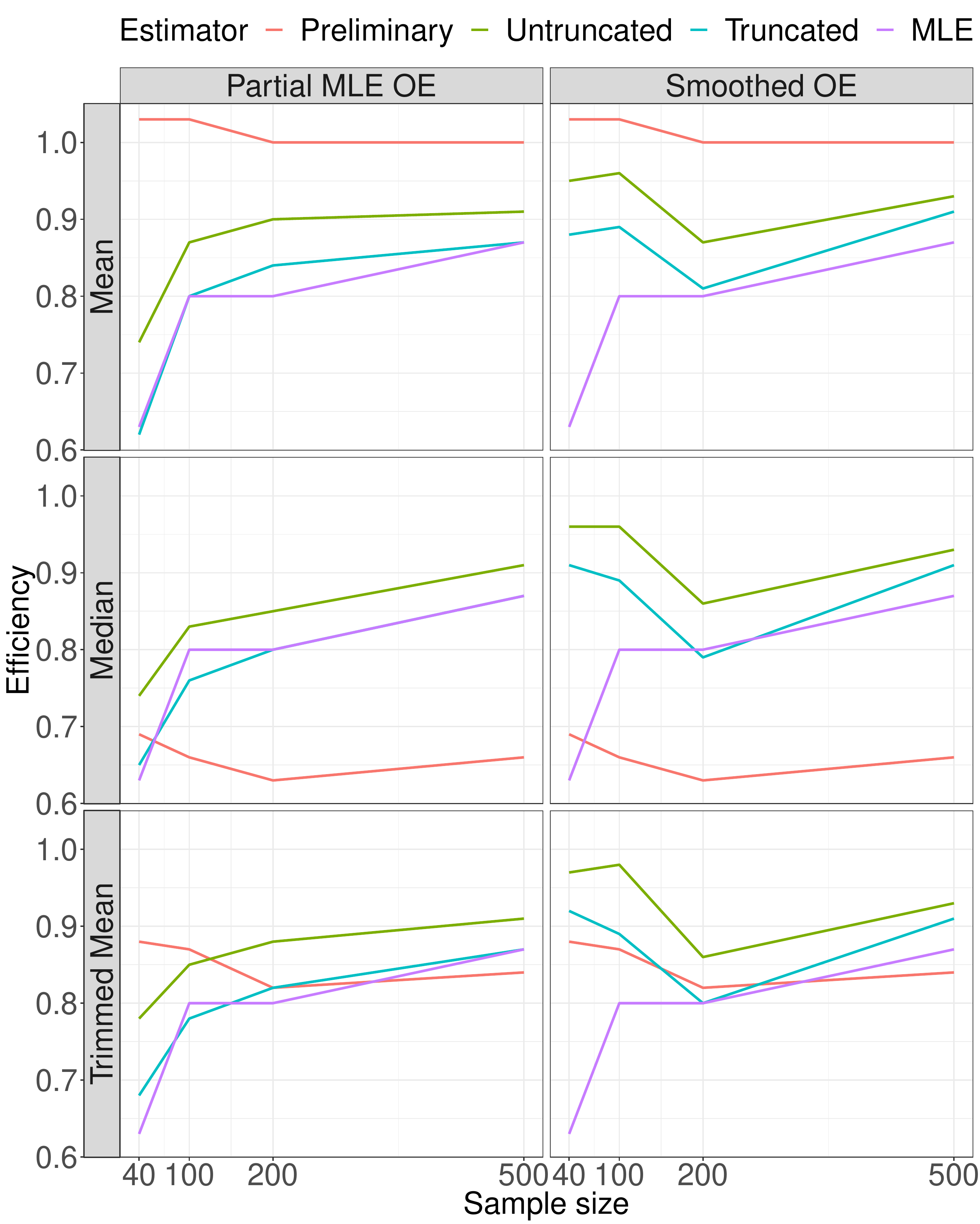}
\caption{Plot of efficiency vs sample size ($n$) for standard Gaussian density. We plotted the efficiency of   the smoothed symmetrized one step estimator in \eqref{25eq3} (Smoothed OE; right column),  the partial MLE one-step estimator  in  \eqref{def: partial MLE estimator} (Partial MLE OE; left column), and the MLE (in purple; present in both columns).  The preliminary estimators $\bth$ corresponding to the one-step estimators are the mean (top), median (middle), and trimmed mean (bottom). 
The preliminary estimator is drawn in red,  the truncated and untruncated estimators are drawn in blue and green respectively. 
}\label{Plot: normal: efficiency}
\end{figure}
\subsection{Logistic density}

The standard logistic density has $\I=1/3$. In this case along with the  mean,  the median, and the trimmed mean, we also consider the parametric MLE  as a preliminary estimator. Figure~\ref{Plot: logistic: efficiency} compares the efficiency of the one step estimators and the MLE. We list the main observations below.
\begin{enumerate}
\item As usual, the parametric MLE, i.e. the mean, has the highest efficiency.
\item Surprisingly, in this case,  the truncated estimator outperforms the untruncated estimator when $n=200$, and $500$, regardless of the choice of the preliminary estimator. For all the other cases, the efficiency of the untruncated estimator dominates the efficiency of the truncated estimator.  
\item   The smoothed symmetrized one step estimator has the highest efficiency in smaller samples (sample size $n=40, 100$). Also, we observe that the MLE has lower efficiency than both the one step estimators, the partial MLE one-step estimator  and the smoothed symmetrized one step estimator.
\end{enumerate}

\begin{figure}[h]
\centering
\includegraphics[width= \textwidth]{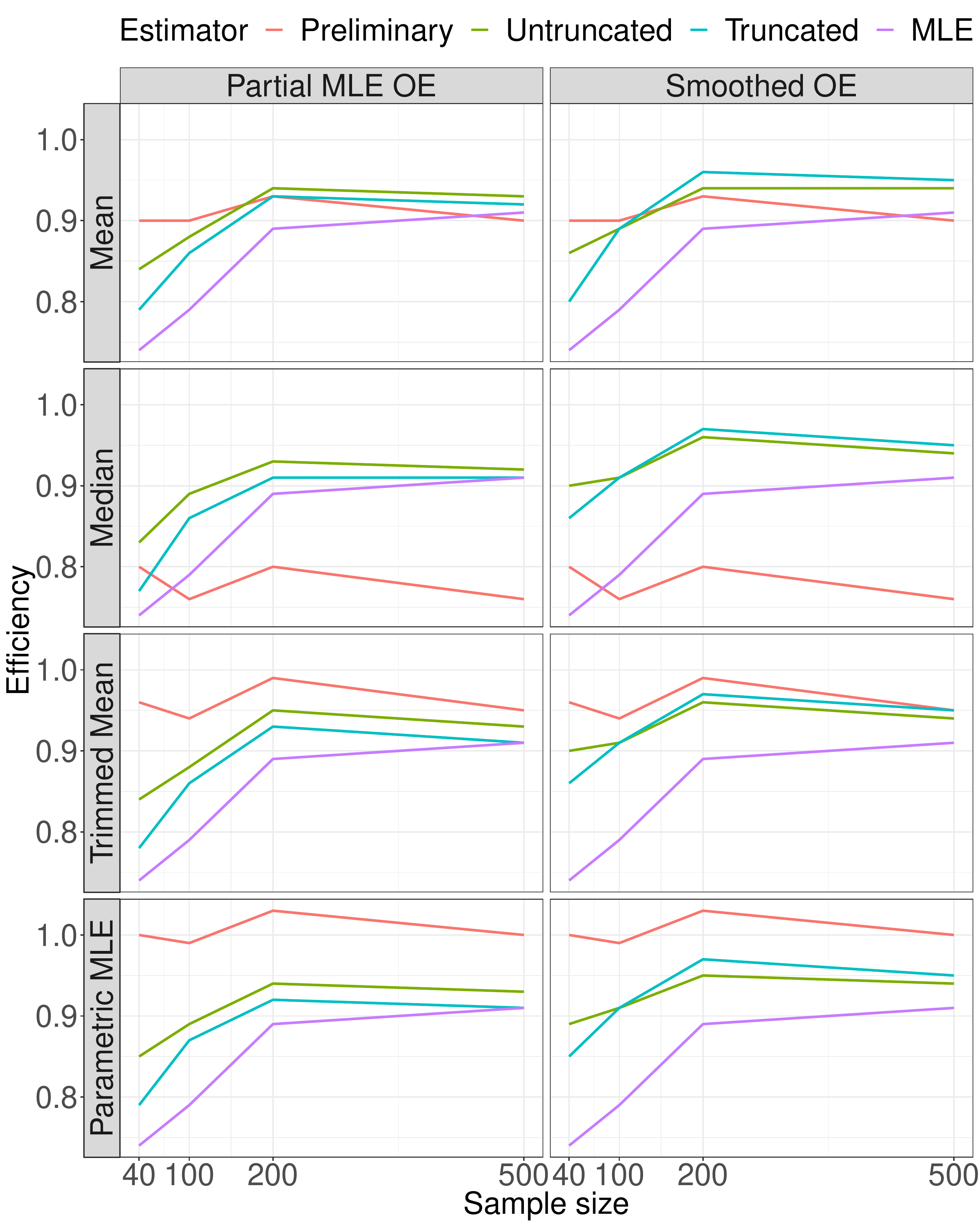}
\caption{Plot of efficiency vs sample size ($n$) when the underlying density is standard Logistic. We calculate the efficiency of   the smoothed symmetrized one step estimator in \eqref{25eq3} (Smoothed OE; right column),  the partial MLE one-step estimator  in  \eqref{def: partial MLE estimator} (Partial MLE OE; left column), and the MLE (in purple; present in both columns). The preliminary estimators $\bth$ corresponding to the one-step estimators (drawn in red) are the mean (first row), median (second row), trimmed mean (third row), and the parametric MLE (fourth MLE). 
The truncated and untruncated estimators are drawn in blue and green respectively. 
}\label{Plot: logistic: efficiency}

\end{figure}
\subsection{Laplace density}

 In this subsection, we simulate observations from the standard Laplace density, which has $\I=1$. In this case, the sample median is the parametric MLE of $\th_0$. Unlike the other densities we considered, this density is not smooth. In fact, it has a kink at $0$ (see Figure~\ref{Figure: comparison of truncated information densities 1}). 
  Figure~\ref{Plot: laplace: efficiency} indicates that the smoothed symmetrized one step estimator in \eqref{25eq3} suffers from the lack of smoothness, and loses its edge to the MLE and the partial MLE one step estimator, both in small and large samples.  When  $n=40$ and $500$, the partial MLE one step estimator has the highest efficiency. For all other sample sizes, Figure~\ref{Plot: laplace: efficiency} suggests that the MLE has the highest efficiency. In all cases, the untruncated one step estimator outperforms the truncated version.
\begin{figure}[h]
\centering
\includegraphics[width= \textwidth]{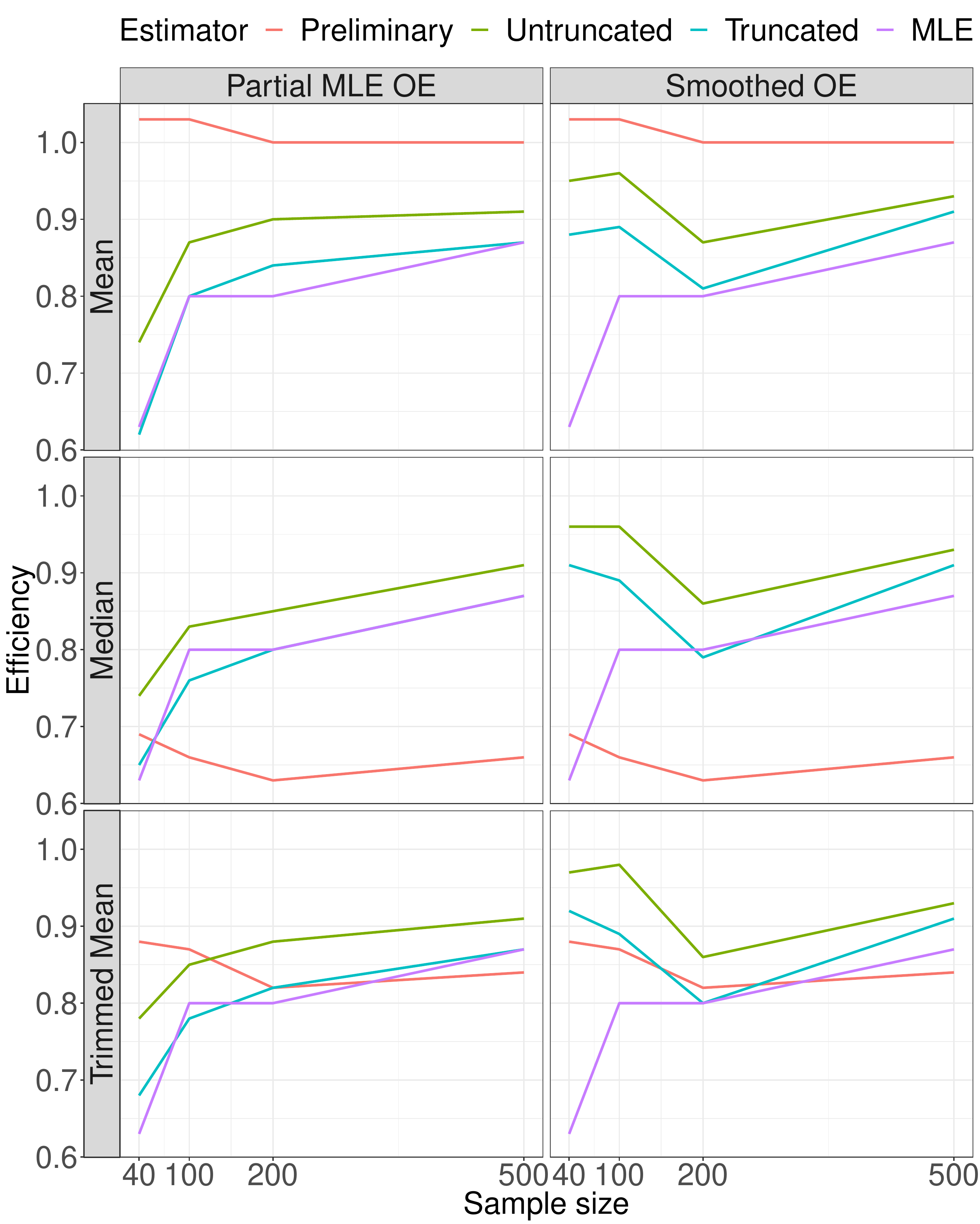}
\caption{Plot of efficiency vs sample size ($n$) for the standard Laplace density. We plotted the efficiency of   the smoothed symmetrized one step estimator in \eqref{25eq3} (Smoothed OE; right column),  the partial MLE one-step estimator  in  \eqref{def: partial MLE estimator} (Partial MLE OE; left column), and the MLE (in purple; present in both columns).  The preliminary estimators $\bth$ corresponding to the one-step estimators are the mean (top), median (middle), and trimmed mean (bottom). 
The preliminary estimator is drawn in red, and the truncated and untruncated estimators are drawn in blue and green respectively. 
}\label{Plot: laplace: efficiency}
\end{figure}

\subsection{Symmetrized beta density}
In this subsection, we consider the symmetrized  beta density in \eqref{definition: Symmetrized beta}, for $r=2.1$, $2.5$, and $4.5$. 
Figure~\ref{Plot: beta: efficiency} compares the efficiency of the MLE and the one step estimators with the mean as the preliminary estimator. The plots for the other preliminary estimators depict the same story, and hence we do not present them. Below we list the main observations from Figure~\ref{Plot: beta: efficiency}.
\begin{enumerate}
\item The efficiency of the estimators increases with $r$. Also, for $r=2.1$ and $2.5$, all the estimators exhibit poor efficiency even for sample size $500$.
\item The untruncated one step estimators display higher efficiency than the truncated one step estimator. We note that for  $r=2.5$ and $4.5$, this gap in the efficiency is higher compared to the other densities.
\item It turns out that the smoothed  untruncated one step estimator has the highest asymptotic efficiency among all the estimators.
\end{enumerate}

\begin{figure}[h]
\centering
\includegraphics[width= \textwidth]{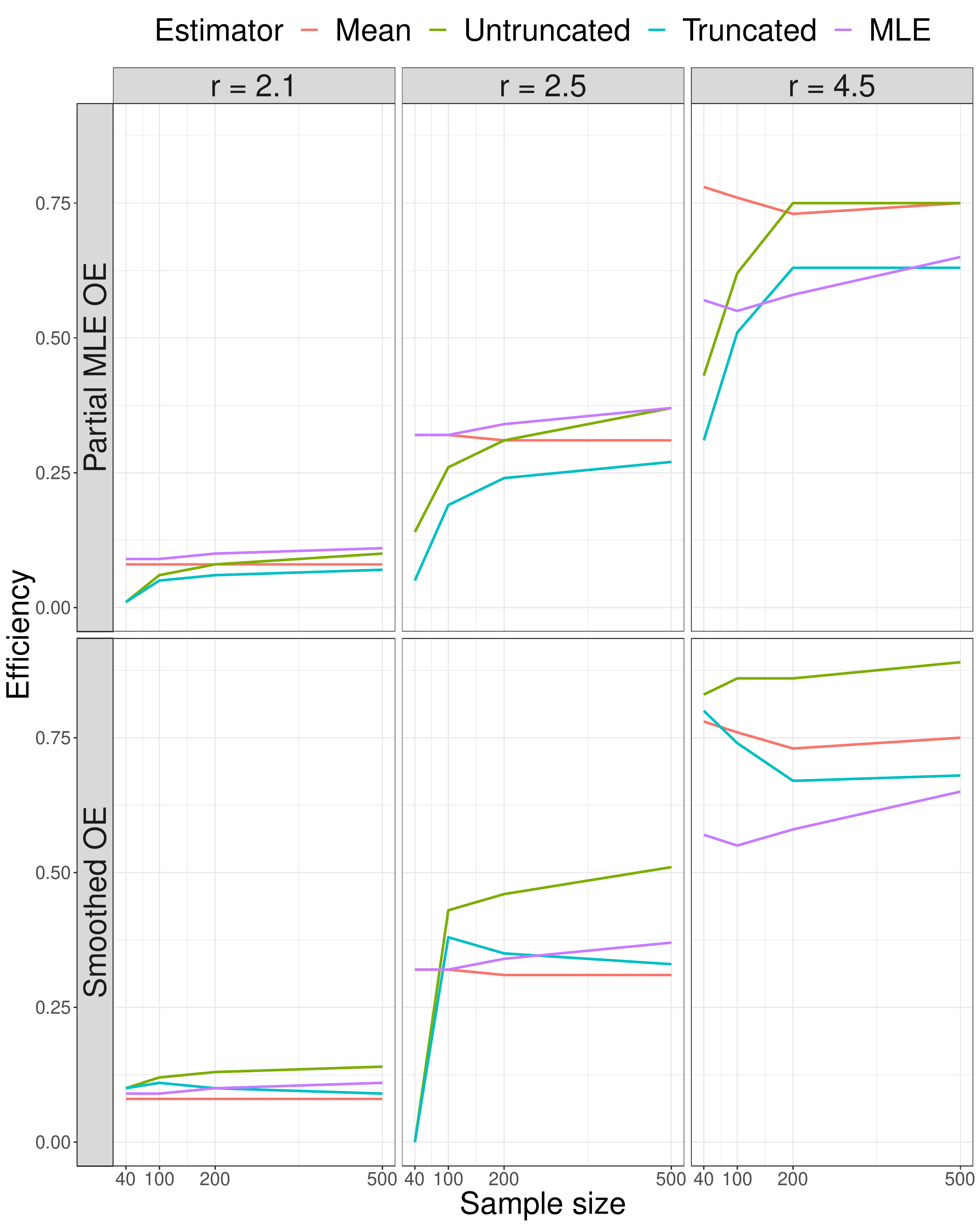}
\caption{Plot of efficiency vs sample size ($n$) when the underlying density is symmetrized beta, which is defined in \eqref{definition: Symmetrized beta}. We consider $r=2.1$ (left), $2.5$ (middle), and $4.5$ (right). The MLE is drawn in purple. Here we calculated the efficiency of   the smoothed symmetrized one step estimator in \eqref{25eq3} (Smoothed OE; right column) and  the partial MLE one-step estimator  in  \eqref{def: partial MLE estimator} (Partial MLE OE; left column) using the  sample mean  as the preliminary estimator. The efficiency of the mean is drawn in red. The truncated and untruncated versions are drawn in blue and green respectively. 
}\label{Plot: beta: efficiency}
\end{figure}

To summarize,  except for the case of standard Laplace,    the smoothed symmetrized one step estimator has the highest efficiency among our estimators. We also note that in most cases the  untruncated one step estimator exhibits higher efficiency than its truncated counterpart.
\FloatBarrier
\section{Discussion}
\label{sec: discussion}
In this article, we revisit the problem of estimating the location parameter $\th_0$ in the symmetric location model.  We show  that the additional assumption of log-concavity can ameliorate the dependence on tuning parameters. In particular, our one-step estimators use only  one additional parameter, and  the MLE does not rely on any tuning parameter. 

We show that the rate of convergence of the MLE \hthm\ is $O_p(n^{-2/5})$ but we conjecture that it is actually $\sqn$-consistent, and asymptotically efficient.
  In our case, the parameter of interest is bundled with  the infinite dimensional nuisance parameter.  Obtaining the precise rate of convergence for the MLE  in such semiparametric models is typically difficult. For example, in context of the monotone current status regression model, this was noted by \cite{murphy1999}, who obtained $n^{-1/3}$ rate for the parametric part in this model. Whether the MLE is $\sqn$-consistent in this model is still unknown \citep{piet}. Bundled parameters also appear in  single index models \citep{bodhida2017}, where not much is known about the MLE. 
  \cite{chen2016} proved the consistency of the joint MLE in  shape-constrained additive index model, which is more general than the single index model, but nothing is known about MLE's rate of convergence.

Although our truncated one step estimator is $\sqn$-consistent, it is not fully asymptotically efficient. The loss of efficiency can be attributed to the use of  a fixed truncation parameter.   A similarity can be drawn with  \cite{piet},  who  study  a  current status regression model. They also  use a  fixed truncation parameter to construct a $\sqn$-consistent $z$-estimator, which, similar to our case,  is not fully asymptotic efficient. However,   our simulations do not exhibit a big difference  between the truncated and untruncated one step estimator, where the latter is conjectured to be $\sqn$-consistent.

  Despite having the above limitations, our estimators are easily computable because  fast algorithms are available for computing log-concave density estimators (using the R packages ``logcondens" and ``logcondens.mode").   Moreover, we show that even if the log-concavity assumption fails, all our estimators of $\th_0$ are still consistent under very mild conditions. In fact, the one step estimators are still $\sqn$-consistent.  This is a consequence of the  impressive stability  of log-concave density estimators \citep{dumbreg, xuhigh}.  This demonstrates the usefulness of  log-concavity  assumption in semiparametric models in facilitating a simplified and robust estimation procedure. 
 
\section{Proofs}
\label{sec:proofs}

 In what follows, we will use the following terminologies. For two sets $A$ and $B$, $A\times B$ will represent the Cartesian product. For any set $A\subset \RR$, and $x\in\RR$, we use the usual notation $A+x$ to denote the translated set  $\{y+x\ :\ y\in A\}$. 
   The notation $\overline{A}$ will refer to the closure of the set $A$.
\subsection{Proofs for the one-step estimators}
\label{sec:proofs:one-step}

 \subsubsection*{Proof of Lemma \ref{lemma: L1 convergence: one step density estimators: model}}
 The proof follows from Lemma~\ref{lemma: L1 convergence: one step density estimators} because when $f_0\in\mP_0$,  the log-concave approximation $\p$ equals $f_0$, which, combined with \eqref{limit of bn},  also implies $\tilde{b}=0$. Hence, $\q=\qsm=g_0$ follows and \eqref{condition: second central moment} is automatically satisfied \citep[][Lemma 1]{theory}.


\subsubsection*{Proof of Theorem \ref{Theorem: the L1 convergence of the density estimators of one-step estimators: model}}
Part (A) of Theorem~\ref{Theorem: the L1 convergence of the density estimators of one-step estimators: model} follows from part (A) of Theorem~\ref{Theorem: the L1 convergence of the density estimators of one-step estimators} and Proposition 2(c) of \cite{theory}. Part (B) and (C) of Theorem~\ref{Theorem: the L1 convergence of the density estimators of one-step estimators: model} follow from part (B) and (C) of Theorem~\ref{Theorem: the L1 convergence of the density estimators of one-step estimators}.
\subsubsection*{Proof of Corollary \ref{lemma: divergence}}
 Suppose that
\[\sup_{x\in\iint(\dom(\ps_0))}|\ps'_0(x)|=\infty,\]
 but there exists $M>0$ such that $\sup_{x\in\iint(\dom(\hln))}|\hln'(x)|<M$ for all sufficiently large $n$. We claim that the probability of such an event is $0$. It is straightforward to see that  there exists $y\in\iint(\dom(\ph_0))$ such that $|\ps_0'(y)|>2M.$ One can also choose $y$ such that $\ps_0'$ is continuous at $y$. 
 Therefore, pointwise convergence of $\hln'$ to $\ps_0'$  fails at $y$. Hence,
 the proof follows by Theorem~\ref{Theorem: the L1 convergence of the density estimators of one-step estimators}.\hfill $\Box$
\subsubsection*{Proof of Lemma \ref{lemma: consistency of Fisher information: model}}
The proof follows from Lemma~\ref{lemma: consistency of Fisher information} because for $f_0\in\mP_0$,
\begin{equation}\label{inlemma: equality of info}
 \Ig(\eta)=\Ig^{sm}(\eta)=\I(\eta).
\end{equation}
\subsubsection*{Proof of Theorem \ref{theorem: main: one-step: model}}
   The proof follows from Theorem~\ref{theorem: main: one-step}. In this case,  $\psp=\pspm=\ps_0$ and \eqref{inlemma: equality of info} holds. As a result, the first terms on the right hand sides of \eqref{statement: one-step: main theorem: misspecified} and \eqref{statement: one-step: main theorem: misspecified: smoothef} converge weakly to a centered Gaussian random variable with variance $\I(\eta)^{-1}$ by the central limit theorem. Also, since $\gamma_{\eta}=\gamma_{\eta}^{sm}=1$, the second terms on the right hand sides of \eqref{statement: one-step: main theorem: misspecified} and \eqref{statement: one-step: main theorem: misspecified: smoothef} vanish, which completes the proof.
\subsubsection*{Proof of Lemma \ref{Lemma: projection theory for a distribution function F} }
Note that each $\ph\in\mathcal{SC}_{\th}$ can be written as $\ps(\mathord{\cdot}-\th)$ for some $\ps\in\mathcal{SC}_0$, and by \eqref{definition: new criterion function } and \eqref{criterion function: xu samworth}, we also have
\begin{equation}\label{connection: om and psi}
\om(\ph,F)=\Psi(\th,\ps,F).
\end{equation}
For $\ps\in\mathcal{SC}_0$ we calculate,
\begin{align*}
\Psi(0,\ps,F)=&\ \dint_{-\infty}^{0}\ps(x)dF(x)+\dint_{0}^{\infty}\ps(x)dF(x)-\edint e^{\ps(x)}dx\\
=&\ -\dint_{0}^{\infty}\ps(-x)dF(-x)+\dint_{0}^{\infty}\ps(x)dF(x)-\edint e^{\ps(x)}dx\\
=&\ \dint_{0}^{\infty}\ps(x)d(F(x)-F(-x))-\edint e^{\ps(x)}dx.
\end{align*}
It also follows that
\begin{align*}
\dint_{0}^{\infty}\ps(x)d(F(x)-F(-x))=\dint_{-\infty}^{0}\ps(x)d(F(x)-F(-x)).
\end{align*}
Therefore,
\[\Psi(0,\ps,F)= 2^{-1}\edint\ps(x)d(F(x)-F(-x))-\edint e^{\ps(x)}dx,\]
implying
\begin{align*}
\Psi(\th,\ps,F)=&\ 2^{-1}\edint\ps(x)d(F(\th+x)-F(\th-x))-\edint e^{\ps(x)}dx\\
=&\  2^{-1}\edint\ps(z-\th)d(F(z)-F(2\th-z))-\edint e^{\ps(z)}dz\\
=&\ \edint\ps(z-\th)dF^{sym}_{\th}(z)-\edint e^{\ps(z)}dz\\
=&\ \Psi(\th,\ps,F^{sym}_{\th}).
\end{align*}
Thus \eqref{connection: om and psi} shows that for any $\ph\in\mathcal{SC}_{\th}$,
\[\om(\ph,F)=\om(\ph,F^{sym}_{\th})\]
Therefore,
\begin{align*}
\argmax_{\ph\in\mathcal{SC}_{\th}}\om(\ph,F)=\argmax_{\ph\in\mathcal{SC}_{\th}}\om(\ph,F^{sym}_{\th}),
\end{align*}
but the distribution function
\[F^{sym}_{\th}(x)=2^{-1}\lb F(x)+1-F(2\th-x)\rb\]
satisfies Condition A and \eqref{condition: symmetry}. Hence, applying Lemma~\ref{lemma: symmetry of projection operator}, we obtain that
\[\argmax_{\ph\in\mathcal{SC}_{\th}}\om(\ph,F^{sym}_{\th})=\argmax_{\ph\in\mathcal{C}}\om(\ph,F^{sym}_{\th}),\]
which exists by Condition A, thus completing the proof.
\hfill $\Box$

\subsubsection*{Proof of Lemma \ref{lemma: L1 convergence: one step density estimators}}
First note that  it suffices to prove the current lemma when $\bth\as \th_0$.  Suppose the strong consistency does not hold but $\bth\to_p\th_0$. Then given any subsequence of $\{\bth\}_{n\geq 1}$, there exists a further subsequence $\{\bar{\th}_{n_k}\}_{k\geq 1}$ such that $\bar{\th}_{n_k}\as \th_0$ as $k\to\infty$. Therefore, along this subsequence $\{n_k\}_{k\geq 1}$, the $L_1$ distance between $\hn$ and $\p$ or $\p^{sym}$ (in case of $\hn=\htsm$) approaches $0$ almost surely. Hence, if $||\hn-\q||_1$ or $||\hn-\q^{sm}||_1$  has any  convergent subsequence at all, the subsequence converges to $0$ almost surely. In that case  \cite{shorack2000} (Theorem~5.7, page 57) implies  that the whole sequence converges in probability to $0$. Therefore, in what follows, we assume that $\bth\as \th_0$.

For our first estimator $\hn^{sym}$, the proof is motivated by Theorem~$4$ of \cite{theory}, which states that if  $\int\log f_0(x) dF_0(x)<\infty$, and  Condition A holds,
\begin{equation}\label{convergence: L1: hf}
\edint |\hf(x)-\p(x)|dx\as 0.
\end{equation}
Note that $f_0\in\mP_1$ is bounded, so $\int\log_{+} f_0(x) dF_0(x)<\infty$ trivially holds, where $\log_{+}(y)=\max\{\log y, 0\}$. 
To prove the $L_1$ consistency of $\hn^{sym}$, we now exploit  the symmetry of $\p$ and $\hf$ about $\th_0$ and $\bth$ respectively. Recalling the definition of $\hn^{sym}$ from \eqref{est 1}, we calculate that
\begin{align}\label{inlemma: L1 convergence of the one step estimators: the ones with the MLE}
\MoveEqLeft2 \edint |\hn^{sym}(x)-\q(x)|dx\nonumber\\
=&\ 2\edint |\hn^{sym}(x)-\p(\th_0+x)|dx\nonumber\\
\leq &\ \edint |\hf(\bth+x)-\p(\bth+x)|dx+\edint |\hf(\bth-x)-\p(\bth-x)|dx\nonumber\\
&\ +\edint|\p(\bth+x)-\p(\th_0+x)|dx+\edint|\p(\bth-x)-\p(\th_0-x)|dx,\nonumber\\
\end{align}
whose first two terms approach $0$ almost surely by \eqref{convergence: L1: hf}. Therefore, we only need to take care of the last two terms of \eqref{inlemma: L1 convergence of the one step estimators: the ones with the MLE}. Since $\p$ is log-concave, it is continuous almost everywhere on $\RR$ with respect to the Lebesgue measure. Therefore, $|\p(\bth+x)-\p(\th_0+x)|\as 0$ almost everywhere  with respect to the  Lebesgue measure. Application of Glick's Theorem \citep[Theorem 2.6][]{devroye1987} then ensures that the last two terms of \eqref{inlemma: L1 convergence of the one step estimators: the ones with the MLE} converge to $0$ almost surely as well.

To establish the convergence of $\htsm$ to $\qsm$ in $L_1$, we will exploit the representation of $\htsm$ in terms of $(\hf)^{sm}$ established in \eqref{representation of htsm}.   
  Since $\qsm$ and $\htsm$ are symmetric about $\th_0$ and $\bth$ respectively,   similar algebra as in the case of $\hn=\hn^{sym}$ leads to
   \begin{align}\label{inlemma: L1 convergence: smoothed}
\MoveEqLeft 2\edint |\htsm(x)-\qsm(x)|dx\nonumber\\
\leq & \edint |\hts(\bth+x)-\psm(\bth+x)|dx +\edint |\hts(\bth-x)-\psm(\bth-x)|dx\nonumber\\
&\ +\edint|\psm(\bth+x)-\psm(\th_0+x)|dx\nonumber\\
&\ +\edint|\psm(\bth-x)-\psm(\th_0-x)|dx.\nonumber\\
\end{align}
The first two terms can be handled by using Theorem~1 of  \cite{smoothed}, which implies that when $f_0$ satisfies \eqref{condition: second central moment}, we have
\begin{equation}\label{convergence: L1: hts}
\edint|\hts(x)-\psm(x)|dx\as 0.
\end{equation}
  Since $\psm$ is the convolution of  a log-concave density and a Gaussian density, it is continuous almost everywhere on $\RR$ with respect to the Lebesgue measure. Therefore, imitating the proof for the case of $\hn=\hn^{sym}$, we can show that the last two terms of \eqref{inlemma: L1 convergence: smoothed} are negligible.
  
%
 
To establish the $L_1$ consistency of the partial MLE estimator $\widehat{g}_{\bth}$, we appeal to  the  projection  theory developed in  \cite{xuhigh}. In this case $\widehat{g}_{\bth}$ is the density of the projection  of \Fn,  the empirical distribution function of the $X_i-\bth$'s, onto the space of the distribution functions with density in $\mathcal{SLC}_0$. By Proposition $6$ of \cite{xuhigh}, if 
\begin{equation}\label{convergence of empirical: wasserstein metric : partial mle estimator}
 d_W(\Fn, G_0)\as 0,
 \end{equation} 
we have $||\widehat{g}_{\bth}-\q||_1\as 0$ 
as long as Condition A holds.  Hence, we only need to show that $d_W(\Fn, G_0)\as 0$,  for which, by Theorem~$6.9$ of \cite{villani2009}, 
 it suffices to show the following. First, we need to show that
 \begin{equation}\label{condition: dw: first condition}
  \edint|x|d\Fn(x)\as \edint|x|dG_0(x),
  \end{equation} 
  and that \Fn\ converges to $G_0$ weakly with probability one.  
Since $\bth$ is strongly consistent for $\th_0$, and
\begin{align}
 \edint|x|d\Fn(x)=&\edint |x-\bth|d\Fm(x), 
\end{align}
 for any $d>0$,  an application of Glivenko-Cantelli Theorem (for example, see Theorem~$2.4.1$ of \cite{wc}) yields
\begin{align*}
\sup\limits_{\bth\in[\th_0-d,\th+d]}\bl\edint |x-\bth|d(\Fm-F_0)(x)\bl\as 0.
\end{align*}
 On the other hand,  strong consistency of $\bth$ implies  $|x-\bth|\leq|x-\th_0|+1$  with probability one for all sufficiently large $n$,  where the latter is integrable. Therefore, the dominated convergence theorem  leads to
\[\edint |x-\bth|dF_0(x)\as \edint |x-\th_0|dF_{0}(x)=\edint |x|dG_0(x),\]
which proves \eqref{condition: dw: first condition}.

Our next step is to prove the weak convergence of \Fn to $G_0$. To this end, we note that
\begin{equation}\label{lemma 3: charles Emp}
\Fn(x)=\Fm(x+\bth)=\edint 1_{(-\infty,x+\bth]}(z)d\Fm(z),
\end{equation}
which converges almost surely to
\begin{align*}
\edint 1_{(-\infty,x+\th_0]}(z)dF_0(z)=G_0(x)
\end{align*}
 by an application of basic Glivenko-Cantelli Theorem (see Theorem~$2.4.1$ of \cite{wc}), and the fact that $F_0(x+\bth)\as F_0(x+\th_0)$ for all $x\in\RR$. This completes the proof of  strong $L_1$ consistency of $\widehat{g}_{\bth}$ for $\q$.
 
 
Finally, we consider the  geometric mean estimator $\hg^{geo,sym}$. 
  Recall that in the course of showing the consistency of $\hn^{sym}$, we have  shown that the right hand side of \eqref{inlemma: L1 convergence of the one step estimators: the ones with the MLE} converges to $0$ almost surely, from which one can derive that
  \begin{equation}\label{L1 convergence: MLE: centered}
  \edint|\hf(\bth\pm x)-\q( x)|dx\as 0.
  \end{equation}
     Proposition $2$(b) of \cite{theory} shows that \eqref{L1 convergence: MLE: centered} leads to almost sure convergence of
 $\hf(\bth\pm x)$ to  $\q( x)$ almost everywhere on $\RR$ with respect to the Lebesgue measure. As a consequence, it follows that
 \[\hg^{geo,sym}(x)/C^{geo}_{n}\as\lb\q(x)\q(-x)\rb^{1/2}=\q(x)\quad a.e.\ \ x.\]
Recall from \eqref{est 5} that
 \[C^{geo}_{n}=\edint\sqrt{\hf(\bth+x)\hf(\bth-x)}dx.\]
   From Scheff\'{e}'s Lemma it follows that $C^{geo}_{n}\as \int d\tilde{G}_0=1$.
 We have thus established that $\hg^{geo,sym}$ converges pointwise to $\q(x)$ almost surely. The desired $L_1$ consistency then follows from Proposition $2$(c) of \cite{theory}.\hfill $\Box$
\subsubsection*{ Theorem \ref{Theorem: the L1 convergence of the density estimators of one-step estimators}}
\begin{theorem}\label{Theorem: the L1 convergence of the density estimators of one-step estimators}
Assume $f_0\in\mathcal{P}_1$. Suppose $\hn$ is one of the estimators of $g_0$ defined in Section \ref{sec: one-step}, where $\bth$ is a  consistent estimator of $\th_0$. Let $\{y_n\}_{n\geq 1}$ be any sequence of random variables  converging to $0$ in probability. Then for $\hn=\hn^{sym}$, $\widehat{g}_{\bth}$, $\hn^{geo,sym}$, or $\hf(\bth\pm\mathord{\cdot})$, on any compact set $K\subset\iint(\supp(g_0))$ we have,
\begin{itemize}
\item[(A)]
\[\sup_{x\in K}|\hn(x+y_n)-\q(x)|\to_p 0.\]
\item[(B)] 
 \[\sup_{x\in K}|\hln(x+y_n)-\psp(x)|\to_p 0.\]
\item[(C)]
Suppose $x\in\iint(\dom(\psp))$ is a continuity point of $\psp'.$  Then
\[\hln'(x+y_n)\to_p\psp'(x).\]
\end{itemize}
Suppose $f_0$ also satisfies \eqref{condition: second central moment}.
 Then for $\hn=\htsm$ or $\hts(\bth\pm\mathord{\cdot})$, (A)-(C) hold  for any compact set $K\subset\RR$, with $\q$ and $\psp$ replaced by $\qsm$ and $\pspm$ respectively. If $\bth$ is a strongly consistent estimator of $\th_0$, and $y_n\as 0$, then the $``\to_p"$ in the above displays can be replaced by $``\as"$.
\end{theorem}
\begin{proof}
As in the proof of Lemma~\ref{lemma: L1 convergence: one step density estimators}, one can show that it suffices to prove   Theorem~\ref{Theorem: the L1 convergence of the density estimators of one-step estimators} when $\bth\as\th_0$, and $y_n\as 0$.  Hence, in what follows,  we assume that $\bth\as\th_0$, and $y_n\as 0$.

First note that by theorem $2.2$ of \cite{dumbreg}, 
\[\iint(\dom(\psp))=\iint(\supp(g_0)).\]
Also, since $\qsm=e^{\pspm}>0$ on $\RR$, we find that $\iint(\dom(\pspm))=\RR$.
Hence, 
\begin{equation*}
K\subset\iint(\dom(\psp))\subset\iint(\dom(\pspm))=\RR.
\end{equation*}

 When $\hn$ equals $\widehat{g}_{\bth}$,  $\hn^{geo,sym}$, $\hf(\bth\pm\mathord{\cdot})$, or $\hts(\bth\pm\mathord{\cdot})$, $\hn$ is log-concave, which simplifies  the  proof of part A and B.  We first consider the cases when $\hn=\widehat{g}_{\bth}$, $\hn^{geo,sym}$ or $\hf(\bth\pm\mathord{\cdot})$. 
Now to prove part A, note that
\begin{align*}
\edint |\hn(x+y_n)-\q(x)|dx\leq &\ \edint |\hn(x+y_n)-\q(x+y_n)|dx\\
&\ + \edint |\q(x+y_n)-\q(x)|dx,
\end{align*}
whose first term
\[\edint |\hn(x+y_n)-\q(x+y_n)|dx=\edint|\hn(x)-\q(x)|dx\as 0,\]
which follows from either Lemma~\ref{lemma: L1 convergence: one step density estimators} or \eqref{L1 convergence: MLE: centered}.
Also, noting 
\[\q(x+y_n)\as \q(x)\quad a.e.\ x.\ \text{Lebesgue},\]
and applying Glick's Theorem \citep[Theorem 2.6][]{devroye1987}, we derive that
\[ \edint |\q(x+y_n)-\q(x)|dx\as 0.\] 
   Thus we obtain that
   \[\edint |\hn(x+y_n)-\q(x)|dx\as 0,\]
 which combined with  Proposition 2(c) of \cite{theory} yields that
 \[\hn(x+y_n)\as\q(x)\quad a.e.\ x\]
  with respect to Lebesgue measure. As a consequence,
  \[\hln(x+y_n)=\log(\hn(x+y_n))\as\psp(x)\quad a.e.\ x.\text{ Lebesgue}.\]
 Since  $\hln$ is concave for our choices of $\hn$, Theorem~10.8 of \cite{rockafellar} indicates that
 pointwise convergence on $\RR$ translates to uniform convergence on all compact sets inside $\iint(\dom(\psp))$, leading to
 \[\sup_{x\in K}|\hln(x+y_n)-\psp(x)|\as 0.\]
 As a corollary,
 \[\sup_{x\in K}|\hn(x+y_n)-\q(x)|\as 0.\]
Thus we  have established part A and part B of  Theorem~\ref{Theorem: the L1 convergence of the density estimators of one-step estimators} for $\hn=\widehat{g}_{\bth}$,  $\hn^{geo,sym}$ and $\hf(\bth\pm\mathord{\cdot})$.
Now consider the case when $\hn=\hts(\bth\pm x)$. Note that \eqref{convergence: L1: hts} and continuity of $\qsm$ imply that
\[\edint |\hts(\bth\pm x)-\qsm(x)|dx\as 0.\]
Therefore, following the same arguments as in the case of $\hn=\widehat{g}_{\bth}$,  $\hn^{geo,sym}$ and $\hf(\bth\pm\mathord{\cdot})$, we can prove part A and part B for $\hn=\hts(\bth\pm\mathord{\cdot})$.


Now we focus on proving part A  when $\hn$ is the mixture density $\hn^{sym}$. From the definition of $\hn^{sym}$ in \eqref{25eq3} we obtain that 
\begin{align*}
\MoveEqLeft 2\sup_{x\in K}|\hn^{sym}(x+y_n)-\q(x)|\\
\leq  &\  \sup_{x\in K}|\hf(\bth+x+y_n)-\q(x)| +\sup_{x\in K}|\hf(\bth-x+y_n)-\q(x)|.
\end{align*}
Since we have already proved part A for $\hf(\bth\pm\mathord{\cdot})$, it is not hard to see that both terms on the right hand side of the above display converge to $0$ almost surely, which proves part A for $\hn^{sym}$. Part A follows for \htsm\ in a similar way noting  that \eqref{representation of htsm}  implies
  \begin{align*}
\MoveEqLeft 2\sup_{x\in K}|\htsm(x+y_n)-\qsm(x)|\\
\leq  &\  \sup_{x\in K}|\hf^{sm}(\bth+x+y_n)-\qsm(x)| +\sup_{x\in K}|\hf^{sm}(\bth-x+y_n)-\qsm(x)|.
\end{align*}

  Now we  prove part B of Theorem~\ref{Theorem: the L1 convergence of the density estimators of one-step estimators} when $\hn=\hn^{sym}$. 
  To this end, note that we can write
 \begin{align*}
\sup\limits_{x\in K}(\hln^{sym}(x+y_n)-\psp(x))= & \sup\limits_{x\in K}\log \lb\dfrac{\hn^{sym}(x+y_n)-\q(x)}{\q(x)}+1\rb\\
\leq &\ \sup\limits_{x\in K}\dfrac{\hn^{sym}(x+y_n)-\q(x)}{\q(x)},
\end{align*}
because $\log(1+x)\leq x$ for any $x>-1$.
Similarly we can show that
\[\sup\limits_{x\in K}(\psp(x)-\hln^{sym}(x+y_n))\leq \sup\limits_{x\in K}\dfrac{\q(x)-\hn^{sym}(x+y_n)}{\hn^{sym}(x+y_n)},\]
leading to
 \begin{align*}
 \sup_{x\in K}\bl\hln^{sym}(x+y_n)-\psp(x)\bl\leq \dfrac{\sup_{x\in K}|\hn^{sym}(x+y_n)-\q(x)|}{\min\lb\inf\limits_{x\in K}\hn^{sym}(x+y_n),\inf\limits_{x\in K}\q(x)\rb},
 \end{align*}
whose numerator converges to $0$ almost surely by part A of Theorem~ \ref{Theorem: the L1 convergence of the density estimators of one-step estimators}.  Thus, to prove part B for $\hln^{sym}$, we only need to show that the denominator of the term on the right hand side of last display is bounded away from $0$. To this end, notice that since $\q$ is unimodal, and $\hn^{sym}$ is a convex combination of two unimodal densities,  they attain their infimum over $K$ at either of  the endpoints. By  part A, it also follows that as $n\to\infty$, the respective minimum values  approach $\inf\limits_K\q$, which is bounded away from $0$ as  $K$ is a closed subset of  $\iint(\dom(\psp))$. Thus part B is proved for $\hln^{sym}$.

One can prove part B  for $\htsm$ in a similar fashion because $\qsm$ is unimodal, and $\htsm$ is a convex combination of  two unimodal densities. 

 
 Now we can proceed to prove part C of the Theorem~\ref{Theorem: the L1 convergence of the density estimators of one-step estimators}. First, we consider the log-concave density estimators $\widehat{g}_{\bth}$,  $\hn^{geo,sym}$, and $\hf(\bth\pm\mathord{\cdot})$, all of which are strongly $L_1$  consistent for \q.  Note that if  $\psp'$ is continuous at some $x\in\iint(\dom(\psp))$, then there exists a neighborhood $[x-\xi,x+\xi]$ where $\psp$ is continuous. It follows that for sufficiently large $n$,  $x+y_n\in[x-\xi,x+\xi] $  with probability $1$. Since $\hln$ is concave, and converges to $\psp$ uniformly over $K$ by part B, Lemma~$3.10$ of \cite{seijo} entails that 
 \begin{equation}
 \sup\limits_{t\in[x-\xi,x+\xi]}|\hln'( t\pm)-\psp'(t)|\as 0, 
 \end{equation}
 from which we obtain that
 \[|\hln'(x+y_n)-\psp'(x+y_n)|\as 0,\]
which results in
\begin{equation}\label{inlemma: one step estimator convergence: c}
|\hln'(x+y_n)-\psp'(x)|\as 0
\end{equation}
 because $\psp'$ is continuous at  $x$, thereby proving  part C for $\hn=\widehat{g}_{\bth}$,  $\hn^{geo,sym}$, and $\hf(\bth\pm\mathord{\cdot})$.
  For  $\hn=\hts(\bth\pm x)$, part C follows similarly. Therefore, we skip the proof here.
  To prove part C for the non-logconcave estimators, $\hn^{sym}$ and $\htsm$, we have to exploit  their connection to $\hf$ and $\hf^{sm}$, respectively. 
Beginning with $\hn^{sym}$,  we observe that
\begin{equation}\label{definition of tpz by tpc}
(\hln^{sym})'(x)=\varrho_n(x)\lb(\log\hf)'(\bth+x)\rb-(1-\varrho_n(x))\lb(\log\hf)'(\bth-x)\rb
\end{equation}
where $\varrho_n(x)=\hf(\bth+x)/2\hn^{sym}(x)\leq 1$.
Thus 
\begin{align*}
\MoveEqLeft|(\hln^{sym})'(x+y_n)-\psp'(x)| \leq  \varrho_n(x+y_n)|(\log\hf)'(\bth+x+y_n)-\psp'(x)|\nonumber\\
&\ +\ (1-\varrho_n(x+y_n))|(\log\hf)'(\bth-x-y_n)-\psp'(-x)|.
\end{align*}
 Since $ \varrho_n$ is uniformly bounded above by $1$,  \eqref{inlemma: one step estimator convergence: c} (applied on $\log\hf(\bth\pm\mathord{\cdot})$) completes the proof of part C for $\hln^{sym}$. Analogously using  \eqref{definition of tpz by tpc}, for $\hn=\htsm$ , one can show that 
 \begin{equation}\label{definition of tp when we have htsm}
((\hln^{sym})^{sm})'(x)=\varrho^{sm}_n(x)\lb(\log\hf^{sm})'(\bth+x)\rb-(1-\varrho^{sm}_n(x))\lb(\log\hf^{sm})'(\bth-x)\rb
\end{equation}
for some $\varrho^{sm}_n(x)<1$. Therefore, following the same steps as in the case of $\hn^{sym}$, we can establish  part C for $\htsm$.
\end{proof}

\subsubsection*{Proof of Lemma \ref{lemma: consistency of Fisher information}}
As in the proof of Lemma~\ref{lemma: L1 convergence: one step density estimators}, one can show that it suffices to prove Lemma~\ref{lemma: consistency of Fisher information} when $\bth$ is a strongly consistent estimator of $\th_0$.

 We will show that $\hin(\eta)\as\Ig(\eta)$ and $\hth\as\th_0$ when $\hn$ equals $\hn^{sym}$, $\widehat{g}_{\bth}$ or $\hn^{geo,sym}$. The proof  in the case of $\hn=\htsm$ follows in a similar way.

We will prove the consistency of $\hin(\eta)$ first. Since Lemma~\ref{lemma: L1 convergence: one step density estimators} indicates that $||\hn-\q||_1\as 0$, the corresponding distribution functions satisfy $||\hH-\pG(\eta)||_{\infty}\as 0$, which combined with the continuity of $\pG$ entails that
 \begin{equation}\label{convergence of distribution function: 1}
 -\xi_n=\hH^{-1}(\eta)\as\pG^{-1}(\eta)=-\xi_0;
 \end{equation}
 and
 \begin{equation}\label{convergence of distribution function: 2}
 \hH^{-1}(1-\eta)\as\pG^{-1}(1-\eta)=\xi_0.
 \end{equation}
    Combined with strong consistency of $\bth$, the above leads to 
 \begin{equation}\label{convergence of inverses : one-step: 1}
 \bth+\xi_n=\bth+\hH^{-1}(1-\eta)\as\th_0+\xi_0=\pF^{-1}(1-\eta),
 \end{equation}
 and
 \begin{equation}\label{convergence of inverses : one-step: 2}
 \bth-\xi_n=\bth+\hH^{-1}(\eta)\as\th_0-\xi_0=\pF^{-1}(\eta).
 \end{equation}
 Observe that
\begin{align}\label{inlemma:consistency of Fisher information decomposition }
|\hin(\eta)-\Ig(\eta)|\leq &\ \bl\dint_{\bth-\xi_n}^{\bth+\xi_n}\hln'(x-\bth)^2d(\Fm-F_0)(x)\bl\nonumber\\
&\ +\bl\dint_{\bth-\xi_n}^{\bth+\xi_n} \hln'(x-\bth)^2 dF_0(x)-\Ig(\eta)\bl.
\end{align}
Let us consider the first term on the right hand side of \eqref{inlemma:consistency of Fisher information decomposition } first. To this end,  note that  Lemma~\ref{lemma: boundedness of tp} combined with \eqref{convergence of distribution function: 1} and \eqref{convergence of distribution function: 2} imply that 
\begin{equation}\label{inlemma: boundedness of tp}
\limsup_n\sup_{z\in[-\xi_n,\xi_n]}|\tp(z)| \leq C_{\xi_0}\quad a.s.,
\end{equation}
where $C(-\xi_0,\xi_0)>0$ is a finite constant depending on the truncation parameter $\eta$ via $\xi_0=\pG^{-1}(1-\eta)$. For the sake of simplicity, we  denote $C(-\xi_0,\xi_0)$ by $M_{\eta}$. 
Suppose $\mathcal{U}_{\eta}$ is the class of  non-increasing functions bounded by $M_{\eta}$, i.e.,
 \[\mathcal{U}_{\eta}=\bigg\{h: \RR\mapsto[-M_{\eta},M_{\eta}]\ \bl\ \ h\text{ is non-increasing}\bigg\},\]
 and denote by $\mathcal{F}_{\eta}$ the class
  \begin{align*}
 \mathcal{F}_{\eta}=\bigg\{h^2:\RR\mapsto & [0,M_{\eta}^2] \ \bl \ h(x)=u(x)1_{[r_1,r_2]}(x),\\
 &\ u\in\mathcal{U}_{\eta},\ [r_1,r_2]\subset\lb\pF^{-1}(\eta/2),\pF^{-1}(1-\eta/2)\rb\bigg\}.
 \end{align*} 
    Theorem~2.7.5 of \cite{wc} (pp. $159$) indicates that  for any $\e>0$, the bracketing entropy
   \begin{equation}\label{inlemma: finite entropy increasing}
  \log N_{[\ ]}(\e,\mathcal{U}_{\eta},L_2(P_0))\lesssim \e^{-1}.
   \end{equation}
   Also, by Example 2.5.4 of \cite{wc}, the class $\mathcal{F}_I$ of all indicator functions of  the form $1_{[z_1,z_2]}$, where $z_1\leq z_2$ with $z_1,z_2\in\RR$,  satisfies
   \begin{equation}\label{inlemma: finite entropy indicator functions}
   \log N_{[\ ]}(\e,\mathcal{F}_I,L_2(P_0))\lesssim 2/\e^2,
   \end{equation}
   implying that
\[\log N_{[\ ]}(\e,\mathcal{F}_{\eta},L_2(P_0))\lesssim \e^{-1}.\]
Since $\tp$ is non-increasing, \eqref{inlemma: boundedness of tp} implies that $(\tp)^2(\mathord{\cdot}-\bth)$ restricted to $[\bth-\xi_n,\bth+\xi_n]$ belongs to $\mathcal{F}_{\eta}$.
Therefore,  by   Glivenko-Cantelli Theorem (see Theorem~$2.4.1$ of \cite{wc}),
\begin{equation}\label{inlemma: consistency: first term}
\dint_{\bth-\xi_n}^{\bth+\xi_n}\hln'(x-\bth)^2d(\Fm-F_0)(x)\as 0.
\end{equation}
Now we claim that the second term in \eqref{inlemma:consistency of Fisher information decomposition } also approaches $0$ almost surely.
This follows by Theorem~\ref{Theorem: the L1 convergence of the density estimators of one-step estimators}(C) and  the dominated convergence theorem since $\tp$ is bounded on $[-\xi_n,\xi_n]$, completing the proof of $\hin(\eta)\as\Ig(\eta)$. 

Our next aim is to prove the consistency of $\hth$.
Observe that from \eqref{def: one-step estimator: truncated} it follows that
\begin{align*}
|\tilde{\theta}_n-\bth|
\leq&\ \hin(\eta)^{-1}\bl\dint_{\bth-\xi_n}^{\bth+\xi_n}\hln'(x-\bth)d(\Fm-F_0)(x)\bl\\
&\ +\hin(\eta)^{-1}\bl\dint_{\bth-\xi_n}^{\bth+\xi_n} \hln'(x-\bth) dF_0(x)\bl.
\end{align*}
Note that $\hin(\eta)\as\Ig(\eta)$ by the first part of the current lemma.
The proof of
\[\hin(\eta)^{-1}\bl\dint_{\bth-\xi_n}^{\bth+\xi_n}\hln'(x-\bth)d(\Fm-F_0)(x)\bl\as 0\]
is very similar to the proof of \eqref{inlemma: consistency: first term}, and follows by Glivenko-Cantelli Theorem (see Theorem~$2.4.1$ of \cite{wc}) noting that $\tp(\mathord{\cdot}-\bth)$  restricted to $[-\xi_n,\xi_n]$ is a member of $\mathcal{U}_{\eta}$, which has finite bracketing entropy by \eqref{inlemma: finite entropy indicator functions}.

Since $\tp$ is bounded  on $[-\xin,\xin]$,  \eqref{convergence of inverses : one-step: 1}, \eqref{convergence of inverses : one-step: 2}, and  Theorem~\ref{Theorem: the L1 convergence of the density estimators of one-step estimators}(C) combined with  the dominated convergence theorem  further entail that 
\begin{align*}
\dint_{\bth-\xi_n}^{\bth+\xi_n} \hln'(x-\bth) dF_0(x)\as \dint_{\pG^{-1}(\eta)}^{\pG^{-1}(1-\eta)}\psp'(x)g_0(x)dx,
\end{align*}
which is $0$ because $\psp'$ is an odd function while $g_0$ is an even function.  Therefore, strong consistency of $\tilde{\theta}_n$ follows from that of $\bth$.\hfill $\Box$
\subsubsection*{Proof of Theorem \ref{theorem: main: one-step}}
As in the proof of Lemma~\ref{lemma: L1 convergence: one step density estimators}, one can show that it suffices to prove Theorem \ref{theorem: main: one-step} when $\bth\as\th_0$. Therefore, in what follows, we assume that $\bth\as\th_0$. 
 
   First we consider the case when $\hn$ equals $\hn^{sym}$, $\widehat{g}_{\bth}$ or $\hn^{geo,sym}$, so that by Lemma~\ref{lemma: L1 convergence: one step density estimators}, the density estimator $\hn$ converges to $\q$ in $L_1$ almost surely.
From  \eqref{def: one-step estimator: truncated} we deduce that
\begin{align*}
-(\hth-\bth)= \dint_{\bth-\xin}^{\bth+\xin}\dfrac{\hln'(x-\bth)}{\hin(\eta)}d\Fm(x)=\dint_{-\xin}^{\xin}\dfrac{\hln'(z)}{\hin(\eta)}d\Fm(z+\bth).
\end{align*}
Denoting $\td=\th_0-\bth$, we observe that the expression in the above display can also be written as
\begin{align*}\label{theorem: main: one-step: first split}
\MoveEqLeft\dint_{-\xin}^{\xin}\dfrac{\tp(z)-\psp'(z-\td)}{\hin(\eta)}d(\Fm(z+\bth)-F_0(z+\bth))\nonumber\\
&\ +\dint_{-\xin}^{\xin}\dfrac{\tp(z)}{\hin(\eta)}\lb f_0(z+\bth)-g_0(z)\rb dz
 +\dint_{-\xin}^{\xin}\dfrac{\tp(z)-\psp'(z)}{\hin(\eta)}g_0(z)dz\nonumber\\
 &\ +  \dint_{-\xin}^{\xin}\dfrac{\psp'(z)}{\hin(\eta)}g_0(z)dz + \dint_{-\xin}^{\xin}\dfrac{\psp'(z-\td)}{\hin(\eta)}d(\Fm(z+\bth)-F_0(z+\bth))\nonumber\\
 =&\ T_{1n}+T_{2n}+T_{3n}+T_{4n}+T_{5n}.
\end{align*}
Observe that  $T_{3n}$ and $T_{4n}$ vanish since $\tp$ and $\psp'$ are odd functions while $\p$ is an even function. Now note that to prove \eqref{statement: one-step: main theorem: misspecified}, it suffices to show  that $\sqn T_{1n}=o_p(1)$, 
 \begin{equation}\label{intheorem: one step: T2 convergence}
  \dfrac{T_{2n}}{\td} \as \dint_{-\xi_0}^{\xi_0}\dfrac{\psp'(z)\ps_0'(z)}{\Ig(\eta)}g_0(z)dz,
  \end{equation}
 and that  
 \[\sqn T_{5n}=\int_{\th_0-\xi_0}^{\th_0+\xi_0}\dfrac{\psp'(x-\th_0)}{\Ig(\eta)}d\mathbb{Z}_n(x)+o_p(1).\]
 When $\hn=\htsm$, \eqref{statement: one-step: main theorem: misspecified: smoothef}  can also be proved in a similar way, and hence we skip the proof.

 We will first show that $\sqn T_{1n}=o_p(1)$. 
  Recall that in Section \ref{sec: Preliminaries} we denoted the empirical process $\sqn(\Fm-F_0)$ by $\mathbb{Z}_n$. Denote by $h_n$ the function
  \[h_n(x)=(\tp(x-\bth)-\php'(x))1_{[\bth-\xi_n,\bth+\xi_n]}(x),\]
  where $\php'(x)$ was previously defined as $\ps_0'(x-\th_0)$.
   It follows that
 \begin{align}\label{intheorem: T1n: representation}
 \sqn  T_{1n}\ =&\ \sqn\dint_{\bth-\xi_n}^{\bth+\xi_n}\dfrac{\tp(x-\bth)-\php'(x)}{\hin(\eta)}d(\Fm-F_0)(x)\nonumber\\
 =&\ \edint\dfrac{ h_n(x)}{\hin(\eta)}d\mathbb{Z}_n(x).
 \end{align}

Lemma~\ref{lemma: boundedness of tp} combined with \eqref{convergence of distribution function: 1} and \eqref{convergence of distribution function: 2} imply that 
\[\limsup_n\sup_{z\in[-\xi_n,\xi_n]}|\tp(z)| \leq C_{\xi_0}\quad a.s.,\]
where $C_{\xi_0}$ is a positive constant depending on $\xi_0$. 
 It is not hard to see that since $\php'$ is monotone,  $|\php'|$ attains its maxima over  $[\bth-\xi_n,\bth+\xi_n]$ at either of the endpoints. Though $[\bth-\xin,\bth+\xin]$ is random, \eqref{convergence of distribution function: 2} and \eqref{convergence of inverses : one-step: 1} indicate that with probability one, this interval is a subset of $[\pF^{-1}(\eta/2),\pF^{-1}(1-\eta/2)]$ for all sufficiently large $n$.  Hence, it follows that
 \begin{equation}\label{intheorem: boundedness of psp}
 \limsup_n\sup_{x\in [\bth-\xi_n,\bth+\xi_n]}|\psp'(x)|<M_{\eta} \quad a.s.
 \end{equation}
  for some $M_{\eta}>0$, which can be chosen in such a way so that
$C_{\xi_0}\leq M_{\eta}$. Therefore, we obtain that
\begin{equation}\label{inlemma: boundedness: hn}
\limsup_n\sup_{x\in [\bth-\xi_n,\bth+\xi_n]}|h_n(x)|<M_{\eta} \quad a.s.
\end{equation}

 
Now  define the class $\mathcal{H}_{\eta}$ by
  \begin{align*}
 \mathcal{H}_{\eta}=\bigg\{h:\RR\mapsto & [-M_{\eta},M_{\eta}] \ \bl \ h(x)=(u(x)-\php'(x))1_{[r_1,r_2]}(x),\\
 &\ u\in\mathcal{U}_{\eta},\ [r_1,r_2]\subset\lb\pF^{-1}(\eta/2),\pF^{-1}(1-\eta/2)\rb\bigg\},
 \end{align*}
 where $\mathcal{U}_{\eta}$ is the class of  non-increasing functions bounded by $M_{\eta}>0$, i.e.
 \[\mathcal{U}_{\eta}=\bigg\{h:\RR\mapsto[-M_{\eta},M_{\eta}]\ \bl\ \ h\text{ is non-increasing}\bigg\}.\] 
 Observe that \eqref{convergence of inverses : one-step: 1}, \eqref{convergence of inverses : one-step: 2}, and \eqref{inlemma: boundedness: hn} imply $h_n\in\mathcal{H}_{\eta}$ almost surely for all sufficiently large $n$. 
 
 
 From \eqref{inlemma: finite entropy increasing} and \eqref{inlemma: finite entropy indicator functions} we derive that
\[\log N_{[\ ]}(\e,\mathcal{H}_{\eta},L_2(P_0))\lesssim \dfrac{1}{\e}.\]
Therefore, we conclude that the bracketing integral 
\[\mathcal{J}_{[\ ]}(1,\mathcal{H}_{\eta},L_2(P_0))=\dint_{0}^{1}\sqrt{1+\log N_{\lbrack
\rbrack}(M_{\eta}\e,\mathcal{H}_{\eta},L_2(P_0))}d\e\lesssim \dint_{0}^{1}\e^{-1/2}d\e,\]
which is finite, thereby establishing that $\mathcal{H}_{\eta}$ is $P_0$-Donsker. Also, notice that
$|h_n|$ is bounded above by $M_{\eta}$, and Theorem \ref{Theorem: the L1 convergence of the density estimators of one-step estimators}(C) indicates that $h_n$ converges to $0$ at the continuity points of $\php'$. Since $\php$ is concave, $\php'$ is continuous almost everywhere with respect to Lebesgue measure on  $\dom(\php)$. Hence, the dominated convergence theorem yields 
\[\edint h_n(x)^2 dF_0(x)\as 0.\]
Since $\mathcal{H}_{\eta}$ is Donsker, Theorem 2.1 of \cite{epindex}  then implies that
\[\edint h_n(x)d\mathbb{Z}_n(x)=o_p(1).\]
Finally, an application of Lemma~\ref{lemma: consistency of Fisher information} leads to $\hin(\eta)\as\I(\eta)$, and thus from \eqref{intheorem: T1n: representation} we conclude that $\sqn T_{1n}=o_p(1)$.

Now we turn our attention to $T_{2n}$. Here we aim to prove \eqref{intheorem: one step: T2 convergence}. Observe that  $T_{2n}/\td$ can be written as
\begin{align}\label{intheorem:t2n:representation}
\dint_{-\xin}^{\xin}\dfrac{\tp(z)}{\hin(\eta)}\dfrac{\lb g_0(z-\td)-g_0(z)\rb}{\td} dz\ =&\ \dint_{-\xin}^{\xin}\dfrac{\tp(z)}{\hin(\eta)}\dfrac{\dint_{z}^{z-\td}g_0'(t)dt}{\td} dz\nonumber\\
= &\ \dint_{-\xin-\td}^{\xin}g_0'(t)\dfrac{\dint_{t}^{t+\td}\tp(z)dz}{\td\hin(\eta)}dt,
\end{align}
where the equality  follows by Fubini's theorem since $g_0$ is absolutely continuous. The concavity of $\hln$ implies that
$\tp$ is non-increasing. Hence, for any $t\in\iint(\dom(\psp))$, we have,
\begin{equation}\label{inlemma: T2: DCT}
\min\bigg\{\tp(t),\tp(t+\td)\bigg\}\leq \dfrac{\dint_{t}^{t+\td}\tp(z)dz}{\td}\leq \max\bigg\{\tp(t),\tp(t+\td)\bigg\}.
\end{equation}
Therefore, if $\psp'$ is continuous at $t\in\iint(\dom(\psp))$,  Theorem   
\ref{Theorem: the L1 convergence of the density estimators of one-step estimators}(C)  leads to
\[\dfrac{\dint_{t}^{t+\td}\tp(z)dz}{\td}\as \psp'(t)\quad\text{ as }n\to\infty.\]
Since $\psp$ is concave, $\psp'$ is continuous almost everywhere with respect to Lebesgue measure on $\dom(\psp))$.
Since $\tp$ is non-increasing, \eqref{inlemma: T2: DCT} further entails that 
\[\sup_{t\in[-\xin-\td,\xin]} \dfrac{\bl\dint_{t}^{t+\td}\tp(z)dz\bl}{|\td|}\leq \tp(-\xin-|\td|),\]
which can be bounded above by $\psp'(\pG^{-1}(\eta/2))$ using \eqref{convergence of distribution function: 1} and the fact that $\td\as 0$. Hence, we deduce that 
\[ |g_0'(t)|\dfrac{\bl\dint_{t}^{t+\td}\tp(z)dz\bl}{|\td|} 1_{[-\xin-\td,\xin]}(z)\leq |g_0'(t)|\lb\psp'(\pG^{-1}(\eta/2))\rb.\]
Note that the symmetry of $g_0$ about $0$ indicates that
\[\edint |g_0'(t)|dt=2\dint_{-\infty}^{0}g_0'(t)dt=2g(0),\]
which is finite because of \eqref{definition: model P1}.
Therefore, \eqref{intheorem:t2n:representation}, Lemma~\ref{lemma: consistency of Fisher information}, and the dominated convergence theorem yield
\[\dfrac{T_{2n}}{\td}\as \dint_{-\xi_0}^{\xi_0}\dfrac{\psp'(z)g_0'(z)}{\Ig(\eta)}dz,\]
 which completes the proof of \eqref{intheorem: one step: T2 convergence}  since $g_0'(z)=\ps_0'(z)g_0(z)$.

Now it only remains to show that the last term $T_{5n}$ converges weakly to a Gaussian random variable with variance $\Ig^{-1}$. Observe that  a change of variable leads to
\begin{align}\label{intheorem: T3n split}
\sqn T_{5n}=&\ \sqn\int_{\bth-\xin}^{\bth+\xin}\dfrac{\psp'(x-\th_0)}{\hin(\eta)}d(\mathbb{F}_n-F_0)(x)\nonumber\\
=&\  \int_{\th_0-\xi_0}^{\th_0+\xi_0}\dfrac{\psp'(x-\th_0)}{\hin(\eta)}d\mathbb{Z}_n(x)\nonumber\\
&\ +\edint 1_{C_n}(x)\dfrac{\psp'(x-\th_0)}{\hin(\eta)}d\mathbb{Z}_n(x),
\end{align}
where 
\[C_n=[\bth-\xin,\bth+\xin]\setminus[\th_0-\xi_0,\th_0+\xi_0].\]
From Lemma~\ref{lemma: consistency of Fisher information} and  Slutsky's theorem it follows that 
\[\int_{\th_0-\xi_0}^{\th_0+\xi_0}\dfrac{\psp'(x-\th_0)}{\hin(\eta)}d\mathbb{Z}_n(x)=\int_{\th_0-\xi_0}^{\th_0+\xi_0}\dfrac{\psp'(x-\th_0)}{\Ig(\eta)}d\mathbb{Z}_n(x)(1+o_p(1)),\]
which equals
\[\int_{\th_0-\xi_0}^{\th_0+\xi_0}\dfrac{\psp'(x-\th_0)}{\Ig(\eta)}d\mathbb{Z}_n(x)+o_p(1),\]
because  the central limit theorem and \eqref{definition: I(eta)} imply   that the first term on the above display is $O_p(1)$.

To deal with the second term, observe that the indicator function $1_{C_n}$ belongss to $\mathcal{F}_1$, the class of all indicator functions of  the form $1_{[z_1,z_2]\setminus[\th_0-\xi_0,\th_0+\xi_0]}$,  where $z_1\leq z_2$ with $z_1,z_2\in\RR$. Since the latter class is Donsker by \eqref{inlemma: finite entropy indicator functions},  Theorem 2.1 of \cite{epindex} entails that  the second term on the right hand side of \eqref{intheorem: T3n split} is of order $o_p(1)$ provided the following holds:
\[\edint 1_{C_n}(x)\dfrac{\psp'(x-\th_0)^2}{\hin(\eta)^2}dF_0(x)\as 0.\]
To this end note that since $\psp'$ is non-increasing, \eqref{convergence of inverses : one-step: 1}, \eqref{convergence of inverses : one-step: 2}, \eqref{intheorem: boundedness of psp}, and Lemma~\ref{lemma: consistency of Fisher information} entail that for all sufficiently large $n$,
\[\edint 1_{C_n}(x)\dfrac{\psp'(x-\th_0)^2}{\hin(\eta)^2}dF_0(x)\leq \dfrac{\psp'(\pG^{-1}(\eta/2))^2}{\Ig^2(\eta)}|\bth-\th_0|\lb\sup_x f_0(x)\rb\quad a.s.\] 
 Since $f_0\in\mP_1$, from the definition of $\mP_1$ in \eqref{definition: model P1} and the strong consistency of $\bth$ we obtain that the term on the right hand side of the last display
 converges to $0$ almost surely. Thus we  conclude that $\sqn T_{5n}$ is asymptotically distributed as a  centered Gaussian random variable with variance $\Ig(\eta)^{-1}$, which completes the proof.\hfill $\Box$
 
\subsection{Proofs for the MLE}
\label{sec:proofs:MLE}
We first state and prove a useful theorem which   plays a crucial role in proving Proposition \ref{Prop: existence of maximizers} and  Lemma~\ref{theorem: continuity of criterion function}. This theorem concerns the continuity of the function  $L(\mathord{\cdot}\ ;F):\RR\mapsto\RR$ given by 
\begin{equation}\label{def: L(th,F)}
L(\th;F)=\sup_{\psi\in\mathcal{SC}_0}\Psi(\th,\ps,F)
\end{equation}
where $F$ is a fixed distribution function, and $\Psi$ is as defined in \eqref{criterion function: xu samworth}.
 
\begin{theorem}\label{thm: continuity: general}
Suppose the distribution function $F$ satisfies Condition A. Then the map $\th\mapsto L(\th;F)$
is continuous on $\RR$, where the function $L(\th;F)$ was defined in \eqref{def: L(th,F)}.
\end{theorem}

\begin{proof}
Suppose $\th_k\to\th$ as $k\to\infty$.
Observe that $\Psi(\th,\ps,F)$ can also be written as
\[\Psi(\th,\ps,F)=\Psi(0,\ps,F(\mathord{\cdot}+\th)).\]
 Hence, to prove Theorem \ref{thm: continuity: general}, it suffices to show that as  $k\to\infty$,
\begin{equation*}
L(\th_k;F)=\sup_{\psi\in\mathcal{SC}_0}\Psi(0,\ps,F(\mathord{\cdot}+\th_k))\to\sup_{\psi\in\mathcal{SC}_0}\Psi(0,\ps,F(\mathord{\cdot}+\th))=L(\th,F).
\end{equation*}
Proposition $6$ of \cite{xuhigh} implies  that under Condition A, the convergence in the above display holds if the Wasserstein distance
\begin{equation}\label{inlemma: continuity: convergence of }
d_W(F(\mathord{\cdot}+\th_k),F(\mathord{\cdot}+\th))\to 0,\quad\text{ as }k\to\infty.
\end{equation}
Now by Theorem $6.9$ of \cite{villani2009} \citep[see also Theorem $7.12$ of][]{villani2003}, \eqref{inlemma: continuity: convergence of } follows  if the followings hold as $k\to\infty$:
\begin{equation}\label{inlemma: continuity: weak convgnce}
F(\mathord{\cdot}+\th_k)\to_d F(\mathord{\cdot}+\th),
\end{equation}
and
\begin{equation}\label{inlemma: convergence: moments}
\edint|x|dF(x+\th_k)\to \edint|x|dF(x+\th).
\end{equation}
 To prove \eqref{inlemma: continuity: weak convgnce}, note that for any bounded continuous function $h$,
\[\edint h(x-\th_k)dF(x)\to\edint h(x-\th)dF(x)\]
by the dominated convergence theorem since $\th_k\to\th$ as $k\to\infty$.
For proving \eqref{inlemma: convergence: moments}, first notice that $F$ has finite first moment  by Condition A. Therefore, another application of the dominated convergence  yields that as $\th_k\to\th$,
\[\edint|x|dF(x+\th_k)=\edint|x-\th_k|dF(x)\to \edint|x|dF(x+\th),\]
which, combined with \eqref {inlemma: continuity: weak convgnce}, leads to \eqref{inlemma: continuity: convergence of }, and completes the proof.

\end{proof}

\subsubsection*{Proof of Proposition \ref{Prop: existence of maximizers}}
Our first step is to show that  $L(F)$  is finite. 
From the definition of $\Psi$ in \eqref{criterion function: xu samworth}, it is not hard to see that
\[L(F)\leq \sup_{\ps\in\mathcal{C}}\lb\edint\ps(x) dF(x)-\edint e^{\ps(x)}dx\rb,\]
where $\mathcal{C}$ denotes the set of all real-valued concave functions. Theorem 2.2 of \cite{dumbreg} entails that under condition A, the term on the right hand side of the above display  is finite. Therefore, $L(F)<\infty$ follows. To show that $L(F)>-\infty$, we show that the map $x\mapsto-|x|\in\mathcal{SC}_{0}$. Therefore, \eqref{criterion function: xu samworth} and \eqref{maximum of criterion function} lead to 
\[L(F)>-\edint |x|dF(x)-\edint e^{-|x|}dx>-\infty,\]
which follows from Condition A. Therefore, we  conclude that $L(F)\in\RR$.

Now we have to show that there exist $\th^*(F)\in\RR$ and $\ps^*(F)\in\mathcal{SLC}_0$ such that
\[\Psi(\th^*(F),\ps^*(F),F)=\sup_{\th\in\RR,\ps\in\mathcal{SLC}_0}\Psi(\th,\ps,F)=\sup_{\th\in\RR}L(\th;F)=L(F).\]
 Now there exists a sequence $\{\th_k\}_{k\geq 1}$ such that $L(\th_k;F)\uparrow L(F)$ as $k\to\infty$. Suppose 
the sequence $\{\th_k\}_{k\geq 1}$ is bounded. Then we can find a  subsequence $\{\th_{k_r}\}_{r\geq 1}$  converging to some $\th'\in\RR$. Since the map $L(\th;F)$ is continuous in $\th$ by Theorem \ref{thm: continuity: general},  we  also have 
\[L(\th';F)=\lim_{r\to\infty}L(\th_{k_r};F)=L(F),\]
which implies that $\th'$ is a maximizer of $L(\th;F)$. The rest of the proof then follows from Proposition 4(iii) of \cite{xuhigh}, 
 which  states that for each $\th\in\RR$,  there exists a unique log-density ${\ps}_{\th}$, which maximizes 
 $\Psi(\th,\ps,F)$ in $\ps\in\mathcal{SC}_{0}$ when $F$ satisfies condition A. Hence, $(\th',\ps_{\th'})$ is a candidate for $(\th^*(F),\ps^*(F))$. Thus, to complete the proof, it remains to show that  $\{\th_k\}_{k\geq 1}$ is bounded. We will show that
$\th_k\to\pm\infty$ leads to  $L(\th_k;F)\to-\infty$, which contradicts the fact that $L(\th_k;F)\to L(F)\in\RR$, thus completing the proof. 
 
 Consider $\th_k\to\pm\infty$. We have already mentioned that by Proposition 4(iii) of \cite{xuhigh}, for each $\th_k$, there exists a log-density $\ps_{\th_k}\in\mathcal{SC}_0$ such that 
 \[L(\th_k;F)=\Psi(\th,\ps_{\th_k},F),\]
which equals 
  \begin{align*}
  \edint \ps_{\th_k}(x-\th)dF(x)-1\leq -\edint\log \lb 2|x-\th_k|\rb dF(x)-1,
  \end{align*}
 where the last inequality follows from Lemma~\ref{lemma: bound on phi: slc}. Now if $\th_k\to\pm\infty$, using Fatou's Lemma, we derive that
 \begin{align*}
 \limsup_{k\to\infty} L(\th_k;F)\leq -\edint\liminf_{k\to\infty}\lb\log |x-\th_k|\rb dF(x)-\log 2=-\infty,
 \end{align*}
 which leads to the desired contradiction.\hfill $\Box$ 

\subsubsection*{Proof of Lemma \ref{theorem: continuity of criterion function}}
The proof follows from \eqref{criterion function: Doss} and Theorem \ref{thm: continuity: general} noting that $\Fm$ satisfies Condition A.\hfill $\Box$

\subsubsection*{Proof of Lemma \ref{Lemma: projection for symmetric densities}}
First we will show that if $F$ satisfies \eqref{condition: symmetry} and $\ps\in\mathcal{SC}_0$, there exists $\ps_1\in\mathcal{SC}_0$
such that
\begin{equation}\label{inlemma: inequality of psi}
\Psi(\th',\ps,F)\leq \Psi(\th,\ps_1,F),
\end{equation}
leading to
\[L(\th';F)\leq L(\th;F),\]
which implies $\th^*(F)=\th$.
Note that $G=F(\mathord{\cdot}+\th)$ satisfies 
\begin{equation}\label{inlemma: condition: symmetry}
G(x)+G(-x)=1.
\end{equation}
Using the symmetry of $\ps$  about $0$ and \eqref{inlemma: condition: symmetry} in the fourth step we deduce that
\begin{align*}
\MoveEqLeft\Psi(\th',\ps,F)\\
=&\ \edint\ps(x-\th')dG(x-\th)-\edint e^{\ps(x)}dx\\
=&\ \dint_{-\infty}^{0}\ps(y+\th-\th')dG(y)+\dint_{0}^{\infty}\ps(y+\th-\th')dG(y)-\edint e^{\ps(x)}dx\\
=&\ \dint_{-\infty}^{0}\ps(y+\th-\th')dG(y)-\dint_{-\infty}^{0}\ps(-y+\th-\th')dG(-y)-\edint e^{\ps(x)}dx\\
=&\ \dint_{-\infty}^{0}\ps(y+\th-\th')dG(y)+\dint_{-\infty}^{0}\ps(y-\th+\th')dG(y)-\edint e^{\ps(x)}dx\\
=&\ 2\dint_{-\infty}^{0}\ps_1(y)dG(y)-\edint e^{\ps(x)}dx,
\end{align*}
where 
\[\ps_1(x)\ =\ 2^{-1}\lb\ps(x+\th-\th')+\ps(x-\th+\th')\rb,\quad x\in\RR.\]
Since $\ps_1\in\mathcal{SC}_0$, we obtain that
\begin{align*}
\Psi(\th',\ps,F)\ =&\ \edint\ps_1(x)dG(x)-\edint e^{\ps(x)}dx\\
= &\ \edint\ps_1(x-\th)dF(x)-\edint e^{\ps(x)}dx\\
=&\ \Psi(\th,\ps_1,F)+\edint e^{\ps_1(x)}dx-\edint e^{\ps(x)}dx.
\end{align*}
Also, by the arithmetic mean-geometric mean inequality,
\[\edint e^{\ps_1(x)}dx\leq \edint 2^{-1}\lb e^{\ps(x+\th-\th')}+e^{\ps(x-\th+\th')}\rb dx=\edint e^{\ps(x)}dx,\]
where equality holds if and only if $\ps(x+\th-\th')=\ps(x-\th+\th')$ almost everywhere with respect to the Lebesgue measure, which
only takes place when $\th'=\th$. Therefore, we have established the inequality 
in \eqref{inlemma: inequality of psi} with equality holding if and only if $\th'=\th$, which proves $\th^*(F)=\th$.
Observe that
\[\Psi(\th,\ps,F)=\edint \ps(x)dG(x)-\edint e^{\ps(x)}dx=\Psi(0,\ps,G).\]
Finally, Lemma~\ref{lemma: symmetry of projection operator} combined with \eqref{inlemma: condition: symmetry} and \eqref{connection: om and psi} entail that
\[\argmax_{\ps\in\mathcal{SC}_0}\Psi(0,\ps,G)=\argmax_{\ph\in\mathcal{SC}_0}\om(\ph,G)=\argmax_{\ph\in\mathcal{C}}\om(\ph,G),\]
 which completes the proof of Lemma~\ref{Lemma: projection for symmetric densities}.\hfill $\Box$
 \subsubsection*{Proof of Lemma \ref{Lemma: projection nw}}
 Lemma~\ref{Lemma: projection nw} follows directly from Lemma~\ref{Lemma: projection theory for a distribution function F}.\hfill $\Box$
\subsubsection*{Proof of Lemma \ref{lem: support of projection map}}
Suppose $\overline{J(F)}=\RR$. Then $\th^*(F)\in\overline{J(F)}$ by Proposition \ref{Prop: existence of maximizers}. Otherwise, since $F$ is non-decreasing, $\overline{J(F)}$  takes one of the following three forms: $[a,b]$, $[a,\infty)$, or $(-\infty,b]$, where $a,b\in\RR$.
  Suppose $\overline{J(F)}=[a,b]$.  We will show that in this case $L(\th;F)$ defined in \eqref{def: L(th,F)} is non-decreasing in $\th$ on $(-\infty,a]$, and non-increasing in $\th$ on $[b,\infty)$. Since $\th^*(F)$ equals $\argmax_{\th\in\RR}L(\th;F)$, the above implies $\th^*(F)\in[a,b]$.  
 To show that $L(\th;F)$  is non-decreasing in $\th$ on $(-\infty,a]$, we first note that for   $\th<\th'\leq a$, and $\ps\in\mathcal{SC}_0$,
  \begin{align*}
 \dint_{a}^{b} \psi(x-\th)dF(x) \leq  \dint_{a}^{b} \psi(x-\th')dF(x),
  \end{align*}
 since $\ps$ is non-increasing on $[0,\infty)$, and 
\[0\leq x-\th'<x-\th\]
 for $x\geq a$.
Therefore, from  \eqref{def: L(th,F)}, it is not hard to see that $L(\th;F)\leq L(\th';F)$. Similarly we can show that for $\th>\th'\geq b$,
 \begin{align*}
 \dint_{a}^{b} \psi(x-\th)dF(x)\leq \dint_{a}^{b} \psi(x-\th')dF(x),
  \end{align*}
since $\psi$ is non-decreasing on $(-\infty,0]$, and $x-\th<x-\th'\leq 0$ for $x\leq b$. Therefore, $L(\th;F)\leq L(\th';F)$, which completes the proof of $\th^*(F)\in[a,b]$. When $\overline{J(F)}$ is of the form $[a,\infty)$ or $(-\infty,b]$, we can prove  $\th^*(F)\in\overline{J(F)}$ in a similar way. Hence, the first part of Lemma~\ref{lem: support of projection map} is proved. The second part of  Lemma~\ref{lem: support of projection map} follows from Lemma~\ref{lemma: domain of xuhigh type MLE}. \hfill $\Box$

\subsubsection*{Proof of Theorem \ref{theorem: existence}}

First part  follows from \eqref{criterion function: Doss} and  Proposition \ref{Prop: existence of maximizers} noting that $\Fm$ satisfies Condition A. The second part follows from Lemma~\ref{lem: support of projection map}.\hfill $\Box$

\subsubsection*{Proof of Theorem \ref{theorem: almost sure convergence: general theorem of the MLE}}

Since $d_W(F_n,F)\to 0$, from Theorem $6.9$ of \cite{villani2009} or Theorem $7.12$ of \cite{villani2003}  it follows that $F_n$ converges to $F$ weakly, and
\begin{equation}\label{Convergence of moments: Transport distance}
\edint|x|dF_n(x)\to \edint|x|dF(x)<\infty.
\end{equation}
Let us denote the measure corresponding to $F_n$ by $P_n$.

For the sake of simplicity,   we use the abbreviated notations $\ps^*_n$, $\ph^*_n$, and $\th^*_n$ in place of $\ps^*(F_n)$, $\ph^*(F_n)$, and $\th^*(F_n)$ respectively.

 Our proof is motivated by the proof of Proposition $6$ of \cite{xuhigh}.
  Since $F_n$ satisfies condition A, $L(F_n)<\infty$ by Proposition~\ref{Prop: existence of maximizers}. Therefore, starting with any subsequence of $\{F_n\}_{n\geq 1}$, we can extract a further subsequence $\{F_{n_k}\}_{k\geq 1}$, such that $L(F_{n_k})\to\vartheta\in[-\infty,\infty]$, and $\thnk\to\th'\in[-\infty,\infty]$. 
  
 The proof of Theorem \ref{theorem: almost sure convergence: general theorem of the MLE} involves the following steps:
    \begin{itemize}[noitemsep,topsep=2pt]
  \item[(A)]  Show that $\vartheta>-\infty$.
  \item[(B)] Show that $\vartheta<\infty$,  and there exists $M>\vartheta+1$ such that \[\limsup_{k\uparrow\infty}\sup_{x\in\RR}\ps^*_{n_k}(x)<M.\]
  \item[(C)]  Show that $\th'\in\RR$.
  \item[(D)]
    Taking $a$ and $b$ to be the left and right endpoints of $J(F)$ respectively, from Lemma~\ref{lem: support of projection map}, we observe that  
 \begin{equation}\label{inlemma: the convergence of hthm lemma: domain of the concave map}
 \iint(\dom(\ps^*(F)))=(-d,d),
 \end{equation} 
where $d=\max(b-\th^*(F),\th^*(F)-a)$. We now define
\begin{align}\label{inlemma: the convgnce of hthm lemma: definition of support}
D_{\th'}=&\ \begin{cases}
\RR, & \text{ if }\ a=-\infty, b=\infty,\\
\dom(\ps^*(F))+\th^*(F)-\th', &\text{ if }b-\th^*(F)>\th^*(F)-a,\\
\dom(\ps^*(F))-(\th^*(F)-\th'), &\text{ if }b-\th^*(F)\leq \th^*(F)-a.
\end{cases}
\end{align}  
 We  Show that $\liminf_k\ps^*_{n_k}(x)>-\infty$ for  $x\in\iint(D_{\th'})$, leading to $\liminf_k\ps^*_{n_k}(0)>-\infty$.
  \item[(E)] Show that there exist $\alpha'$, $\beta'>0$ such that for all sufficiently large $k$,
   \[\ps^*_{n_k}(x)\leq -\alpha'|x|+\beta',\quad\text{ for all }x\in\RR.\]
  \item[(F)] Show that $\vartheta=L(F)$, $\th'=\th$, and $\ps^*_{n_k}$ converges pointwise to $\psi^*(F)$  on $\iint(\dom(\ps^*(F))$.
  \end{itemize}
The proof of $(A)$ is quite straightforward. To see this, observe that the function $\ps(x)=-|x|$ is a member of $\mathcal{SC}_0$. Therefore,
\[\vartheta= \lim_{k\to\infty}L(F_{n_k})\geq\lim_{k\to\infty}\Psi(0,\psi,F_{n_k}) \geq -\lim_{k\to\infty}\edint|x|dF_{n_k}(x)-\edint e^{-|x|}dx,\]
which equals $-\int|x|dF(x)-2 $  by \eqref{Convergence of moments: Transport distance}. Since $\int|x|dF(x)<\infty$ by Condition A, step $(A)$ follows.

The proof of $(B)$ hinges on the  proof of Theorem $2.2$ of \cite{dumbreg}. Suppose $\limsup_k \sup_{x\in\RR}\psnk(x)=\infty$. In that case we can replace $\{n_k\}$ by a further subsequence such that $\lim_k \sup_{x\in\RR}\psnk(x)=\infty$. Hence, without loss of generality, we assume that $\lim_k \sup_{x\in\RR}\psnk(x)=\infty$. Also, note that $\sup_x\ps^*_{n_k}(x)=\ps^*_{n_k}(0)=M_k$, say.

   
  Letting $D_{k,t}$  be the level set $\{\ps^*_{n_k}\geq t\}$,  for any $c>0$, we can bound $L(F_{n_k})$ in the following way:
 \begin{align}\label{proof:  thm on continuity of criterion function: bound on criterion}
 L(\Fk)=&\ \edint\ps^*_{n_k}(x-\th^*_{n_k})d\Fk(x)-1\nonumber\\
 =&\ \edint\ps^*_{n_k}(x)d\Fk(x+\thnk)-1\nonumber\\
 \leq& -c M_k \Pk\lb\RR\setminus D_{k,-cM_k}+\thnk\rb+M_k\Pk( D_{k,-cM_{k}}+\thnk)-1\nonumber\\
 =&\ -c M_k\lb 1-\Pk(D_{k,-cM_{k}}+\thnk)\rb+M_k\Pk( D_{k,-cM_{k}}+\thnk)-1\nonumber\\
 =&\ -(c+1)M_k\lb\dfrac{c}{c+1}-\Pk( D_{k,-cM_{k}}+\thnk)\rb-1.
\end{align}  
Let us denote the Lebesue measure on $\RR$ by Leb.
 We intend to show that $\text{Leb}(D_{k,-cM_k})\to 0$ as $k\to\infty$. 
 Applying  Lemma~$4.1$ of \cite{dumbreg}, 
   we obtain that as $M_k\to\infty$, 
 \begin{align}\label{inlemma: as convergence of hthm: Leb goes to 0}
 \text{Leb}(D_{k,-cM_k})\leq&\ (1+c)M_ke^{-M_k}/\dint_{0}^{(1+c)M_k}te^{-t}dt\nonumber\\
=&\ (1+c)M_ke^{-M_k}/(1+o(1))\to 0.
 \end{align}
Now we will show that  $\Pk(D_{k,-cM_k}+\thnk)\to 0$. We consider two cases, $|\th'|<\infty $ and $\th'=\pm\infty$. Let us focus on the  case $|\th'|<\infty$ first. Now the concavity of $\psnk$ indicates that sets of the form $D_{k,-cM_k}$ are intervals including $0$. Since $\text{Leb}(D_{k,-cM_k})\to 0$ and $\thnk\to\th'$ as $k\to\infty$,  for sufficiently large $k$,  the set $ D_{k,-cM_k}+\thnk$  is an interval around $\th'$ whose length shrinks to $0$ as $k$ increases. Now since $F$ is non-degenerate, one can certainly find a $x_0\neq \th'$ such that $x_0\in\iint(J(F))$. As a result,  $x_0\notin D_{k,-cM_k}+\thnk$ for all sufficiently large $k$. 
 Thus,  for all sufficiently large $k$,
 \[D_{k,-cM_k}+\thnk\in\{C\subset\RR \text{ closed and convex, }\ x_0\notin \iint(C)\}=D_{x_0},\]
say. However, since $\Fk$ weakly converges to $F$, Lemma~$2.13$ of \cite{dumbreg} implies that
\begin{equation}\label{inlemma: bound on P not}
\limsup_k\sup_{C\in D_{x_0}}P_{n_k}(C)\leq \sup_{C\in D_{x_0}}P_{0}(C).
\end{equation}
By Lemma~$2.13$ of \cite{dumbreg}, $\sup_{C\in D_{x_0}}P(C)$ is bounded away from $1$ if 
\[x_0\in\bigcap_{C\in\mathcal{D}}C,\]
where
\[\mathcal{D}=\bigg\{C\ \bl\ C\text{ is closed and convex, }P(C)=1\bigg\}.\]
Since $x_0\in\iint (J(F))$, the above holds.
 Therefore, we conclude that the term in the left hand side of \eqref{inlemma: bound on P not} is less than $c/(c+1)$ for sufficiently large $c$. Since $D_{k,-cM_k}+\thnk\in D_{x_0}$, the above implies that $P_{n_k}(D_{k,-cM_k}+\thnk)<c/(c+1)$ for sufficiently large $k$. Combined with \eqref{proof:  thm on continuity of criterion function: bound on criterion}, this result leads to
 \[ L(\Fk)\to-\infty\quad\text{ as }k\to\infty,\]
   which contradicts the fact that $\vartheta>-\infty$. Therefore, for this case, we see that $\lim_k M_k=\infty$ leads to a contradiction.
   
   At this point it is clear that for the other case also, i.e. when $\th'=\pm\infty$, it is enough to show that $\limsup_k P_{n_k}(D_{k,-cM_k}+\thnk)<c/(c+1)$ for sufficiently large $c$. We will consider the case $\th'=\infty$. The proof for $\th'=-\infty$ will follow similarly. Consider $y$ such that $1-F(y)<c/(c+1)$. Since $\text{Leb}(D_{k,-cM_k})\to 0$, we observe that as $\th^{*}_{n_k}\to\infty$, for all sufficiently large $k$, 
   \[D_{k,-cM_k}+\thnk\subset (y,\infty).\]
    Noting $\Fk\to_w F$, we obtain that
   \[\limsup_k P_{n_k}(D_{k,-cM_k}+\thnk)\leq \limsup_k(1-\Fk(y))=1-F(y)<c/(c+1),\]
   which completes the proof of $(B)$.  Therefore, we assume that $\{M_k\}_{k\geq 1}$ is bounded above by a number, say $M$, which can be chosen to be greater than $\vartheta+1$.

Our next step is  step (C), which follows from Lemma~\ref{lemma: bound on phi: slc}. To see this, observe that  Lemma~\ref{lemma: bound on phi: slc} implies
\[\phnk(x)\leq -\log |2(x-\thnk)|.\]
Suppose $\limsup_k|\thnk|=\infty$. Then Fatou's Lemma implies that
\begin{align*}
\liminf_{k}\edint e^{\phnk(x)}dx\leq &\ \liminf_{k}\edint e^{-\log |2(x-\thnk)|}dx\\
\leq &\ \edint\liminf_{k} e^{-\log |2(x-\thnk)|}dx=0.
\end{align*}
 However, the left hand side of the above display equals $1$ since $e^{\phnk}$ is a density, which leads to a contradiction. Therefore, step (C) follows.


Next we will prove  (D). For $x\in D_{\th'}$, denote  $A_{n_k}=(-x+\thnk,x+\thnk)$.  We calculate
 \begin{align*}
 L(\Fk)=&\ \edint\psnk(z)d\Fk(z+\thnk)-1\\
 =& \dint_{-x}^{x}\psnk(z)d\Fk(z+\thnk)+\dint_{-\infty}^{-x}\psnk(z)d\Fk(z+\thnk)\\
 &\ +\dint_{x}^{\infty}\psnk(z)d\Fk(z+\thnk)-1\\
 \leq &\ \psnk(0)P_{n_k}(A_{n_k})+\psnk(x)P_{n_k}(\RR\setminus A_{n_k})-1,
 \end{align*}
 which follows since $\psnk\in\mathcal{SC}_0$. Now observe that the above implies
  \begin{align}\label{inlemma: as convergence of theta: D}
 \liminf\limits_{k\to\infty}\psnk(x)\geq&\ \liminf\limits_{k\to\infty}\dfrac{L(\Fk)+1-\psnk(0)P_{n_k}(A_{n_k})}{P_{n_k}(\RR\setminus A_{n_k})}\nonumber\\
 \geq &\ -\dfrac{M-1-\vartheta}{\liminf\limits_{k\to\infty}P_{n_k}(\RR\setminus A_{n_k})}.
 \end{align}
 Note that for every $\e>0$, 
 \[\th'-\e<\thnk<\th'+\e,\quad \text{ for all sufficiently large }k.\]
 Therefore, we can write
 \begin{align*}
\MoveEqLeft \liminf\limits_{k\to\infty}P_{n_k}(\RR\setminus A_{n_k})\\
 = &\ \liminf\limits_{k\to\infty}
 \Fk(-x+\thnk)+1-\limsup\limits_{k\to\infty}\Fk\lb(x+\thnk)-\rb\\
 \geq &\ \liminf\limits_{k\to\infty}
 \Fk(-x+\th'-\e)+1-\limsup\limits_{k\to\infty}\Fk\lb(x+\th'+\e)-\rb\\
 \geq &\ F\lb(-x+\th'-\e)-\rb+1-F\lb(x+\th'+\e)-\rb.
  \end{align*}
  If $J(F)=\RR$, it is clear that the right hand side of the last display is positive. Suppose $J(F)\neq \RR$. Then at least one of $a$ and $b$, the endpoints  of $J(F)$, is finite.
  Now suppose $b-\th^*(F)>\th^*(F)-a$.
 Noting  $x\in\iint(D_{\th'})$, we argue that for small enough $\e>0$, $x+\e\in\iint(D_{\th'})$. Then from \eqref{inlemma: the convergence of hthm lemma: domain of the concave map} and \eqref{inlemma: the convgnce of hthm lemma: definition of support} it follows that
 \[x+\e<b-\th^*(F)+\th^*(F)-\th',\]
 implying that $x+\e+\th'<b$, leading to $F((x+\th'+\e)-)<1$ since $b$ is the right endpoint of $J(F)$.
Now consider the case when $b-\th^*(F)\leq \th^*(F)-a$. Analogous to the previous case,  \eqref{inlemma: the convergence of hthm lemma: domain of the concave map} and \eqref{inlemma: the convgnce of hthm lemma: definition of support}  yield
 \[x+\e<\th^*(F)-a-\th^*(F)+\th',\]
 or equivalently,
 \[a<-x+\th'-\e,\]
 implying  $F((-x+\th'-\e)-)>0$, 
which leads to
 \[\limsup\limits_{k\to\infty}\Pm(\RR\setminus A_{n_k})>0.\] 
Since in step (B), $M$ was chosen such that $M>\vartheta+1$, \eqref{inlemma: as convergence of theta: D} entails that
\[\liminf\limits_{k\to\infty}\psr(x)>-\infty\] 
which proves (D).


 To prove (E), we first observe that since $\psnk\in\mathcal{SC}_0$,  Lemma~$3$ of \cite{exist} implies that
   \[\psnk(x)\leq \psnk(0)+1-e^{\psnk(0)}|x|,\quad x\in\RR.\]
 Observe that (B) and (D), combined with the above, imply   that there exist $\alpha'>0$ and $\beta'\in\RR$ such that for all sufficiently large $k$,
 \begin{equation}\label{inlemma: existence proof: linear bound}
 \psnk(x)\leq \beta'-\alpha'|x|\quad\text{ for all }x\in\RR.
 \end{equation}

 
  Now we move on to step (F). Since we have established \eqref{inlemma: existence proof: linear bound} and showed that $\liminf\limits_{k\to\infty}\psr(x)>-\infty$ for $x\in\iint(D_{\th'})$,  by  Lemma~$4.2$ of \cite{dumbreg},
    we can replace $\{\psnk\}_{k\geq 1}$, if necessary, by a subsequence such that for a concave function $\bar{\ps}$ the following conditions are met:
  \begin{align}\label{statement of lemma: Lemma 4.2 of Regression paper(cite dumbreg)}
\begin{matrix}
\iint(D_{\th'})\subset\dom(\bar{\ps}),&\\
\bar{\ps}(x)\leq \beta'-\alpha'|x| &\quad \text{ for all }x\in\RR,k\in\NN,\\
\lim\limits_{k\to\infty,y\to x}\psnk(y)=\bar{\ps}(x)&\quad \text{ for all }x\in\iint(\dom(\bar{\ps})),\\
\limsup\limits_{k\to\infty,y\to x}\psnk(y)\leq\bar{\ps}(x)&\quad  \text{ for all }x\in\RR.
\end{matrix}
\end{align}
We will now establish that $\bar{\ps}\in\mathcal{SC}_0$. To see this, first notice that since $\psnk\in\mathcal{SC}_0$, for $x\notin \dom(\bar{\ps})$, 
\[\lim_{k\to\infty}\psnk(-x)=\lim_{k\to\infty}\psnk(x)\leq \bar{\ps}(x)=-\infty.\]
Hence, $-x\notin\iint(\dom(\bar{\ps}))$, suggesting that either $-x\notin\dom(\bar{\ps})$, or  $-x\in \dom(\bar{\ps})$ is a boundary point of $\dom(\bar{\ps})$. Therefore, $\dom(\bar{\ps})$ is of the form $(-c,c)$, $[-c,c]$, $(-c,c]$, or $[-c,c)$ for some $c>0$. For the last two cases, if we extend $\bar{\ps}$ to $[-c,c]$ so that $\bar{\ps}$ is continuous on $[-c,c]$ and $-\infty$ everywhere else, $\bar{\ps}$ still satisfies \eqref{statement of lemma: Lemma 4.2 of Regression paper(cite dumbreg)}. Therefore, without loss of generality, we assume that $\dom(\bar{\ps})$ is of the form $(-c,c)$ or $[-c,c]$.
 For $y\in(-c,c)$  we obtain that
 \[\bar{\ps}(y)=\lim_{k\to\infty}\psnk(y)=\lim_{k\to\infty}\psnk(-y)=\bar{\ps}(-y).\]
 It remains to show that when $\dom(\bar{\ps})$ is the closed interval $[-c,c]$, $\bar{\ps}(-c)$ and $\bar{\ps}(c)$ agree. 
 Since $\ps^{*}$ is concave, it is right continuous at $-c$, and left continuous at $c$. Therefore, the symmetry of $\ps^{*}$ on $(-c,c)$ implies symmetry on $[-c,c]$, which establishes that $\bar{\ps}\in\mathcal{SC}_0$.
 Now we claim that $e^{\bar{\ps}}\in\mathcal{SLC}_0$. To see this observe that since $\psnk(x)\leq \beta'-\alpha'|x|$,  application of the  dominated convergence theorem yields
 \[\dint_{\iint(\dom(\bar{\ps}))} e^{\bar{\ps}(x)}dx=\dint_{\iint(\dom(\bar{\ps}))}e^{\lim_k \psnk(x)}dx=\lim_k \dint_{\dom(\bar{\ps})}e^{\psnk(x)}dx=1.\]
 
 Since Condition B holds, note that if we can show 
 \[\vartheta=\Psi(\th',\bar{\ps},F)=L(F),\]
we have $\th'=\th^*(F)$,  $\bar{\ps}=\ps^*(F)$, and \eqref{Convergence: maximized criterion function: Fn} also follows.  Therefore, we have $\th^*_n\to\th^*(F)$, and \eqref{statement of lemma: Lemma 4.2 of Regression paper(cite dumbreg)}  leads to the pointwise convergence of $\ps^*_n$ to  $\ps^*(F)$. The pointwise convergence of $\ps^*_n$'s implies pointwise convergence of $e^{\ps^*_n}$, which, in its turn leads to the weak convergence of the corresponding distribution functions. Then the rest of the proof of Theorem \ref{thm: continuity: general} follows from  Proposition $2$ of \cite{theory}. We prove the equations in the last display by first establishing that
 \begin{equation}\label{intheorem: v and L(F) inequality 1}
 \vartheta\leq \Psi(\th',\bar{\ps},F)\leq L(F),
 \end{equation}
 and then showing $\vartheta\geq L(F)$.

Observe that since $\Fk$ converges to $F$ weakly,  by Skorohod's theorem, there exists a probability space $(\Omega,\mathcal{A},P^*)$ with random variable $X_k\sim\Fk$, and $X\sim F$ such that $X_k\as X$ as $k\to\infty$.
Denote by $E^*$ the expectation operator with respect to the probability space $(\Omega,\mathcal{A},P^*)$.  Denoting the positive random variable $\beta'-\alpha'|X_k|-\phnk(X_k)$ by $H_k$, we calculate 
\begin{align*}
\vartheta=&\ \limsup_{k\to\infty}\edint \phnk(x)d\Fk(x)-1\\
=&\ \lim_{k\to\infty}\edint (\beta'-\alpha'|x|)d\Fk(x)-\liminf_{k\to\infty}E^*[H_k]-1\\
\leq&\ \beta'-\alpha'\edint|x|dF(x)-E^*\bigg[\liminf_{k\to\infty}H_k\bigg]-1,
\end{align*}
which follows using \eqref{Convergence of moments: Transport distance} and Fatou's Lemma.
As $\thnk\to\th'\in\RR$, and $X_k\as X$, \eqref{statement of lemma: Lemma 4.2 of Regression paper(cite dumbreg)} indicates that the following holds with probability $1$:
\begin{align*}
\liminf_{k\to\infty}H_k=\alpha'-\beta'|X|-\limsup_k\psnk(X_k-\thnk)\geq \beta'-\alpha'|X|-\bar{\ps}(X-\th').
\end{align*}
Since $X\sim F$, the above leads to
\begin{align*}
\vartheta\leq &\ \beta'-\alpha'\edint|x|dF(x)-E^*\bigg[ \beta'-\alpha'|X|-\bar{\ps}(X-\th')\bigg]-1\\
=&\ \edint\bar{\ps}(x-\th')dF(x)-1\\
=&\ \Psi(\th',\bar{\ps},F)
\leq L(F),
\end{align*}
which completes the proof of \eqref{intheorem: v and L(F) inequality 1}.   Hence, it only remains to prove that $\vartheta\geq L(F)$.

 Let us denote $c_0=\ps^*(F)(0)$.  From Lemma~\ref{lemma: Lemma about the integrable bound of concave functions } it follows that for each $\e>0$, there exist $\ps^{(\e)} \in\mathcal{SC}_0$  such that
  \begin{equation}\label{intheorem: bound on ps(epsilon)}
   -|x|/\e+c_0\leq \ps^{(\e)}(x)\quad\text{ for all }x\in\RR,
  \end{equation}
  and
  \begin{equation}\label{intheorem: bound on ps(epsilon) 2}
  \ps^*(F)\leq \ps^{(\e)}\leq c_0.
  \end{equation}
   Further,  $\ps^{(\e)}$ decreases to $\ps^*(F)$ pointwise as $\e\downarrow 0$.   
  Using the random variables $X_k\sim\Fk$ and $X\sim F$ constructed in the proof of  \eqref{intheorem: v and L(F) inequality 1}, we write
 \begin{align}\label{lower bound: v}
 \vartheta=\lim_{k\to\infty} L(\Fk)\geq &\ \lim_{k\to\infty} \Psi(\th^*(F),\ps^{(\e)},\Fk)\nonumber\\
 =&\ \lim_{k\to\infty} E^*\bigg[\ps^{(\e)}(X_k-\th^*(F))\bigg]-\edint e^{\ps^{(\e)}(x)}dx.
 \end{align}
 Note that \eqref{intheorem: bound on ps(epsilon)} and \eqref{intheorem: bound on ps(epsilon) 2} indicate that $\ps^{(\e)}(x)$ is bounded above and below by $c_0$ and $ -|x|/\e+c_0$ respectively.
From \eqref{Convergence of moments: Transport distance} it follows that
 \[E^*|X_k|=\edint|x|d\Fk(x)\to\edint|x|dF(x), \quad \text{ as }k\to\infty.\]
 Moreover, as $k\to\infty$, we also have $X_k\as X$. Observe that $\ps^{(\e)}$ is concave, and by \eqref{intheorem: bound on ps(epsilon)}, $\dom(\ps^{(\e)})=\RR$, which implies that $\ps^{(\e)}$ is continuous. Therefore, 
 \[\ps^{(\e)}(X_k-\th^*(F))\as \ps^{(\e)}(X-\th^*(F)).\]
Therefore, using Lemma~\ref{lemma: bill1}, from \eqref{lower bound: v} we  obtain that
  \begin{align*}
\vartheta\geq &\  E\bigg[\lim_{k\to\infty}\ps^{(\e)}(X_k-\th^*(F))\bigg]-\edint e^{\ps^{(\e)}(x)}dx\\
 =&\ \edint\ps^{(\e)}(x-\th^*(F))dF(x)-\edint e^{\ps^{(\e)}(x)}dx.
  \end{align*}

 Since $\ps^{(\e)}$  decreases to $\ps^*(F)$ pointwise as $\e\downarrow 0$, the monotone convergence theorem and \eqref{intheorem: bound on ps(epsilon) 2} lead to
 \[\lim_{\e\downarrow 0}\edint\lb c_0-\ps^{(\e)}(x-\th^*(F))\rb dF(x)=\edint\lb c_0-\ps^*(F)(x-\th^*(F))\rb dF(x),\]
 which can be rewritten as
 \begin{equation}\label{intheorem: consistency: MCT 1}
 \lim_{\e\downarrow 0}\edint \ps^{(\e)}(x-\th^*(F)) dF(x)=\edint\ps^*(F)(x-\th^*(F))dF(x).
 \end{equation}
 On the other hand, note that \eqref{intheorem: bound on ps(epsilon) 2} implies that  $g^*(F)\leq e^{\ps^{(\e)}}\leq g^*(F)(0)$. Since $e^{\ps^{(\e)}}$ decreases to $g^*(F)$ pointwise as $\e$ decreases to $0$, the monotone convergence theorem implies that
 \[\lim_{\e\downarrow 0}\edint \lb g^*(F)(0)-e^{\ps^{(\e)}(x)}\rb dx=\edint\lb g^*(F)(0)- g^*(F)(x)\rb dx,\]
 or equivalently,
 \begin{equation}\label{intheorem: consistency: MCT 2}
 \lim_{\e\downarrow 0}\edint e^{\ps^{(\e)}(x)}dx=\edint g^*(F)(x)dx=1.
 \end{equation}
 Therefore, from \eqref{intheorem: consistency: MCT 1} and \eqref{intheorem: consistency: MCT 2}, we conclude that
   \begin{align*}
 \vartheta\geq &\ \lim_{\e\downarrow 0}\edint\ps^{(\e)}(x-\th^*(F))\ dF(x)-\lim_{\e\downarrow 0}\edint e^{\ps^{(\e)}(x)}dx\\
  =&\ \edint \ps^*(F)(x-\th^*(F))\ dF(x)-1=L(F),
 \end{align*}
  which completes the proof. \hfill $\Box$
 \subsubsection*{Proof of Theorem \ref{MLE: Rate results}}
In their proof of Theorem $3.1$, \cite{exist} show that if a sequence of log-concave functions $\{f_n\}_{n\geq 1}$ (which can be stochastic as well) satisfies
 \begin{equation}\label{mle: Hellinger consistency: likelihood is larger}
  \si\log f_n(X_i)\geq \si\log f_0(X_i) \quad a.s.,
 \end{equation}
we have $H(f_n,f_0)\as 0$,  provided 
 \begin{equation*}
 P\lb\sup_x\log f_n(x)=o\lb\dfrac{\sqn}{\log n}\rb\rb =1.
 \end{equation*}
 If we take $f_n=\hgf$, we have 
 \[\sup_x\log \hgf(x)=\hpfm(\hthm)=\hpm(0).\]
 Theorem \ref{theorem: almost sure convergence of the MLE :specific theorem}(c) and Corollary \ref{corollary: almost sure convergence of the MLE :symmetry} entail that 
 \[P\lb\limsup_n\hpm(0)<\infty\rb =1.\]
 Also, note that being the MLE of $f_0$,  $\hgf$ automatically satisfies \eqref{mle: Hellinger consistency: likelihood is larger},
 which implies
 \begin{equation}\label{inlemma: hellinger consistency of MLE of fnot}
  H(\hgf,f_0)\as 0.
  \end{equation} 
 Since $f_0\in\mP_0$ is log-concave and continuous, Proposition 2(c) of \cite{theory} and \eqref{inlemma: hellinger consistency of MLE of fnot} entail that
 \begin{equation}\label{inlemma: uniform convergence of MLE of f}
 \sup\limits_{z\in\RR}|\hgf(z)-f_0(z)|\as 0.
 \end{equation}
   Note that
 \begin{align*}
 2H^2(\hgm,g_0) =&\ \edint\lb\sqrt{\hgm(z-\hthm)}-\sqrt{g_0(z-\hthm)}\rb^2dz\\
 \leq &\ 2\edint\lb\sqrt{\hgm(z-\hthm)}-\sqrt{g_0(z-\th_0\vphantom{\hthm})}\rb^2dz\\
 &\ +2\edint\lb\sqrt{g_0(z-\hthm)}-\sqrt{g_0(z-\th_0\vphantom{\hthm})}\rb^2dz,
 \end{align*}
 where the first term on the right hand side  of the last display approaches $0$ almost surely by \eqref{inlemma: hellinger consistency of MLE of fnot}. The second term  is also bounded above by a constant multiple of $g_0(z-\hthm)+g_0(z-\th_0)$, which is integrable. Therefore, using Lemma~\ref{lemma: bill1}, Theorem \ref{theorem: almost sure convergence of the MLE :specific theorem}(a), and Corollary \ref{corollary: almost sure convergence of the MLE :symmetry}, we deduce that this term also converge to $0$ almost surely.  Hence $H^2(\hgm,g_0)\as 0$ follows. 
 
Our next step is to find  the rate of convergence of $H(\hgf, f_0)$. To do so, we first introduce the class of functions 
 \begin{align*}
 \mathcal{P}_{M,0}=\bigg\{f\in\mathcal{LC}\ \bl\ &\ \sup\limits_{x\in\RR}f(x)<M,\ f(x)>1/M\text{ for all }|x|<1,\\
 &\  \text{Supp}(f)\subset\text{Supp}(f_0)\bigg\}.
 \end{align*}
 We will show  that without loss of generality, one can assume that $f_0\in\mP_{M,0}$ for some  $M>0.$ To this end, we translate and rescale the data letting $\tilde{X}_i=\aalpha X_i+\abeta$, where $\alpha>0$ and $\abeta\in\RR$. Observe that the rescaled data has density $\tilde{f}_0(x)=\aalpha^{-1}f_0((x-\abeta)/\aalpha)$. Denote by $\tilde{f}_{0,n}$ the MLE of $f_0$ based on the rescaled data. Note that the  MLE is affine-equivalent, which  entails that $\tilde{f}_{0,n}(x)=\aalpha^{-1}\hgf((x-\abeta)/\aalpha).$ Noting Hellinger distance  is invariant under affine transformations, we observe that $H(\hgf,f_0)=H(\tilde{f}_{0,n},\tilde{f}_0).$  
   Therefore, it suffices to show that $H(\tilde{f}_{0,n},\tilde{f}_0)\as 0$. Note that since $f_0$ is log-concave,  $\iint(\dom(f_0))$ contains an interval. We can choose $\aalpha$ and $\abeta$ in a way such that  $( x-\abeta)/\aalpha$ lie inside that interval for $x=\pm 1$. Then it is possible to find $M>0$ large enough such that
 \[f_0((x-\abeta)/\aalpha)>\aalpha/M,\quad x=\pm 1,\]
 or
 \[\min (\tilde{f}_0(-1),\tilde{f}_0(1))>1/M,\]
leading to
 \[\inf_{x\in[-1,1]}\tilde{f}_0(x)>1/M,\]
 since $f_0,$ or equivalently $\tilde{f}_0$ is unimodal. Hence, without loss of generality, we can assume that there exists $M>0$ such that $f_0(x)>1/M$ for $x\in[-1,1].$ 
 We can choose $M$ large enough such that additionally, $\sup\limits_{x\in\RR} f_0(x)<M$.   On the other hand,  \eqref{inlemma: uniform convergence of MLE of f} implies that the following inequalities hold with probability one: 
  \[\limsup_n\sup\limits_{x\in\RR} \hgf(x)<M,\]
  and
   $$\lim_n\hgf(\pm 1)>1/M.$$ 
    Therefore, $f_0\in\mathcal{P}_{M,0}$, and with probability one,  $\hgf\in\mathcal{P}_{M,0}$ as well for all sufficiently large $n$. \cite{dossglobal} obtained the bracketing entropy of the class $\mathcal{P}_{M,0}$. They showed that for any $\e>0$,
 \[\log N_{[\ ]}(\e,\mathcal{P}_{M,0},H)\lesssim \e^{-1/2}.\] 
The rest of the proof now follows from an application of Theorem $3.4.1$ and $3.4.4$ of \cite{wc}.
To this end, fixing $\e>0,$ for $f\in \mathcal{P}_{M,0}$  we define the function
 \[m_f(x)=\log \lb\dfrac{f(x)+f_0(x)}{2 f_0(x)}\rb,\]
 and we let $\mathcal{M}_{\e}$ denote the class
  \[\mathcal{M}_{\e}=\{m_f-m_{f_0}\ :\ H(f,f_0)\leq \e\}.\]
 Also, we denote
  \[||\mathbb{Z}_n||_{\mathcal{M}_{\e}}=\sup_{h\in\mathcal{M}_{\e}}\int h d\mathbb{Z}_n.\]
Then  Theorem $3.4.4$ of \cite{wc} yields 
\[P (m_f-m_{f_0})\lesssim -H^2(f,{f_0}),\]
 and
\[E_{P}||\mathbb{Z}_n||_{\mathcal{M}_{\e}}\lesssim \tilde{J}_{[\ ]}(\e,\mathcal{P}_{M,0},H)\lb 1+\dfrac{\tilde{J}_{[\ ]}(\e,\mathcal{P}_{M,0},H)}{\e^2 \sqn}\rb,\]
where
\begin{align*}
\tilde{J}_{[\ ]}(\e,\mathcal{P}_{M,0},H) =&\ \dint_{0}^{\e}\sqrt{\log N_{[\ ]}(t,\mathcal{P}_{M,0},H)}dt
\lesssim  \e^{3/4}.
\end{align*} 
Therefore, it is easy to see that
\[E_{P}||\mathbb{Z}_n||_{\mathcal{M}_{\e}}\lesssim \e^{3/4}+\e^{-1/2}n^{-1/2}=g(\e),\] 
where
  $$g(t)=t^{3/4}+t^{-1/2}n^{-1/2},\quad t>0.$$
  Notice that $g(t)/t^{3/4}$ is decreasing in $t.$  Also, for any constant $c>0,$ we have
  \[n^{4/5}g(cn^{-2/5})= (c^{3/4}+c^{-1/2})n^{1/2}.\] 
 Therefore, an application of Theorem $3.4.1$ of \cite{wc} yields
  \[H(\hgf,{f_0})=O_p(n^{-2/5}),\]
  which completes the proof of part A of Theorem \ref{MLE: Rate results}.
  
   Now we turn to the proof of part B. If $x-\th_0$ is a continuity point of $g_0'$, noting $\hthm\as\th_0$ by Theorem \ref{theorem: almost sure convergence of the MLE :specific theorem} and Corollary \ref{corollary: almost sure convergence of the MLE :symmetry}, we obtain that 
  \[\dfrac{\sqrt{g_0(x-\hthm)}-\sqrt{\vphantom{\hthm}g_0(x-\th_0)}}{(\hthm-\th_0)}\as \dfrac{g_0'(x-\th_0)}{2\sqrt{g_0(x-\th_0)}}.\]
 Noting $g_0'$ is continuous almost everywhere with respect to Lebesgue measure, and using  Fatou's lemma and part A of the current theorem, we obtain that
  \begin{align}\label{inlemma: inequality: hellinger distance and distance between theta}
  &\liminf_n \dfrac{\edint \lb\sqrt{g_0(x-\hthm)}-\sqrt{\vphantom{\hthm} g_0(x-\th_0)}\rb^2dx}{(\hthm-\th_0)^2}\nonumber\\
  \geq &\ \edint \lb\dfrac{g_0'(x-\th_0)}{2\sqrt{g_0(x-\th_0)}}\rb^2dx=\dfrac{\I}{4}\ a.s.
  \end{align} 
  Now observe that
  \begin{align*}
\MoveEqLeft 2H(\hgf,f_0)^2\\
 = &\ \edint\lb\sqrt{\hgm(x-\hthm)}-\sqrt{\vphantom{\hgm}g_0(x-\th_0)}\rb^2dx\\
 =&\ \edint\lb\sqrt{\hgm(x-\hthm)}-\sqrt{\vphantom{\hgm}g_0(x-\hthm)}\rb^2dx\\
 &\ +
 \edint\lb\sqrt{g_0(x-\hthm)}-\sqrt{\vphantom{\hthm}g_0(x-\th_0)}\rb^2dx+T_c,
   \end{align*}
   where
   \[T_c=2\edint\lb\sqrt{\hgm(x-\hthm)}-\sqrt{g_0(x-\hthm)}\rb\lb\sqrt{g_0(x-\hthm)}-\sqrt{\vphantom{\hgm}g_0(x-\th_0)}\rb dx.\]
      The inequality in \eqref{inlemma: inequality: hellinger distance and distance between theta} entails that  for all sufficiently large $n$,
   \begin{equation}\label{inlemma: lower bound on hellinger dist between f and hgf}
  2H(\hgf,f_0)^2\nonumber\geq  2H(\hgm,g_0)^2+\dfrac{(\hthm-\th_0)^2\I}{4}-|T_c|\quad a.s.
   \end{equation}
 We aim to show that the cross-term $|T_c|$ is small. In fact, we show that 
 \begin{equation}\label{claim: cross term small for rough rate calculations}
 \dfrac{|T_c|}{|\hthm-\th_0|^2+H(\hgm,g_0)^2}=o_p(1).
 \end{equation}
 Suppose \eqref{claim: cross term small for rough rate calculations} holds. Then  it follows that 
 \begin{align*}
 2H(\hgf,f_0)^2 \geq &\ 2H(\hgm,g_0)^2\\
 &\ +\dfrac{(\hthm-\th_0)^2\I}{4}-o_p(1)H(\hgm,g_0)^2-o_p(1)(\hthm-\th_0)^2,
 \end{align*}
  which completes the proof because $\I>0$. 
  
 Hence, it remains to prove \eqref{claim: cross term small for rough rate calculations}. To this end, notice that $T_c$ can be written as
  \begin{align*}
  T_c= 2\edint\lb\sqrt{\hgm(x)}-\sqrt{\vphantom{\hgm}g_0(x)}\rb\lb\sqrt{\vphantom{\hgm}g_0(x)}-\sqrt{\vphantom{\hgm}g_0(x+\hthm-\th_0)}\rb dx.
  \end{align*}
  For the sake of simplicity, we will denote  $\th_0-\hthm$ by $\d$ from now on.
  Since $g_0$ is absolutely continuous, we can write
  \begin{align*}
 |T_c|= \MoveEqLeft \bl 2\edint\lb\sqrt{\hgm(x)}-\sqrt{\vphantom{\hgm}g_0(x)}\rb\lb\dint_{-\d}^{0}\dfrac{g_0'(x+t)}{2\sqrt{g_0(x+t)}}dt\rb dx\bl.
  \end{align*}
  Since $g_0\in\mathcal{S}_0$, we have
  \begin{align*}
 |T_c|=&\ 2\bl \dint_{0}^{\infty} \lb\sqrt{\hgm(x)}-\sqrt{\vphantom{\hgm}g_0(x)}\rb\lb\dint_{-\d}^{0}\dfrac{g_0'(x+t)}{2\sqrt{g_0(x+t)}}dt\rb dx\\
&\ -\ \dint_{-\infty}^{0} \lb\sqrt{\hgm(-x)}-\sqrt{\vphantom{\hgm}g_0(-x)}\rb\lb\dint_{-\d}^{0}\dfrac{g_0'(-x-t)}{2\sqrt{g_0(-x-t)}}dt\rb dx\bl\\
=&\ 2\bl \dint_{0}^{\infty} \lb\sqrt{\hgm(x)}-\sqrt{\vphantom{\hgm}g_0(x)}\rb\lb\dint_{-\d}^{0}\dfrac{g_0'(x+t)}{2\sqrt{g_0(x+t)}}dt\rb dx\\
&\ -\dint_{0}^{\infty} \lb\sqrt{\hgm(x)}-\sqrt{\vphantom{\hgm}g_0(x)}\rb\lb\dint_{-\d}^{0}\dfrac{g_0'(x-t)}{2\sqrt{g_0(x-t)}}dt\rb dx\bl,
  \end{align*}
  yielding
  \begin{equation}
  |T_c|=2\bl \dint_{0}^{\infty} \lb\sqrt{\hgm(x)}-\sqrt{\vphantom{\hgm}g_0(x)}\rb\lb \dint_{-\d}^{0}\lb\dfrac{g_0'(x+t)}{2\sqrt{g_0(x+t)}}-\dfrac{g_0'(x-t)}{2\sqrt{g_0(x-t)}}\rb dt\rb dx\bl.
  \end{equation}
  Using the Cauchy-Schwarz inequality, we obtain that
  \begin{align*}
\dfrac{|T_c|}{2|\d|}
\leq &\ \lb\dint_{0}^{\infty} \lb\sqrt{\hgm(x)}-\sqrt{\vphantom{\hgm}g_0(x)}\rb^2dx\rb^{1/2}\\
&\ \lb\dint_{0}^{\infty}\lb\dint_{-\d}^{0}\dfrac{1}{|\d|}\lb\dfrac{g_0'(x+t)}{2\sqrt{g_0(x+t)}}-\dfrac{g_0'(x-t)}{2\sqrt{g_0(x-t)}}\rb dt\rb^2 dx\rb^{1/2}.
  \end{align*}
  Since $\sqrt{\hgm(x)}-\sqrt{\vphantom{\hgm}g_0(x)}$ is an even function, the first term on the right hand side of the last inequality is $\sqrt{2}H(\hgm,g_0).$ Hence,
  \begin{align*}
  \dfrac{T_c^2}{8H(\hgm,g_0)^2\d^2}\leq &\ \dint_{0}^{\infty}\lb\dint_{-\d}^{0}\dfrac{1}{|\d|}\lb\dfrac{g_0'(x+t)}{2\sqrt{g_0(x+t)}}-\dfrac{g_0'(x-t)}{2\sqrt{g_0(x-t)}}\rb dt\rb^2 dx,
  \end{align*}
  which, noting that
  \[t\mapsto \dfrac{g_0'(x+t)}{2\sqrt{g_0(x+t)}}-\dfrac{g_0'(x-t)}{2\sqrt{g_0(x-t)}}\]
  is an even function for each $x>0$, can be bounded above by  
 \[  \dint_{0}^{\infty}|\d|\lb\dint_{0}^{|\d|}\dfrac{1}{(\d)^2}\lb\dfrac{g_0'(x+t)}{2\sqrt{g_0(x+t)}}-\dfrac{g_0'(x-t)}{2\sqrt{g_0(x-t)}}\rb ^2 dt\rb dx\]
  using the Cauchy-Schwarz inequality. Therefore, we obtain 
  \begin{align}\label{inlemma: rough rate: intermediate limsup}
   \dfrac{T_c^2}{2H(\hgm,g_0)^2\d^2}\leq&\ \dfrac{1}{|\d|}\dint_{0}^{|\d|}\bigg[ \dint_{0}^{\infty}\lb\dfrac{g_0'(x+t)}{\sqrt{g_0(x+t)}}\rb^2dx\nonumber+\dint_{0}^{\infty}\lb\dfrac{g_0'(x-t)}{\sqrt{g_0(x-t)}}\rb^2dx\\
      &\ -2\dint_{0}^{\infty}\dfrac{g_0'(x-t)}{\sqrt{g_0(x-t)}}\dfrac{g_0'(x+t)}{\sqrt{g_0(x+t)}} dx\bigg ]dt.
  \end{align}
  For $t\geq 0,$ 
   \begin{align*}
  \dint_{0}^{\infty}\lb\dfrac{g_0'(x+t)}{\sqrt{g_0(x+t)}}\rb^2dx
 =  \dint_{t}^{\infty}\lb\dfrac{g_0'(x)}{\sqrt{g_0(x)}}\rb^2dx\leq \dfrac{\I}{2}.
  \end{align*}
  Now observe that for  $z\in(-|\d|,0),$
  \begin{equation}\label{intheorem: bound of ps'}
  |g_0'(z)/\sqrt{g_0(z)}|=|\ps_0'(z)|\sqrt{g_0(z)}\leq |\ps_0'(\d)|\sqrt{g_0(0)}=O_p(1),
  \end{equation}
  since $\ps_0\in\mathcal{SC}_0$, and $\d\as 0$. 
    Hence, for $t\in(0,|\d|)$,
  \begin{align*}
  \dint_{0}^{\infty}\lb\dfrac{g_0'(x-t)}{\sqrt{g_0(x-t)}}\rb^2dx \ =&\ \dint_{-t}^{\infty}\lb\dfrac{g_0'(z)}{\sqrt{g_0(z)}}\rb^2dz \\
 =&\ \dint_{-t}^{0}\lb\dfrac{g_0'(x)}{\sqrt{g_0(z)}}\rb^2dz+\dint_{0}^{\infty}\lb\dfrac{g_0'(z)}{\sqrt{g_0(z)}}\rb^2dz\\
 \leq &\ |\d| \ps_0'(\d)^2\sqrt{g_0(0)}+\I/2\\
 =&\ |\d| O_p(1)+\I/2,
  \end{align*}
  where the last step follows from \eqref{intheorem: bound of ps'}.
  Hence, for any $t\in(0,|\d|),$
  \begin{equation}\label{inlemma: rough rate: sum of sums}
   \dint_{0}^{\infty}\lb\dfrac{g_0'(x+t)}{\sqrt{g_0(x+t)}}\rb^2dx+\dint_{0}^{\infty}\lb\dfrac{g_0'(x-t)}{\sqrt{g_0(x-t)}}\rb^2dx
  = |\d| O_p(1)+\I.
  \end{equation}
 
  Our objective is to apply Fatou's lemma on the  third term on the right hand side of \eqref{inlemma: rough rate: intermediate limsup}. Therefore,  we want to ensure that the integrand is non-negative.
   Note that when $x\geq|\d|$ and $t\in(0, |\d|)$, we have $x>t$, which leads to
  \begin{equation}\label{intheorem: positivity of cross term}
  g_0'(x-t)g_0'(x+t)\geq 0.
  \end{equation}
   Keeping that in mind, we partition the term 
  \begin{align*}
  \MoveEqLeft - \dint_{0}^{\infty}\dfrac{g_0'(x-t)}{\sqrt{g_0(x-t)}}\dfrac{g_0'(x+t)}{\sqrt{g_0(x+t)}} dx\\
  =&\ - \dint_{|\d|}^{\infty}\dfrac{g_0'(x-t)}{\sqrt{g_0(x-t)}}\dfrac{g_0'(x+t)}{\sqrt{g_0(x+t)}} dx -\dint_{0}^{|\d|}\dfrac{g_0'(x-t)}{\sqrt{g_0(x-t)}}\dfrac{g_0'(x+t)}{\sqrt{g_0(x+t)}} dx\\
  \leq &\ - \dint_{|\d|}^{\infty}\dfrac{g_0'(x-t)}{\sqrt{g_0(x-t)}}\dfrac{g_0'(x+t)}{\sqrt{g_0(x+t)}} dx+|\d| O_p(1),
  \end{align*}
  where the last step follows from \eqref{intheorem: bound of ps'}.
The above combined with \eqref{inlemma: rough rate: intermediate limsup} and \eqref{inlemma: rough rate: sum of sums} leads to
  \begin{align}\label{intheorem: the bound on the ratio}
 \MoveEqLeft \limsup_n  \dfrac{T_c^2}{2H(\hgm,g_0)^2\d^2}\nonumber\\
  \leq &\ \limsup_n\dfrac{1}{|\d|}\dint_{0}^{|\d|}\bigg[ |\d| O_p(1)+\I- 2\dint_{|\d|}^{\infty}\dfrac{g_0'(x-t)}{\sqrt{g_0(x-t)}}\dfrac{g_0'(x+t)}{\sqrt{g_0(x+t)}} dx\bigg]dt\nonumber\\
 =  &\ O_p(1)\limsup_n|\d|+\I\nonumber\\
 &\ -2\liminf_n\dfrac{1}{|\d|}\dint_{0}^{|\d|}\dint_{|\d|}^{\infty}\dfrac{g_0'(x-t)}{\sqrt{g_0(x-t)}}\dfrac{g_0'(x+t)}{\sqrt{g_0(x+t)}} dxdt\nonumber\\
 = &\ 0+\I -2\liminf_n \dint_{|\d|}^{\infty}\dfrac{\dint_{0}^{|\d|}\dfrac{g_0'(x+t)}{\sqrt{g_0(x+t)}}\dfrac{g_0'(x-t)}{\sqrt{g_0(x-t)}}dt}{|\d|}dx.
   \end{align}
  Therefore, an application of Fatou's Lemma and \eqref{intheorem: positivity of cross term} yield
  \[\liminf_n \dint_{|\d|}^{\infty}\dfrac{\dint_{0}^{|\d|}\dfrac{g_0'(x+t)}{\sqrt{g_0(x+t)}}\dfrac{g_0'(x-t)}{\sqrt{g_0(x-t)}}dt}{|\d|}dx\geq \dint_{0}^{\infty}\dfrac{g_0'(x)^2}{g_0(x)}dx=\dfrac{\I}{2}.\]
  Thus \eqref{intheorem: the bound on the ratio} leads to
  \[    \dfrac{2T_c^2}{4H(\hgm,g_0)^2\d^2}=o_p(1).\]
  from which it is obvious that
  \[  \dfrac{\sqrt{2}|T_c|}{|\d|^2+H(\hgm,g_0)^2}\leq  \dfrac{\sqrt{2}|T_c|}{2H(\hgm,g_0)|\d|}=o_p(1),\]which proves \eqref{claim: cross term small for rough rate calculations} and thus completes the proof of part B of Theorem \ref{MLE: Rate results}.                                \hfill $\Box$
 
   \subsubsection{Proof of Lemma \ref{lemma: MLE: convg-lemma}}

From part A of Theorem \ref{MLE: Rate results}, we obtain that $\hgf$ is strongly Hellinger consistent for $f_0$.
 As a consequence, $||\hgf-f_0||_1\to_p 0$. Therefore, from proposition 2(c) of \cite{theory}, it follows that, 
  \begin{equation}\label{inlemma: convergence of density}
 \sup\limits_{z\in\RR}|\hgf(x)-f_0(x)|\as 0.
 \end{equation}
Now consider any compact set $K=[b_1,b_2]\subset\dom(\ph_0).$ Observe that
 \begin{align*}
\MoveEqLeft\sup_{x\in K}|\hpfm(x)-\ph_0(x)|\\
= &\ \sup_{x\in K}\log\dfrac{\max(\hgf(x),f_0(x))}{\min(\hgf(x),f_0(x))}\\
 = &\ \sup_{x\in K}\log\lb\dfrac{|\hgf(x)-f_0(x)|}{\min(\hgf(x),f_0(x))}+1\rb\\
 &\leq \log\lb\dfrac{\sup_{x\in K}|\hgf(x)-f_0(x)|}{\min(\hgf(b_1),\hgf(b_2),f_0(b_1),f_0(b_2))}+1\rb,
 \end{align*}
since $\hgf$ and $f_0$ are unimodal. From \eqref{inlemma: convergence of density} it follows that $\sup_{x\in K}|\hgf(x)-f_0(x)|\as 0$. Also, since $b_1\in\dom(\ph_0)$, $\hgf(b_1)\as f_0(b_1)>0$. The same holds for $\hgf(b_2)$, which establishes that
 \[\sup_{x\in K}|\hpfm(x)-\ph_0(x)|\as 0.\] 
 Now from Lemma~$F.10$ of \cite{bodhida2017}, it follows that if $\ph_0$ is differentiable at $x\in K$, leading to
\[\hpfm'(x)\as\ph_0'(x),\] 
which combined with \eqref{inlemma: convergence of density} yields
   \[\hgf'(x)\as f_0'(x).\] 
   Part A of Theorem \ref{MLE: Rate results} also indicates that  $H(\hgm,g_0)\as 0$, from which, proceeding like above, one can  show that similar results hold for $\hgm$ and $\hpm$ as well, which completes the proof.
\hfill $\Box$

\subsection{Additional lemmas}

\subsection{Additional lemmas for the one-step estimators}

\begin{lemma}\label{lemma: boundedness of tp} 
Consider $\zeta,\zeta'\in\iint(\dom(\psp))$ such that $\zeta\leq\zeta'$.  Suppose   $\zeta_n\as \zeta$, and $\zeta'_n\as \zeta'$.   Letting $\tp$ denote $\log\hn$, for $\hn=\hn^{sym}$, $\widehat{g}_{\bth}$, $\hn^{geo,sym}$, or $\hf(\bth\pm\mathord{\cdot})$, under the conditions of Theorem \ref{Theorem: the L1 convergence of the density estimators of one-step estimators}, we have 
\[\limsup_n\sup_{z\in[\zeta_n,\zeta_n']}|\tp(z)|\leq C_{\zeta,\zeta'}\quad a.s.,\]
where $C_{\zeta,\zeta'}$ is a constant depending on $\zeta$ and $\zeta'$.
When $\hn=\htsm$ or $\hts(\bth\pm \mathord{\cdot})$, the above conclusion holds for any $\zeta,\zeta'\in\RR$.
\end{lemma}

\begin{proof}
When $\hn$ equals the densities $\widehat{g}_{\bth}$, $\hg^{geo,sym}$, $\hf(\bth\pm\mathord{\cdot})$, or $\hts(\mathord{\cdot}+\bth)$, it is clear that $\hln$ is concave.  We will consider these  choices of $\hn$  first.   For the first three cases, the pointwise limit of $\hn$ is $\q$ by Theorem~\ref{Theorem: the L1 convergence of the density estimators of one-step estimators}. Note that we can choose  $\zeta_1,\zeta_1'\in\dom(\psp)$ such that  $\psp'$ is continuous at $\zeta_1$ and $\zeta'_1$, and  $[\zeta-\e,\zeta'+\e]\subset[\zeta_1,\zeta_1']$ for some sufficiently  small $\e>0$. 
 Now the conditions on $\zeta_n$ and $\zeta_n'$ underscore that for sufficiently large $n$,
 \[[\zeta_n,\zeta_n']\subset[\zeta_1,\zeta_1']\quad a.s. \]
   Now note that a concave $\hln$ leads to a non-increasing $\tp$. Therefore, on any interval, $|\tp|$ attains its maxima on either of the endpoints. As a consequence, the supremum of $|\tp|$ on the  interval $[\zeta_1,\zeta_1']$ does not exceed  $|\tp(\zeta_1)|+|\tp(\zeta'_1)|$,  which converges almost surely to $|\psp'(\zeta_1)|+|\psp'(\zeta'_1)|$ by Theorem~\ref{Theorem: the L1 convergence of the density estimators of one-step estimators}. Therefore, we conclude that 
   \begin{equation}\label{inlemma: additional: one-step: hn first four cases}
\limsup_n \sup_{z\in[\zeta_n,\zeta_n']}|\tp(z)|< |\psp'(\zeta_1)|+|\psp'(\zeta'_1)|\quad a.s.,
   \end{equation}
   when $\hn=\widehat{g}_{\bth}$, $\hg^{geo,sym}$, or $\hf(\mathord{\cdot}+\bth)$. Now consider  $\hn=\hts(\bth\pm\mathord{\cdot})$. Note that the $L_1$ limit of this $\hn$ is $\qsm$, and $\dom(\pspm)=\RR$. Therefore, for any $\zeta,\zeta'\in\RR$, the same arguments as the previous case will lead to
   \begin{equation}\label{inlemma: additional: one-step: smoothed case}
\limsup_n \sup_{z\in[\zeta_n,\zeta_n']}|\tp(z)|< |(\pspm)'(\zeta_1)|+|(\pspm)'(\zeta'_1)|\quad a.s.
   \end{equation}


 Now let us  consider the non-log-concave cases. Assume $\hn=\hn^{sym}$. Then from \eqref{definition of tpz by tpc} we obtain that
 \[\tp(x)=\varrho_n(x)\lb(\log\hf)'(\bth+x)\rb-(1-\varrho_n(x))\lb(\log\hf)'(\bth-x)\rb,\]
 where $\hf$ is the log-concave MLE and $\varrho_n(x)<1$,   leading to
 \[\limsup_n\sup_{x\in [\zeta_n,\zeta_n']}|\tp(x)|\leq \sup_{x\in [\zeta_n,\zeta_n']}|(\log\hf)'(\bth+x)|+\sup_{x\in [\zeta_n,\zeta_n']}|(\log\hf)'(\bth-x)|,\]
which is bounded above by \eqref{inlemma: additional: one-step: hn first four cases}.
Finally, when $\hn=\htsm$, from  \eqref{definition of tp when we have htsm} we observe  that $\tp$ can be expressed in terms of  $\hts$  as
\[\tp(x)=\varrho^{sm}_n(x)\lb(\log\hts)'(\bth+x)\rb-(1-\varrho^{sm}_n(x))\lb(\log\hts)'(\bth-x)\rb,\] 
  for some $\varrho^{sm}_n(x)<1$. Boundedness of $\tp$ then follows from \eqref{inlemma: additional: one-step: smoothed case}. 
\end{proof}
\subsection{Additional lemmas for the MLE}

\begin{lemma}\label{lemma: bound on phi: slc}
Suppose $e^{\ps}\in\mathcal{SLC}_0$. Then $\ps(x)\leq -\log |2x|$ for all $x\in\RR$.
\end{lemma}
\begin{proof}
Note that since $e^{\ps}\in\mathcal{SLC}_0$,  for $x>0$, we have,
\[1=\edint e^{\ps(z)}dz\geq \dint_{-x}^{x} e^{\ps(z)}dz,\]
which is bounded below by $2e^{\ps(x)}x$ because $\ps\in\mathcal{SC}_0$. Hence the result follows.
\end{proof}

\begin{lemma}\label{lemma: domain of xuhigh type MLE}
Suppose $F$ is a distribution function satisfying Condition A. Denote 
\[\iint(J(F))=(a,b),\]
for some $a,b\in[-\infty,\infty]$.
Suppose $\ps_{\th}$ is the maximizer of $\Psi(\th,\ps,F)$ over $\ps\in\mathcal{SC}_{0}$ for some $\th\in\RR$. Then 
\[\iint(\dom(\ps_{\th}))=(-d,d),\] 
where   $d=(b-\th)\vee (\th-a).$
\end{lemma}
Note that $a$ or $b$ can be $\infty$, in which case $d=\infty$, and $\iint(\dom(\ps_{\th}))=\RR$.
\begin{proof}
Note that $\iint(\dom(\ps_{\th}))=(-c,c)$ for some $c>0$. We first show that $c\geq d$.   Consider the set $A=\overline{(-d,d)}\setminus \dom(\ps_{\th})$.
Note that unless $d=\infty$, $\overline{(-d,d)}=[-d,d]$. Letting $\Delta$ denote the indicator function $1_{A}$, we observe that $\ps_{\th}+\Delta=\ps_{\th}$. Also, note that 
\[\ps_{\th}=\argmax_{\ps\in\mathcal{SC}_0}\Psi(0,\ps,F(\mathord{\cdot}+\th))=\argmax_{\ps\in\mathcal{SC}_0}\om(\ps,F(\mathord{\cdot}+\th)),\]
which follows from \eqref{criterion function: xu samworth} and \eqref{definition: new criterion function }.
 Hence, Proposition $5$(ii) of \cite{xuhigh} leads to
\[\dint_{-\infty}^{\infty}\Delta(x)dF(x+\th)\leq \dint_{-\infty}^{\infty}\Delta(x)e^{\ps_{\th}(x)}dx.\]
Since $\Delta=0$ on $\dom(\ps_{\th})$, it follows that
\[\dint_{A}dF(x+\th)= 0,\]
implying $(a,b)\cap (A+\th)=\emptyset$. 

   It is easy to see that if either $a$ or $b$ equals $\infty$, $d=\infty$, and  $A+\th=\RR \setminus\dom(\ps_{\th})$. Therefore,
\begin{equation*}
(a,b)\cap (A+\th)=(a,b)\setminus\lb \dom(\psi_{\th})+\th\rb=\emptyset,
\end{equation*}
which indicates that
\begin{equation}\label{inlemma:domains}
(a,b)\subset(-c,c)+\th,
\end{equation}
implying $(-c,c)=\RR$. Hence $(-c,c)=(-d,d)$, which completes the proof for this case. Therefore, it only remains  to prove Lemma~\ref{lemma: domain of xuhigh type MLE} for the case when both $a$ and $b$ are finite.

When $a,b<\infty$, there can be two sub-cases. In the first case, $b-\th\geq \th-a$, which indicates $d=b-\th$. In this case, $[-d,d]+\th=[2\th-b,b]$, and 
\[[a,b]\subset [2\th-b,b]=[-d,d]+\th,\]
leading to
\[(a,b)\setminus\lb \dom(\psi_{\th})+\th\rb=(a,b)\cap (A+\th)=\emptyset,\]
from which \eqref{inlemma:domains} follows. 
Therefore, $b\leq c+\th$, or $d\leq c$. In the second case $b-\th\leq \th-a$, $d=\th-a$ and   $[-d,d]+\th=[a,2\th-a]$. Proceeding similarly as in the first case,   we can again show that $c\geq d$. Hence, we have established that $c\geq d$ in general.

Now suppose $c>d$, which implies that $a,b<\infty$. Observe that in this case, there exists a $\ps\in\mathcal{SC}_0$, which equals $\ps_{\th}$ on $(-d,d)$ and $-\infty$ everywhere else. It is easy to verify that when  $c> d$,
\begin{align*}
\MoveEqLeft  \Psi(\th,\ps_{\th},F)-\Psi(\th,\ps,F)\\
=&\ \dint_{-d}^{d}e^{\ps_{\th}(x)}dx-\dint_{-c}^{c}e^{\ps_{\th}(x)}dx< 0.
\end{align*}
However, $\ps_{\th}$ satisfies $ \Psi(\th,\ps_{\th},F)\geq \Psi(\th,\ps,F)$ for all $\ps\in\mathcal{SC}_0$, which leads to a contradiction.
 Hence, we conclude that $c=d$, which completes the proof. 
\end{proof}

 \begin{lemma}\label{lemma: symmetry of projection operator}
 Suppose $F$ satisfies Condition A and \eqref{condition: symmetry}. Then  \[\argmax_{\ph\in\mathcal{C}}\om(\ph,F)=\argmax_{\ph\in\mathcal{SC}_{\th}}\om(\ph,F),\]
  where the criterion function $\om$ is defined in \eqref{definition: new criterion function }. 
 \end{lemma}
 
 \begin{proof}
 First we consider the special case $\th=0$. 
 For $\ph\in\mathcal{C}$ and $F$ symmetric about $0$, using \eqref{condition: symmetry} with $\theta=0$,  the criterion function $\om(\ph,F)$ can be written as
 \begin{align*}
\MoveEqLeft \dint_{-\infty}^{0}\ph(x)dF(x)+\dint_{0}^{\infty}\ph(x)dF(x)-\edint e^{\ph(x)}dx\\
 =&\ -\dint_{0}^{\infty}\ph(-x)dF(-x)+\dint_{0}^{\infty}\ph(x)dF(x)-\edint e^{\ph(x)}dx\\
 =&\  \dint_{0}^{\infty}\ph(-x)dF(x)+\dint_{0}^{\infty}\ph(x)dF(x)-\edint e^{\ph(x)}dx\\
=&\ \edint \dfrac{\ph(-x)+\ph(x)}{2}dF(x)-\edint e^{\ph(x)}dx,
 \end{align*}
 which is not larger than
 \[\edint\dfrac{\ph(x)+\ph(-x)}{2}dF(x)-\edint e^{(\ph(x)+\ph(-x))/2}dx,\]
 by the convexity of the exponential function. This proves that a $\phi\in\mathcal{SC}_0$ maximizes $\om(\phi,F)$ over $\mathcal{C}$ when $F$ is symmetric about $0$. The proof  for a general $\th$ follows in a similar way.
 \end{proof}
 
 \begin{lemma}\label{lemma: Lemma about the integrable bound of concave functions }
For any $\ps\in\mathcal{SC}_0$ with nonempty domain, and each $\e>0$,  there exists a function
$\ps^{(\e)}\in\mathcal{SC}_0$ for each $\e>0$ such that the following conditions are met:
\begin{enumerate}
\item[(A)] $\ps(x)\leq\ps^{(\e)}(x)\leq \ps(0)$ for all $x\in\RR$.
\item[(B)]
\[\ps^{(\e)}(x)\geq -|x|/\e+\ps(0),\]
 for all $x\in\RR$. 
\item[(B)]  $\ps^{(\e)}(x)\downarrow\ps(x)$ for each $x\in\RR$ as $\e\downarrow 0$.
\end{enumerate}
 \end{lemma}

 \begin{proof}
 
 Following the construction  in Lemma~$4.3$ of \cite{dumbreg} we set
 \begin{equation}\label{def:ps(epsilon)(0)}
 \ps^{(\e)}(x)=\inf_{(t,c)\in A} (tx+c),\quad x\in\RR,
 \end{equation}
 \[\]
 where
 \[A=\bigg\{(t,c)\in\RR\times\RR\ \bl\ |t|\leq 1/\e,\ tx+c\geq \ps(x)\text{ for all }x\in\RR\bigg\}.\]
 Note that $\ps^{(\e)}(x)$ is concave by Lemma~$4.3$ of \cite{dumbreg}. To show that $\ps^{(\e)}$ is symmetric about $0$, observe that if $tx+c\geq \ps(x)$ for all $x\in\RR$, the symmetry of $\ps$ about $0$ leads to
 \[-t(-x)+c\geq\ps(-x)\text{ for all }x\in\RR,\]
 which can be rewritten as
 \[-tx+c\geq \ps(x)\text{ for all }x\in\RR.\] 
  Hence, $(t,c)\in A$ implies that  $(-t,c)\in A$. As a result,
  \[\inf_{(t,c)\in A} (tx+c)=\inf_{(-t,c)\in A} (tx+c)=\inf_{(t,c)\in A} (-tx+c),\]
  which implies that $\ps^{(\e)}\in\mathcal{SC}_0$.
 That $\ps\leq \ps^{(\e)}$ follows from Lemma~$4.3$ of \cite{dumbreg}. Hence, to prove part A, we need to show that $\ps^{(\e)}$ is bounded above by $\ps(0)$. 
 
 Observe that since $\ps^{(\e)}\in\mathcal{SC}_0$,  it is bounded above by $\ps^{(\e)}(0)$, which, by \eqref{def:ps(epsilon)(0)}, equals the infrimum of the set
 \[A'=\{c\ :\ (t,c)\in A \text{ for some }t\in\RR\}.\]
 However, since $c\in A'$ satisfies $tx+c\geq \ps(x)$ for all $x\in\RR$, and some $t$, we conclude that 
  Thus we obtain
  \begin{equation}\label{inlemma: the integrable bound of concave functions: lower bound on A}
 \ps^{(\e)}(0)=\inf A'\geq\ps(0).
 \end{equation}
 
 In light of \eqref{inlemma: the integrable bound of concave functions: lower bound on A}, to prove part A of the current lemma, it suffices to show that $\ps(0)\in A'$.  Since $\ps\in\mathcal{SC}_0$, for $x>0$, we have  $\ps(x)\leq \ps'(0+)x+\ps(0)$.  When $x\leq 0$, using the fact that $\ps'(0+)\leq 0$, we calculate
 \[\ps(x)=\ps(-x)\leq -\ps'(0+)x+\ps(0)\leq \ps'(0+)x+\ps(0).\]
 Therefore, for all $x\in\RR$, the following holds:
 \[\ps(x)\leq \ps'(0+)x+\ps(0).\]
 When  $-\ps'(0+)\leq 1/\e$,  \eqref{def:ps(epsilon)(0)} thus indicates that $(-\ps'(0+),\ps(0))\in A$. In case $-\ps'(0+)>1/\e$, we still have 
\[\ps(x)\leq \ps'(0+)x+\ps(0)< -x/\e +\ps(0),\quad\text{ for all }x> 0.\]
For $x\leq 0$,
\[\ps(x)=\ps(-x)\leq x/\e +\ps(0)\leq -x/\e+\ps(0).\]
Therefore, it turns out that if $-\ps(0+)>1/\e$,
\[\ps(x)\leq -x/\e +\ps(0)\text{ for all }x\in\RR,\]
which implies that $(1/\e,\ps(0))\in A$. The above calculations indicate that
there always exists a $t\in\RR$ such that $(t,\ps(0))\in A$, entailing that $\ps(0)\in A'$. This fact, combined with \eqref{inlemma: the integrable bound of concave functions: lower bound on A}, yields that $\ps^{(\e)}(0)=\ps(0)$. Part A follows from \eqref{inlemma: the integrable bound of concave functions: lower bound on A}.

 Note that \eqref{def:ps(epsilon)(0)} and \eqref{inlemma: the integrable bound of concave functions: lower bound on A} imply that
\[\ps^{(\e)}(x)\geq -|x|/\e+\ps(0),\quad x\in\RR,\]
which settles the proof of part B.

  Lemma~\ref{lemma: Lemma about the integrable bound of concave functions }(C) follows directly from Lemma~$4.3$ of \cite{dumbreg}.
 \end{proof}
 
The next lemma is Pratt's lemma \citep[][Theorem 1]{Pratt1960}.
 We state it here for convenience.
 \begin{lemma}\label{lemma: bill1}
 Suppose $(\Omega,\mathcal{F},\mu)$ is a measure space and $a_n,b_n,c_n$ are sequences of functions on $\Omega$ converging almost everywhere to functions $a,b,c$ respectively. Also, all functions are integrable and
 $\int a_n d\mu\to \int ad\mu$ and $\int c_n d\mu\to \int cd\mu.$ Moreover, $a_n\leq b_n\leq c_n.$ Then 
 \[\dint b_n d\mu\to\dint b d\mu.\]
\end{lemma}

\bibliographystyle{natbib}
\bibliography{location_estimation}

\begin{thebibliography}{}

\bibitem[Beran(1974)Beran]{beran}
Beran, R. (1974).
\newblock Asymptotically efficient adaptive rank estimates in location models.
\newblock {\em Ann. Statist.}, {\bf 2}, 63--74.

\bibitem[Bhattacharyya and Bickel(2013)Bhattacharyya and Bickel]{sharmada}
Bhattacharyya, S. and Bickel, P.~J. (2013).
\newblock Adaptive estimation in elliptical distributions with extensions to
  high dimensions by.

\bibitem[Bickel {\em et~al.}(1998)Bickel, Klaassen, Ritov, and
  Wellner]{jonsemi}
Bickel, P.~J., Klaassen, C. A.~J., Ritov, Y., and Wellner, J.~A. (1998).
\newblock {\em Efficient and Adaptive Estimation for Semiparametric Models\/}.
\newblock Springer-Verlag, New York.

\bibitem[Birg{\'e}(1997)Birg{\'e}]{birge1997}
Birg{\'e}, L. (1997).
\newblock Estimation of unimodal densities without smoothness assumptions.
\newblock {\em Ann. Statist.}, {\bf 25}, 970--981.

\bibitem[Chen and Samworth(2013)Chen and Samworth]{smoothed}
Chen, Y. and Samworth, R.~J. (2013).
\newblock Smoothed log-concave maximum likelihood estimation with applications.
\newblock {\em Statist. Sinica\/}, {\bf 23}, 1373--1398.

\bibitem[Chen and Samworth(2016)Chen and Samworth]{chen2016}
Chen, Y. and Samworth, R.~J. (2016).
\newblock Generalized additive and index models with shape constraints.
\newblock {\em J R Stat Soc Series B Stat Methodol\/}, {\bf 78}, 729--754.

\bibitem[Cule and Samworth(2010)Cule and Samworth]{theory}
Cule, M. and Samworth, R. (2010).
\newblock Theoretical properties of the log-concave maximum likelihood
  estimator of a multidimensional density.
\newblock {\em Electron. J. Statist.}, {\bf 4}, 254--270.

\bibitem[Devroye(1987)Devroye]{devroye1987}
Devroye, L. (1987).
\newblock {\em A course in density estimation\/}.
\newblock Progress in probability and statistics. Birkh{\"a}user.

\bibitem[Doss and Wellner(2016)Doss and Wellner]{dossglobal}
Doss, C.~R. and Wellner, J.~A. (2016).
\newblock Global rates of convergence of the {MLEs} of log-concave and $ s
  $-concave densities.
\newblock {\em Ann. Statist.}, {\bf 44}, 954--981.

\bibitem[Doss and Wellner(2019)Doss and Wellner]{dosssymmetric}
Doss, C.~R. and Wellner, J.~A. (2019).
\newblock Univariate log-concave density estimation with symmetry or modal
  constraints.
\newblock {\em Electron. J. Stat.}, {\bf 25}, 2391--2461.

\bibitem[D{\"u}mbgen and Rufibach(2009)D{\"u}mbgen and Rufibach]{2009rufi}
D{\"u}mbgen, L. and Rufibach, K. (2009).
\newblock Maximum likelihood estimation of a log-concave density and its
  distribution function: Basic properties and uniform consistency.
\newblock {\em Bernoulli\/}, {\bf 15}, 40--68.

\bibitem[D{\"u}mbgen {\em et~al.}(2011)D{\"u}mbgen, Samworth, and
  Schuhmacher]{dumbreg}
D{\"u}mbgen, L., Samworth, R., and Schuhmacher, D. (2011).
\newblock Approximation by log-concave distributions, with applications to
  regression.
\newblock {\em Ann. Statist.}, {\bf 39}, 702--730.

\bibitem[Groeneboom and Hendrickx(2018)Groeneboom and Hendrickx]{piet}
Groeneboom, P. and Hendrickx, K. (2018).
\newblock Current status linear regression.
\newblock {\em Ann. Statist.}, {\bf 46}, 1415--1444.

\bibitem[Hodges and Lehmann(1963)Hodges and Lehmann]{hodges}
Hodges, J.~L. and Lehmann, E.~L. (1963).
\newblock Estimates of location based on rank tests.
\newblock {\em Ann. Math. Statist.}, {\bf 34}, 598--611.

\bibitem[Huber(1964)Huber]{huber}
Huber, P.~J. (1964).
\newblock Robust estimation of a location parameter.
\newblock {\em Ann. Math. Statist.}, {\bf 35}, 73--101.

\bibitem[Kim and Samworth(2016)Kim and Samworth]{global2015}
Kim, A. K.~H. and Samworth, R.~J. (2016).
\newblock Global rates of convergence in log-concave density estimation.
\newblock {\em Ann. Statist.}, {\bf 44}, 2756--2779.

\bibitem[Kuchibhotla {\em et~al.}(2017)Kuchibhotla, Patra, and
  Sen]{bodhida2017}
Kuchibhotla, A.~K., Patra, R.~K., and Sen, B. (2017).
\newblock Efficient estimation in convex single index models.
\newblock {\em arXiv:1708.00145v2\/}.

\bibitem[Murphy {\em et~al.}(1999)Murphy, Van~der Vaart, and
  Wellner]{murphy1999}
Murphy, S.~A., Van~der Vaart, A., and Wellner, J. (1999).
\newblock Current status regression.
\newblock {\em Math. Methods Statist.}, {\bf 8}, 407--425.

\bibitem[Pal {\em et~al.}(2007)Pal, Woodroofe, and Meyer]{exist}
Pal, J.~K., Woodroofe, M., and Meyer, M. (2007).
\newblock Estimating a p{\'o}lya frequency function$_2$.
\newblock {\em Lecture Notes-Monograph Series\/}, {\bf 54}, 239--249.

\bibitem[Pratt(1960)Pratt]{Pratt1960}
Pratt, J.~W. (1960).
\newblock On interchanging limits and integrals.
\newblock {\em Ann. Math. Statist.}, {\bf 31}, 74--77.

\bibitem[Rockafellar(1970)Rockafellar]{rockafellar}
Rockafellar, R.~T. (1970).
\newblock {\em Convex {Analysis}\/}.
\newblock Princeton {University} {Press}.

\bibitem[Sacks(1975)Sacks]{saks}
Sacks, J. (1975).
\newblock An asymptotically efficient sequence of estimators of a location
  parameter.
\newblock {\em Ann. Statist.}, {\bf 3}, 285--298.

\bibitem[Seijo and Sen(2011)Seijo and Sen]{seijo}
Seijo, E. and Sen, B. (2011).
\newblock Nonparametric least squares estimation of a multivariate convex
  regression function.
\newblock {\em Ann. Statist.}, {\bf 39}, 1633--1657.

\bibitem[Shorack(2000)Shorack]{shorack2000}
Shorack, G.~R. (2000).
\newblock {\em Probability for {S}tatisticians\/}.
\newblock Springer.

\bibitem[Silverman(1982)Silverman]{silverman1982}
Silverman, B.~W. (1982).
\newblock On the estimation of a probability density function by the maximum
  penalized likelihood method.
\newblock {\em Ann. Statist.}, pages 795--810.

\bibitem[Stein(1956)Stein]{stein}
Stein, C. (1956).
\newblock Efficient nonparametric testing and estimation.
\newblock {\em Proc. Third Berkley Symp. Math. Statist. Prob\/}, {\bf 1},
  187--196.

\bibitem[Stone(1975)Stone]{stone}
Stone, C.~J. (1975).
\newblock Adaptive maximum likelihood estimators of a location parameter.
\newblock {\em Ann. Statist.}, {\bf 3}, 267--284.

\bibitem[Takeuchi(1975)Takeuchi]{unimodal1}
Takeuchi, K. (1975).
\newblock A survey of robust estimation of location: models and procedures,
  especially in case of measurement of a physical quantity.
\newblock {\em Bull. Inst. Internat. Statist.}, {\bf 46}, 336--348.
\newblock With discussion.

\bibitem[{V}an~der Vaart(1998){V}an~der Vaart]{vdv}
{V}an~der Vaart, A. (1998).
\newblock {\em Asymptotic Statistics\/}.
\newblock Asymptotic Statistics. Cambridge University Press.

\bibitem[Van~der Vaart and Wellner(1996)Van~der Vaart and Wellner]{wc}
Van~der Vaart, A.~W. and Wellner, J.~A. (1996).
\newblock {\em Weak Convergence and Empirical Processes\/}.
\newblock Springer, New York.

\bibitem[van~der Vaart and Wellner(2007)van~der Vaart and Wellner]{epindex}
van~der Vaart, A.~W. and Wellner, J. A.~W. (2007).
\newblock Empirical processes indexed by estimated functions.
\newblock {\em Asymptotics: Particles, Processes and Inverse Problems\/}, {\bf
  55}, 234--252.

\bibitem[Van~Eeden(1970)Van~Eeden]{eden}
Van~Eeden, C. (1970).
\newblock Efficiency-robust estimation of location.
\newblock {\em Ann. Math. Statist.}, {\bf 41}, 172--181.

\bibitem[Villani(2003)Villani]{villani2003}
Villani, C. (2003).
\newblock {\em Topics in optimal transportation\/}, volume~58 of {\em Graduate
  Studies in Mathematics\/}.
\newblock American Mathematical Society, Providence, RI.

\bibitem[Villani(2009)Villani]{villani2009}
Villani, C. (2009).
\newblock {\em Optimal transport; Old and New\/}, volume 338 of {\em
  Grundlehren der Mathematischen Wissenschaften\/}.
\newblock Springer-Verlag, Berlin.
\newblock Old and new.

\bibitem[Xu and Samworth(2019)Xu and Samworth]{xuhigh}
Xu, M. and Samworth, R.~J. (2019).
\newblock High-dimensional nonparametric density estimation via symmetry and
  shape constraints.
\newblock {\em https://arxiv.org/abs/1903.06092v1\/}.

\end{thebibliography}

\end{document}